\documentclass[12pt,a4paper,twoside,reqno]{amsart}

\usepackage{amsmath, amssymb, amsthm, enumerate, enumitem}
\usepackage[normalem]{ulem}

\makeatletter
\def\l@section{\@tocline{1}{10pt}{1pc}{}{}}
\def\l@subsection{\@tocline{2}{0pt}{1pc}{4.6em}{}}
\def\l@subsubsection{\@tocline{3}{0pt}{1pc}{7.6em}{}}
\renewcommand{\tocsection}[3]{%
  \indentlabel{\@ifnotempty{#2}{\makebox[2.3em][l]{%
    \ignorespaces#1 #2.\hfill}}}\textbf{#3}}
\renewcommand{\tocsubsection}[3]{%
  \indentlabel{\@ifnotempty{#2}{\hspace*{2.3em}\makebox[2.3em][l]{%
    \ignorespaces#1 #2.\hfill}}}#3}
\renewcommand{\tocsubsubsection}[3]{%
  \indentlabel{\@ifnotempty{#2}{\hspace*{4.6em}\makebox[3em][l]{%
    \ignorespaces#1 #2.\hfill}}}#3}
\makeatother

\newcommand{\MM}{\mathcal{M}}
\newcommand{\CC}{\mathcal{C}}
\newcommand{\IR}{\mathbb{R}}

\newcommand{\IH}{\mathbb{H}}

\newcommand{\IZ}{\mathbb{Z}}

\newcommand{\TT}{\mathcal{T}}

\newcommand{\ov}[1]{\overline{#1}}

\newcommand{\td}[1]{\widetilde{#1}}

\DeclareMathOperator{\stan}{stan}
\DeclareMathOperator{\scal}{scal}
\DeclareMathOperator{\Int}{Int}

\DeclareMathOperator{\Ric}{Ric}

\DeclareMathOperator{\area}{area}

\DeclareMathOperator{\tr}{tr}

\DeclareMathOperator{\dist}{dist}

\DeclareMathOperator{\diam}{diam}

\DeclareMathOperator{\vol}{vol}

\DeclareMathOperator{\Rm}{Rm}

\newcommand{\nangle}{\sphericalangle}

\newcommand{\dotcup}{\ensuremath{\mathaccent\cdot\cup}}
\newcommand{\EMPTY}[1]{}

\newtheorem{Theorem}{Theorem}[section]
\newtheorem{Lemma}[Theorem]{Lemma}
\newtheorem{Corollary}[Theorem]{Corollary}
\newtheorem{Proposition}[Theorem]{Proposition}
\newtheorem{Definition}[Theorem]{Definition}

\newtheorem{Remark}[Theorem]{Remark}
\numberwithin{equation}{section}

\title{Long-time analysis of 3 dimensional Ricci flow III}
\author{Richard H Bamler}
\address{Stanford University, Department of Mathematics, 450 Serra Mall, building 380, Stanford, California 94305}
\email{rbamler@stanford.edu}
\date{\today}

\begin{document}
\begin{abstract}
In this paper we analyze the long-time behavior of $3$ dimensional Ricci flows with surgery.
Our main result is that if the surgeries are performed correctly, then only finitely many surgeries occur and after some time the curvature is bounded by $C t^{-1}$.
This result confirms a conjecture of Perelman.
In the course of the proof, we also obtain a qualitative description of the geometry as $t \to \infty$.
 
This paper is the third part of a series.
Previously, we had to impose a certain topological condition $\mathcal{T}_2$ to establish the finiteness of the surgeries and the curvature control.
The objective of this paper is to remove this condition and to generalize the result to arbitrary closed 3-manifolds.
This goal is achieved by a new area evolution estimate for minimal simplicial complexes, which is of independent interest.
\end{abstract}

\maketitle
\tableofcontents

\section{Introduction} \label{sec:Introduction}
\subsection{Statement of the main result}
In this paper we analyze the long-time behavior of Ricci flows with surgery on $3$ dimensional manifolds without imposing any conditions on the topology of the initial manifold.

In a few words, our main result can be summarized as follows.
We refer to Theorem \ref{Thm:MainTheorem-III} on page \pageref{Thm:MainTheorem-III} for a precise statement.

\begin{quote}
\textit{Let $(M,g)$ be a closed and orientable $3$ dimensional Riemannian manifold. \\
Then there is a long-time existent Ricci flow with only  \emph{finitely} many surgeries whose initial time-slice is $(M, g)$.
Moreover, there is a constant $C$ such that the Riemannian curvature in this flow is bounded everywhere by $C t^{-1}$ for large times $t$.}
\end{quote}

The Ricci flow with surgery has been used by Perelman to solve the Poincar\'e and Geometrization Conjecture (\cite{PerelmanI}, \cite{PerelmanII}, \cite{PerelmanIII}).
Given any initial metric on a closed $3$-manifold, Perelman managed to construct a solution to the Ricci flow with surgery on a maximal time-interval and showed that its surgery times do not accumulate.
Hence every finite time-interval contains only a finite number of surgery times.
Furthermore, he could prove that if the given manifold is a homotopy sphere (or more generally a connected sum of prime, non-aspherical manifolds), then this flow goes extinct in finite time and the number of surgeries is finite.
This implies that the initial manifold is a sphere if it is simply connected and hence establishes the Poincar\'e Conjecture.
On the other hand, if the Ricci flow with surgery continues to exist for infinite time, Perelman could show that the manifold decomposes into a thick part, which approaches a hyperbolic metric, and an thin part, which becomes arbitrarily collapsed on local scales.
Based on this collapse, it is then possible to show that the thin part can be decomposed into pieces whose topology is well understood (\cite{ShioyaYamaguchi}, \cite{MorganTian}, \cite{KLcollapse}).
Eventually, this decomposition can be reorganized to a geometric decomposition, establishing the Geometrization Conjecture.

Observe that although the Ricci flow with surgery was used to solve such hard problems, some of its basic properties are still unknown, because they surprisingly turned out to be irrelevant in the end.
For example, it was only conjectured by Perelman that in the long-time existent case there are finitely many surgeries, i.e. that after some time the flow can be continued by a conventional smooth, non-singular Ricci flow defined up to time infinity.
Furthermore, it is still unknown whether and in what way the Ricci flow exhibits the the full geometric decomposition of the manifold.

In \cite{LottTypeIII}, \cite{LottDimRed} and \cite{LottSesum}, Lott and Lott-Sesum could give a description of the long-time behavior of certain non-singular Ricci flows on manifolds whose geometric decomposition consists of a single component.
However, they needed to make additional curvature and diameter or symmetry assumptions.
In \cite{Bamler-longtime-I}, the author proved that under a purely topological condition $\TT_1$, which roughly states that the manifold only consists of  hyperbolic components (see \cite[Definition 1.2]{Bamler-longtime-II}), there are only finitely many surgeries and the curvature is bounded by $C t^{-1}$ after some time.
In \cite{Bamler-longtime-II}, this condition was generalized to a far more general topological condition $\TT_2$, which requires that the non-hyperbolic pieces in the geometric decomposition of the underlying manifold contain sufficiently many incompressible surfaces.
For example, manifolds of the form $\Sigma \times S^1$ for closed, orientable surfaces $\Sigma$, in particular the $3$-torus $T^3$, satisfy property $\TT_2$, but the Heisenberg manifold does not.
We refer to \cite[sec 2]{Bamler-longtime-II} for a precise definition and discussion of the conditions $\TT_1$ and $\TT_2$.

In this paper we remove the condition $\TT_2$ and only assume that the initial manifold is closed and orientable.

We now state our main result.
The notions relating to ``Ricci flows with surgery'', which are used in the following, are explained in subsection 2.1 of \cite{Bamler-longtime-II}.
In a few words, a ``Ricci flow with surgery $\MM$, which is performed by $\delta(t)$-precise cutoff'', is a sequence of Ricci flows $(M^1, (g^1_t)_{t \in [0,T^1)}), (M^2, (g^2_t)_{t \in [T^1, T^2)})$ such that the time-$T^i$ slice $(M^{i+1}, g^{i+1}_{T^i})$ is obtained from the singular metric $g^i_{T^i}$ on $M^i$ by a so called surgery process which amounts to a geometric version of an inverse connected sum decomposition at a scale less than $\delta(T^i)$ and the removal of spherical components.
We allow the case in which there are only finitely many surgery times $T^i$ and $T^i = \infty$ for the final index $i$.
Observe that we have chosen our notion such that a  $\delta(t)$-precise cutoff is $\delta'(t)$-precise if $\delta'(t) \geq \delta(t)$.

In \cite{PerelmanII} Perelman showed the existence of a (non-explicit) function $\delta(t)$ such that every normalized Riemannian manifold $(M,g)$ can be evolved into a Ricci flow with surgery $\MM$ which is performed by $\delta(t)$-precise cutoff and that for any such Ricci flow with surgery---performed by $\delta(t)$-precise cutoff and with normalized initial conditions---the surgery times $T^i$ do not accumulate.
So if there were infinitely many surgery times (or equivalently infinitely many surgeries), then we must have $\lim_{i \to \infty} T^i = \infty$.
Our main result now states that this cannot happen under normalized initial conditions and if $\delta(t)$ has been chosen sufficiently small.
Note that these assumptions are not very restrictive since they already had been made in Perelman's work.

\begin{Theorem}
Given a surgery model $(M_{\stan}, g_{\stan}, D_{\stan})$, there is a continuous function $\delta : [0, \infty) \to (0, \infty)$ such that the following holds: \nopagebreak  \label{Thm:MainTheorem-III}

Let $\MM$ be a Ricci flow with surgery with normalized initial conditions which is performed by $\delta(t)$-precise cutoff.

Then $\MM$ has only finitely many surgeries and there are constants $T, C < \infty$ such that $|{\Rm_t}| < C t^{-1}$ on $\MM(t)$ for all $t \geq T$.
\end{Theorem}

In the course of the proof of this theorem we will obtain a more detailed description of the geometry of the time-slices $\MM(t)$ for large times $t$.
An even more precise characterization of the long-time behavior would however still be desirable.

We mention three interesting direct consequences of Theorem \ref{Thm:MainTheorem-III} which can be expressed in a completely elementary way and which illustrate the power of this theorem.
None of these results have been known so far to the author's knowledge.
The first consequence is just a restatement of the main theorem in the case in which $\MM$ is non-singular.
Note that even in this particular case our proof does not simplify significantly apart from various technicalities.
In fact, the reader is advised to only consider non-singular Ricci flows upon first reading of this and the previous papers.

\begin{Corollary}
Let $(M, (g_t)_{t \in [0, \infty)})$ be a non-singular, long-time existent Ricci flow on a closed $3$-manifold $M$.
Then there is a constant $C < \infty$ such that
\[ |{\Rm_t}| < \frac{C}{t+1} \qquad \text{for all} \qquad t \geq 0. \]
\end{Corollary}

The next result provides a characterization of when the condition of the previous corollary can indeed be satisfied.
\begin{Corollary}
Let $M$ be a closed $3$-manifold.
Then there exists a long-time existent Ricci flow $(g_t)_{t \in [0, \infty)}$ on $M$ if and only if $\pi_2(M) = \pi_3(M) = 0$.
\end{Corollary}

Note that this topological condition is equivalent to $M$ being aspherical which is equivalent to $M$ being irreducible and not diffeomorphic to a spherical space form.

This corollary can be deduced as follows.
Any normalized (see \cite[Definition 2.12]{Bamler-longtime-II}) Riemannian metric $g$ on an aspherical manifold $M$, can be evolved to a long-time existent Ricci flow with surgery $\MM$ on the time-interval $[0, \infty)$, which is performed by $\delta(t)$-precise cutoff, due to Perelman (\cite{PerelmanII}, see also \cite[Proposition 2.16]{Bamler-longtime-II}).
The topological condition ensures that all surgeries on $\MM$ are trivial and hence that every time-slice of $\MM$ has a component which is diffeomorphic to $M$.
By Theorem \ref{Thm:MainTheorem-III}, there is a final surgery time $T < \infty$ on $\MM$.
So the flow $\MM$ restricted to the time-interval $[T, \infty)$ and the component, which is diffeomorphic to $M$, is non-singular.
Shifting this flow in time by $-T$ yields the desired Ricci flow.
The reverse direction is well known, for example it is a direct consequence of \cite[Proposition 8.5]{Bamler-longtime-II} and finite-time extinction (see \cite{PerelmanIII}, \cite{ColdingMinicozziextinction}, \cite{MTRicciflow}).

The third application was inspired by Andrew Sanders.
Denote by $\mathcal{D}$ the space of all smooth Riemannian metrics on a closed and orientable $3$-manifold $M$ and equip $\mathcal{D}$ with the $C^0$ (i.e. Bilipschitz) topology.
Let $\mathcal{D}' \subset \mathcal{D}$ be the subspace of all metrics starting from which the Ricci flow does not develop any singularities.
Then

\begin{Corollary}
$\mathcal{D}'$ is an open subset of $\mathcal{D}$.
\end{Corollary}

Note that by standard parabolic theory, for every metric $g \in \mathcal{D}'$ and every time $t_0 > 1$ there is an $\varepsilon > 0$ such that whenever $g' \in \mathcal{D}$ is $(1+\varepsilon)$-Bilipschitz to $g$ then the Ricci flow starting from $g'$ stays $C^k$-close (modulo a diffeomorphism) on the time-interval $[1,t_0]$ to the Ricci flow $(g_t)_{t \in [0, \infty)}$ with $g_0 = g$, for any $k$.
So the Ricci flow with surgery $\MM'$ starting from $(M, g')$ does not develop any singularities before time $t_0$.
It can be seen that the time $T$ from Theorem \ref{Thm:MainTheorem-III}, after which $\MM'$ does not develop any singularities, can be chosen independently of $g'$ if $\varepsilon$ is small enough (due to the $C^k$-closeness at time $1$ all geometric quantities which influence $T$ stay bounded).
So if we pick $t_0 > T$, then $\MM'$ must be non-singular everywhere.

\subsection{Outline of the proof}
The proof of the main theorem relies strongly on the solution of the Geometrization Conjecture and the results of the previous paper \cite{Bamler-longtime-II} where we established Theorem \ref{Thm:MainTheorem-III} under the additional topological condition $\mathcal{T}_2$.
In this paper we will mostly only refer to these results without repeating the complete statements or proofs since they can often be taken over largely unchanged.

As in the previous two papers (\cite{Bamler-longtime-I}, \cite{Bamler-longtime-II}), the key to proving Theorem \ref{Thm:MainTheorem-III} is to establish the curvature bound $|{\Rm_t}| < C t^{-t}$ for large times $t$.
This bound then implies immediately that surgeries stop to occur eventually, since they can only arise where the curvature goes to infinity.
For simplicity, we will only consider Ricci flows without surgery in this subsection, i.e. families of metrics $(g_t)_{t \in [0, \infty)}$ on a closed, orientable $3$-manifold $M$ which satisfy the evolution equation
\[ \partial_t g_t = - 2 \Ric_{g_t}. \]
In the next paragraphs, we will point out how we obtain the curvature bound $|{\Rm_t}| < C t^{-1}$ on $(M, g_t)$ for large $t$.
The case in which the given Ricci flow has surgeries follows similarly, apart from various technicalities.
Consider the rescaled and reparameterized flow $(\td{g}_t)_{t \in (-\infty, \infty)}$ with $\td{g}_t = e^{-t} g_{e^{t}}$ which satisfies the flow equation
\[ \partial_t \td{g}_t = - 2 \Ric_{\td{g}_t} -  \td{g}_t. \]
Then the curvature bound $|{\Rm_t}| < C t^{-1}$ for $g_t$ is equivalent to the bound $|{\Rm_t}| <C$ for $\td{g}_t$.
For clarity we will only consider the metrics $\td{g}_t$ instead of $g_t$ for the rest of this subsection, we will however work with the metric $g_t$ in the main part of this paper.

We first give an overview over the proof in \cite{Bamler-longtime-II} and point out why we had to impose condition $\mathcal{T}_2$ there.
A crucial result in the previous paper was Proposition 8.2 of \cite{Bamler-longtime-II} (also called ``second step'') which asserted that at large times $t$ the curvature (of $\td{g}_t$) is bounded by a uniform constant everywhere except possibly on finitely many pairwise disjoint, embedded, incompressible solid tori $S_1, \ldots, S_m \subset M$.
By ``incompressible'' we mean that each $S^1$-fiber of $S_i \approx S^1 \times D^2$ and every non-trivial multiple of this fiber is non-contractible in $M$.
Moreover, these solid tori $S_i$ had the following geometric property:
For each $S_i$ we can find a ``thin and long'' collar neighborhood of the boundary $2$-torus $\partial S_i$ inside $S_i$.
This collar neighborhood is diffeomorphic to $T^2 \times I$, in such a way that the diameter of each $T^2$-fiber is bounded by $1$ and the distance between its two boundary tori is large depending on the diameter of $S_i$.
Proposition 8.2 also implied that the curvature on each $S_i$ is bounded in terms of its diameter for large $t$.
It hence remained to exclude the possibility that the diameters of the solid tori $S_i$ become unbounded as $t \to \infty$.
Note that Proposition 8.2 is true for all topologies and does not require condition $\mathcal{T}_2$ to hold.

The topological condition $\mathcal{T}_2$ then came into play when we controlled the diameter of the solid tori $S_i$.
Here we used an idea of Hamilton (cf Lemma 7.2 in \cite{Bamler-longtime-II} or \cite{Ham}):
For any (not necessarily connected) surface $\Sigma$, every incompressible map $f_0 : \Sigma \to M$ and every time $t > 0$ let $A_t (f_0)$ be the area of the area-minimizer $f_t : \Sigma \to M$ within the homotopy class of $f_0$.
Using the minimal surface equation and Gau\ss-Bonnet, it is possible to derive an evolution inequality for $A_t(f_0)$ and hence a bound of the form $A_t (f_0) < A(\Sigma) + O(e^{-t/4})$ where $A(\Sigma)$ only depends on the topology of $\Sigma$.
Now condition $\mathcal{T}_2$ ensured that we could choose a so called filling map $f_0 : \Sigma \to M$.
By ``filling'' we roughly mean that the image of every map $f : \Sigma \to M$ which is homotopic to $f_0$ (e.g. $f_t$) has to intersect every incompressible loop $\sigma \subset M$.
For example, if the given manifold was the $3$-torus $T^3$, then $f_0$ would be the embedding of the $3$ standard, pairwise orthogonal $2$-tori.
This property then implied that the image of each map $f_t$ had to intersect each $S^1$-fiber in each solid torus $S_i$ from Proposition 8.2.
It was then not hard extract a compressing disk for $S_i$ from this map, i.e. a map $D^2 \to M$ which maps $\partial D^2$ to a non-contractible loop in the boundary $2$-torus $\partial S_i$.
Obviously, the area of this compressing disk had to be $< A(\Sigma) + O(e^{-t/4})$.

The existence of this compressing disk of bounded area had two implications:
On the one hand, we were able to show in Proposition 8.3 of \cite{Bamler-longtime-II} (aka ``third step'') that, modulo some insignificant modifications, the collar neighborhoods inside the solid tori $S_i$ from Proposition 8.2 were thin, meaning that the $T^2$-fibers were bounded in diameter by a constant which goes to $0$ as the diameter of $S_i$ goes to infinity.
This improved control was an important ingredient for the proof of Proposition 8.4 of \cite{Bamler-longtime-II} which asserted that solid tori $S_i$ of large diameter persist when going far backwards in time.
On the other hand, it helped to rule out exactly the possibility that a solid torus $S_i$ could persist for a long time while having large diameter.
The main idea here was that, again by the same area estimate of Hamilton, the area of a minimal compressing disk for $S_i$ has to go to zero after a certain time which can be computed in terms of $A(\Sigma)$.
This contradiction was used to bound the diameter of the solid tori $S_i$ and concluded the proof of the main theorem of \cite{Bamler-longtime-II}.

The difficulty in proving Theorem \ref{Thm:MainTheorem-III} in the present paper hence comes from the fact that there will in general not be such a filling map for $M$.
For example, if $M$ is the Heisenberg manifold, i.e. a twisted $S^1$-bundle over $T^2$, then the only incompressible surfaces inside $M$ are the $2$-tori which are vertical with respect to the fibration and hence we cannot guarantee that one of their homotopes intersects a given $S^1$-fiber.
The strategy that we employ to get around this issue is to allow $\Sigma$ to have junctions and vertices, i.e. we allow that $\Sigma$ is homeomorphic to a $2$ dimensional simplicial complex which we will in the following denote by $V$.
For example if $V$ is the $2$-skeleton of a triangulation of $M$ and $f_0 : V \to M$ is the inclusion map, then $f_0$ is filling in a generalized sense:
Every incompressible loop $\sigma \subset M$ has to intersect the image of $f_0$ since otherwise it would be contained in the interior of a $3$-simplex.
It can be seen that the same is true for the image of every map $f : V \to M$ which is homotopic to $f_0$.

It thus becomes important to estimate the evolution of the quantity $A_t (f_0)$, which is the infimum the area functional at time $t$ evaluated on all maps $f : V \to M$ homotopic to $f$.
An inspection of the arguments leading to the bound in the case, in which $V = \Sigma$ was a surface, shows that if the existence of an area minimizing map $f_t : V \to M$ is guaranteed, then all previous estimates can be carried out.
Here we have to make use of the Euler-Lagrange equations for $f_t$ along the edges of $V$, which state that around every edge the faces meet in directions which add up to zero.
This implies that certain boundary integrals arising in the application of Gau\ss-Bonnet cancel each other out.

Unfortunately, an existence and regularity theory for such minimizers $f_t$ does not exist to the author's knowledge and seems to be difficult to achieve.
We note that however if we allow the combinatorial structure of $V$ to vary, then a result of Choe (cf \cite{Choe})---which relies heavily on this fact---states that for every Riemannian metric $g$ on $M$, there is a smooth, minimal embedding $f_g : V_g \to M$ such that the complement of $f_g (V_g)$ is a topological ball.
Such a map would be filling, but it seems to be difficult to control the number of vertices of $V_g$ and this number influences the bound $A (V_{g_t})$ in the area evolution estimate of $A_t (f_0)$.
In fact, it is very likely that there are metrics $g_1, g_2, \ldots$ on $M$ for which the number of vertices of the corresponding minimal simplicial complex $V_{g_k}$ diverges.

In order to get around this issue, we will employ the following trick in this paper.
Instead of looking for a minimizer of the area functional, we will try to find a minimizer of the perturbed functional
\[ f \longmapsto \area f + \lambda \ell (f|_{V^{(1)}}). \]
Here $\lambda > 0$ is a small constant, $\ell (f |_{V^{(1)}})$ denotes the sum of the lengths of $f$ restricted to all edges of $V$ and $f : V \to M$ is any map which is homotopic to $f_0$.
The existence and regularity of a minimizer for the perturbed functional follows now easily (apart from some issues arising from possible self-intersections of the $1$-skeleton).
However, the extra term $\lambda \ell (f |_{V^{(1)}})$ introduces an extra term in the Euler-Lagrange equations along each edge of $V$ and hence the boundary integrals in the evolution estimate for the minimum of this perturbed functional will not cancel each other out, but add up to a new term.
Luckily, it will turn out that this term has the right sign to carry out this evolution estimate.
Now letting $\lambda$ go to $0$, we obtain the desired evolution estimate for $A_t (f_0)$ and conclude that there are homotopes $f_t : V \to M$ of $f_0$ whose time-$t$ area is bounded by $A(V) + O(e^{-t/4})$ where $A(V)$ only depends on $V$.

Next, we need to construct compressing disks of bounded area for each solid torus $S_i$ from Proposition 8.2 of \cite{Bamler-longtime-II} from the maps $f_t$.
Unlike in the case in which $V = \Sigma$ was smooth, this task is surprisingly difficult since the preimages $f_t^{-1} (S_i)$ could be arbitrarily complex.
In fact, it is not known to the author whether such a construction works in the special case in which $M$ is covered by a $2$-torus bundle over a circle.
If $M$ is not covered by such a bundle, then we can resolve this issue by choosing $V$ very carefully and by analyzing certain combinatorial distance functions on the universal cover of $V$.
However, we will only be able to construct possibly self-overlapping ``compressing multiply connected domains'' of bounded area.
In a nutshell, these domains will be constructed as follows:
Intersecting a large number of polyhedral spheres in the universal cover $\td{M}$ of $M$ with a lift $\td{S}_i \subset \td{M}$ of $S_i$ yields a large number of ``compressing multiply connected domains'' for $\td{S}_i$ and hence for $S_i$.
We will be able to show that these domains are contained in a bounded number of fundamental domains of $\td{S}_i$ and conclude that their area is bounded on average.
So one of these domains must have controlled area.
Later, this ``compressing multiply connected domains'' of bounded  will be converted into a compressing disk of bounded area.
The existence of such disks will then enable us to use arguments of \cite{Bamler-longtime-II} to establish Theorem \ref{Thm:MainTheorem-III} whenever $M$ is not covered by a $2$-torus bundle over a circle.

Finally, consider the case in which $M$ is covered by a $2$-torus bundle over a circle, i.e. when $M$ is a quotient of the $3$-torus, the Heisenberg manifold or the Solvmanifold.
Assume for simplicity that $M$ itself is a $2$-torus bundle over a circle.
We can then use the topological fact that $M$ admits arbitrarily large finite covers $\pi_n : \widehat{M}_n \to M$ for which $\widehat{M}_n$ is diffeomorphic to $M$.
So we can view each $\pi_n$ as a map from $M$ to $M$ and for any map $f_0 : V \to M$ we can construct the maps $f_{n,0} = \pi_n \circ f_0 : V \to M$.
For each $n$ we can apply the area estimate from before and find homotopes $f_{n,t} : V \to M$ whose time-$t$ area is bounded by $A(V) + O_n (e^{-t/4})$.
Here $A(V)$ is independent of $n$, and $O_n (e^{-t/4})$ is a quantity which depends on $n$ and which goes to $0$ as $t \to \infty$.
Choosing said covers carefully, we can ensure that for all $n \geq 1$ every incompressible loop $\sigma \subset M$ has to intersect the image of $f_n$ at least $n$ times.
For large $n$ and large $t$, this will then imply that the manifold $(M, \td{g}_t)$ can nowhere locally collapse to a $2$-dimensional space.
A consequence of this is that the set of solid tori $S_i$ is empty and hence the curvature is bounded everywhere on $M$.

This paper is organized as follows:
In section \ref{sec:Defpolygonalcomplex} we give a precise definition of the simplicial complexes that we will be working with.
Section \ref{sec:existenceofvminimizers} contains the existence and regularity theory for simplicial complexes which minimize the perturbed area functional.
These results will then be used in section \ref{sec:areaevolutionunderRF} to derive the infimal area evolution estimate for simplicial complexes, i.e. the bound on $A_t(f_0)$.
In section \ref{sec:constructanalysispolygonalcomplex} we construct the simplicial complex $V$ and the map $f_0 : M \to V$, prove the existence of ``compressing multiply connected domains'' if $M$ is not covered by a $2$-torus bundle, and construct the maps $f_{n, 0}$ otherwise.
Finally, in section \ref{sec:mainproof} we convert the ``compressing multiply connected domains'' into compressing disks and finish the proof of Theorem \ref{Thm:MainTheorem-III}.

Note that in the following, all manifolds are assumed to be $3$ dimensional unless stated otherwise.

\subsection{Acknowledgments}
I would like to thank Gang Tian for his constant help and encouragement, John Lott for many long conversations and Richard Schoen for pointing out Choe's work to me.
I am also indebted to Bernhard Leeb and Hans-Joachim Hein, who contributed essentially to my understanding of Perelman's work.
Thanks also go to Simon Brendle, Alessandro Carlotto, Will Cavendish, Otis Chodosh, Daniel Faessler, Robert Kremser, Tobias Marxen, Rafe Mazzeo, Hyam Rubinstein, Andrew Sanders, Stephan Stadler and Brian White.

\section{Simplicial complexes} \label{sec:Defpolygonalcomplex}
We briefly recall the notion of simplicial complexes which will be used throughout the whole paper.
Note that in the following we will only be interested in simplicial complexes which are $2$ dimensional, pure and locally finite.
For brevity we will always implicitly make these assumptions when refering to the term ``simplicial complex''

\begin{Definition}[simplicial complex] \label{Def:simplcomplex}
A \emph{$2$-dimensional simplicial complex} $V$ is a topological space which is the union of embedded, closed $2$-simplices (triangles), $1$-simplices (intervals) and $0$-simplices (points) such that any two distinct simplices are either disjoint or their intersection is equal to another simplex whose dimension is strictly smaller than the maximal dimension of both simplices.
$V$ is called \emph{finite} if the number of these simplices is finite.

In this paper, we assume $V$ moreover to be \emph{locally finite} and \emph{pure}.
The first property demands that every simplex of $V$ is contained in only finitely many other simplices and the second property states that every $0$ or $1$-dimensional simplex is contained in a $2$-simplex.
We will also assume that all $2$ and $1$-simplices are equipped with differentiable parameterizations which are compatible with respect to restriction.

We will often refer to the $2$-simplices of $V$ as \emph{faces}, the $1$-simplices as \emph{edges} and the $0$-simplices as \emph{vertices}.
The \emph{$1$-skeleton} $V^{(1)}$ is the union of all edges and the \emph{$0$-skeleton} $V^{(0)}$ is the union of all vertices of $V$.
The \emph{valency} of an edge $E \subset V^{(1)}$ denotes the number of adjacent faces, i.e. the number of $2$-simplices which contain $E$.
The \emph{boundary} $\partial V$ is the union of all edges of valency $1$.
\end{Definition}

We will also use the following notion for maps from simplicial complexes into manifolds.

\begin{Definition}[piecewise smooth map]
Let $V$ be a simplicial complex, $M$ an arbitrary differentiable manifold (not necessarily $3$-dimensional) and $f : V \to M$ a continuous map.
We call $f$ \emph{piecewise smooth} if $f$ restricted to the interior of each face of $V$ is smooth and bounded in $W^{1,2}$ and if $f$ restricted to each edge $E \subset V^{(1)}$ is smooth away from finitely many points.
\end{Definition}

Given a Riemannian metric $g$ on $M$ and a sufficiently regular map $f : V \to M$ (e.g. piecewise smooth) we define its area, $\area (f)$, to be the sum of $\area (f|_{\Int F})$ over all faces $F \subset V$ and the length of the $1$-skeleton $\ell (f |_{V^{(1)}})$ to be the sum of $\ell (f|_E)$ over all edges $E \subset V^{(1)}$.

\section{Existence of minimizers of simplicial complexes} \label{sec:existenceofvminimizers}
\subsection{Introduction and overview}
Let in this section $(M, g)$ always be a compact Riemannian manifold (not necessarily $3$ dimensional) with $\pi_2 (M) = 0$.
We will also fix the following notation: for every continuous contractible loop $\gamma : S^1 \to M$ we denote by $A(\gamma)$ the infimum over the areas of all continuous maps $f : D^2 \to M$ which are continuously differentiable on the interior of $D^2$, bounded in $W^{1,2}$ and for which $f |_{\partial D^2} = \gamma$.

Consider a finite simplicial complex $V$ as well as a continuous map $f_0 : V \to M$ such that $f_0 |_{\partial V}$ is a smooth embedding.
The goal of this section is motivated by the question of finding an area-minimizer within the same homotopy class of $f_0$, i.e. a map $f : V \to M$ which is homotopic to $f_0 : V \to M$ relative to $\partial V$ and whose area is equal to 
\[ A(f_0) : = \inf \big\{ \area f' \;\; : \;\; f' \simeq f_0 \;\; \text{relative to $\partial V$} \big\}. \]
(Here the maps $f' : V \to M$ are assumed to be continuous and continuously differentiable when restricted to $V \setminus V^{(1)}$ and $V^{(1)}$ as well as bounded in $W^{1,2}$ when restricted to each face of $V$.)
This however seems to be a difficult problem, since it is not clear how to control e.g. the length of the $1$-skeleta of a sequence of minimizers.

To get around these analytical issues, we instead seek to minimize the quantity $\area (f) + \ell (f |_{V^{(1)}})$.
Here $ \ell (f |_{V^{(1)}})$ denotes the sum of the lengths of all edges of $V$ under $f$.
It will turn out that this change has no negative effect when we apply our results to the Ricci flow in section \ref{sec:areaevolutionunderRF}.
In other words, we are looking for maps $f : V \to M$ which are homotopic to $f_0$ relative $\partial V$ and for which $\area (f) + \ell (f |_{V^{(1)}})$ is equal (or close) to
\[ A^{(1)} (f_0) := \inf \big\{ \area (f') + \ell( f' |_{V^1} )  \;\; : \;\; f' \simeq f_0 \;\; \text{relative to $\partial V$} \big\}. \]
We will be able to show that such a minimizer exists in a certain sense.
To be precise, we will find a map $f : V^{(1)} \to M$ of regularity $C^{1,1}$ on the $1$-skeleton which can be extended to $V$ to a minimizing sequence for $A^{(1)}$.
This implies that the sum of $A(f|_{\partial F})$ over all faces $F \subset V$ plus $\ell (f)$ is equal to $A^{(1)}(f_0)$.
So the existence problem for $f$ is reduced to solving the Plateau problem for each loop $f |_{\partial F}$.
The only difficulty that we may encounter then is that $f |_{\partial F}$ can a priori have (finitely or infinitely many) self-intersections.
Unfortunately, taking this possibility into account makes several arguments quite tedious and might obscure the main idea in a forest of details.

The second goal of this section (see subsection \ref{subsec:1skeletonstruc}) is to understand the geometry of a minimizer along the $1$-skeleton.
In the case in which $f : V^{(1)} \to M$ is injective our findings can be presented as follows.
In this case we can solve the Plateau problem for the loop $f |_{\partial F}$ for each face $F \subset V$ and extend $f : V^{(1)} \to M$ to a map $f : V \to M$ which is smooth on $V \setminus V^{(1)}$ and $C^{1,1}$ on $V^{(1)}$.
Consider and edge $E \subset V^{(1)} \setminus \partial V$ of valency $v_E$ and denote by $\kappa : E \to TM$ the geodesic curvature (defined almost everywhere) of $f |_E$ and let $\nu^{(1)}_E, \ldots, \nu^{(v_E)}_E : E \to TM$ be unit vector fields which are normal to $f |_E$ and outward pointing tangential to $f$ to the faces $F \subset V$ which are adjacent to $E$.
A simple variational argument will then yield the identities
\begin{equation} \label{eq:simplenuiskappa}
 \nu^{(1)}_E +  \ldots + \nu^{(v_E)}_E = \kappa_E \qquad \text{and} \qquad \big\langle \nu^{(1)}_E +  \ldots + \nu^{(v_E)}_E, \kappa_E  \big\rangle \geq 0.
\end{equation}
This equality and inequality are the second main result of this section and some time is spent on expressing these identities in the case in which the loops $f |_{\partial F}$ possibly have self-intersections.
We remark that in the case in which $f |_{V^{(1)}}$ is injective this equality and a bootstrap argument can be used to show that $f$ is actually smooth on all of $V$.

Observe however that in general it might happen that two or more edges are mapped to the same segment under $f$ (this could also happen for subsegments of these edges or for subsegments of one and the same edge).
It would then be necessary to take the sum over all faces which are adjacent to either of these edges on the left hand side of (\ref{eq:simplenuiskappa}) and a multiple of $\kappa_E$ on the right hand side of the equation in (\ref{eq:simplenuiskappa}).
These combinatorics become even more involved by the fact that, at least a priori, $f |_{\partial F}$ can for example intersect in a subset of empty interior but positive measure.

\subsection{Construction and regularity of the map on the 1-skeleton} \label{subsec:regularityon1skeleton}
Consider again the given continuous map $f_0 : V \to M$ for which $f_0 |_{\partial V}$ is a smooth embedding and let $f_1, f_2, \ldots : V \to M$ be a minimizing sequence for $A^{(1)}(f_0)$.
To be precise, we want each $f_k$ to be continuous and homotopic to $f_0$ relative $\partial V$, continuously differentiable when restricted to $V \setminus V^{(1)}$ and $V^{(1)}$ as well as bounded in $W^{1,2}$ when restricted to each face and
\[ \lim_{k \to \infty} \big( \area (f_k) + \ell (f_k |_{V^{(1)}} ) \big) = A^{(1)}(f_0). \]
By compactness of $M$ we may assume that, after passing to a subsequence, $f_k |_{V^{(0)}}$ converges pointwise.
Next, observe that every edge $E \subset V^{(1)}$ is equipped with a standard parameterization by an interval $[0,1]$ (see Definition \ref{Def:simplcomplex}).
We can then reparameterize each $f_k$ such that for every edge $E \subset V^{(1)}$ the restriction $f_k |_E$ is parameterized by constant speed.
Since $\ell ( f_k |_E )$ is uniformly bounded, we can pass to another subsequence such that $f_K |_E$ converges uniformly.
So we may assume that $f_k |_{V^{(1)}}$ converges uniformly to a map $f : V^{(1)} \to M$ such that $f |_E$ is Lipschitz and such that $\ell (f |_{V^{(1)}}) \leq \liminf_{k \to \infty} \ell (f_k |_{V^{(1)}})$.
It is our first goal to derive regularity results for $f$.
Before doing this we first characterize the map $f$, so that we can forget about the sequence $f_k$.

\begin{Lemma} \label{Lem:existenceon1skeleton}
The map $f$ is parameterized by constant speed and if $F_1, \ldots, F_n$ are the faces of $V$, then
\[ A( f |_{\partial F_1} ) + \ldots + A (f |_{\partial F_n}) + \ell(f) = A^{(1)} (f_0). \]
Moreover, for every continuous map $f' : V^{(1)} \to M$ which is homotopic to $f_0$ relative to $\partial V$ we have
\[ A( f' |_{\partial F_1} ) + \ldots + A (f' |_{\partial F_n}) + \ell(f') \geq A^{(1)} (f_0). \]
\end{Lemma}

\begin{proof}
For every face $F_j$ consider the boundary loop $f |_{\partial F_j} : \partial F_j \approx S^1 \to M$ which is a Lipschitz map.
Recall that the loops $f_k |_{\partial F_j}$ converge uniformly to $f |_{\partial F_j}$.
So using the exponential map and assuming that $k$ is large enough, we can find a homotopy $H_k : \partial F_j \times [0,1] \to M$ between $f_k |_{\partial F_j}$ and $f |_{\partial F_j}$ which is Lipschitz everywhere and smooth on $\partial F_j \times (0,1)$ and whose area goes to $0$ as $k \to \infty$.
Gluing $H_k$ together with $f_k |_{F_j} : F_j \to M$ and mollifying around the seam yields a continuous map $f^*_{j,k} : F_j \to M$ which is smooth on $\Int F_j$ such that $f^*_{j,k} |_{\partial F_j} = f |_{\partial F_j}$ and such that $\area f^*_{j,k} - \area f_k |_{F_j}$ goes to $0$ as $k \to \infty$.
Hence $A( f |_{\partial F_j} ) \leq \liminf_{k \to \infty} \area f_k |_{F_j}$ and we obtain
\begin{multline*}
 A( f |_{\partial F_1} ) + \ldots + A (f |_{\partial F_n}) + \ell(f) \\ \leq \liminf_{k \to \infty} \big( \area (f_k |_{\partial F_1}) + \ldots + \area (f_k |_{\partial F_n}) + \ell ( f_k |_{V^{(1)}} ) \big) = A^{(1)} (f_0).
\end{multline*}
It remains to establish the reverse inequality, i.e. the last statement of the claim.
This will then also imply that $\lim_{k \to \infty} \ell (f_k |_{\partial V^{(1)}} ) = \ell (f)$ and hence that $f$ is parameterized by constant speed.

Consider a continuous map $f' : V^{(1)} \to M$.
We can find smoothings $f'_k : V^{(1)} \to M$ of $f'$ such that $f'_k$ converges uniformly to $f'$ and $\lim_{k \to \infty} \ell (f'_k) = \ell (f')$.
Now for every face $F_j$, we can again find a homotopy $H'_k : \partial F_j \times [0,1] \to M$ of small area and by another gluing argument, we can construct continuous maps $f'_{j,k} : F_j \to M$ which are smooth on $\Int F_j$ such that $\lim_{k \to \infty} \area f'_{j,k} = A ( f' |_{\partial F_j} )$.
Hence, we can extend each $f'_k : V^{(1)} \to M$ to a map $f''_k : V \to M$ of the right regularity such that 
\[ A^{(1)} (f_0) \leq \lim_{k \to \infty} \big( \area(f''_k) + \ell (f''_k |_{V^{(1)}}) \big) = A(f' |_{\partial F_1}) + \ldots + A(f' |_{\partial F_n}) + \ell (f'). \]
This proves the desired result.
\end{proof}

We also need the following isoperimetric inequality.

\begin{Lemma} \label{Lem:isoperimetricdistancetoaxis}
Let $\gamma : S^1 \to \IR^n$ be a rectifiable loop such that $\gamma$ restricted to the lower semicircle parameterizes an interval on the $x_1$-axis $x_2 = \ldots = x_n = 0$ and $\gamma$ restricted to the upper semicircle has length $l$.
Denote by $a$ the maximum of the euclidean norm of the $(x_2, \ldots, x_n)$ component of all points on $\gamma$ (i.e. the maximal distance from the $x_1$-axis).
Then $A(\gamma) \leq la$.
\end{Lemma}

\begin{proof}
Let $0 = s_0 < s_2 < \ldots < s_m = l$ be a subdivision of the interval $[0,l]$, let $y_i$ be the $x_1$-coordinate of $\gamma(s_i)$ and $\sigma_i$ a straight segment between $\gamma(s_i)$ and $(y_i, 0, \ldots, 0)$ for each $i = 0, \ldots, m$.
For each $i = 1, \ldots, m$ let $\gamma_i$ be the loop which consists of $\gamma |_{[s_{i-1}, s_i]}, \sigma_{i-1}, \sigma_i$ and the interval between $(y_{i-1}, 0, \ldots, 0), (y_i, 0, \ldots, 0)$.
We set $A^* (s_0, \ldots, s_m) = A(\gamma_1) +  \ldots + A( \gamma_m)$.

Let $i \in \{ 1, \ldots, m-1 \}$.
We claim that if we remove $s_i$ from the list of subdivisions, then the value of $A^* (s_0, \ldots, s_m)$ does not increase.
In fact, if $y_{i-1} \leq y_i \leq y_{i+1}$ or $y_{i-1} \geq y_i \geq y_{i+1}$, this is claim is true since any two maps $h_i, h_{i+1} : D^2 \to M$ which restrict to $\gamma_i, \gamma_{i+1}$ on $S^1$ can be glued together along $\sigma_i$.
On the other hand, if $y_{i-1} \leq y_{i+1} \leq y_i$, then $h_i, h_{i+1}$ can be glued together along the union of $\sigma_i$ with the interval between $(y_{i+1}, 0, \ldots, 0), (y_i, 0, \ldots, 0)$.
The other cases follow analogously.
Multiple application of this finding yields $A(\gamma) \leq A^* (s_0, \ldots, s_m)$.

Let now $\gamma'_i$ be the loop which consists of the straight segment between  $\gamma(s_{i-1}), \linebreak[1] \gamma(s_i)$, the segments $\sigma_{i-1}, \sigma_i$ and the interval between $(y_{i-1}, 0, \ldots, 0), (y_i, 0, \ldots, 0)$.
Moreover, let $\gamma''_i$ be the loop which consists of straight segment between $\gamma(s_{i-1}), \linebreak[1] \gamma(s_i)$ and $\gamma|_{[s_{i-1}, s_i]}$.
Then by the isoperimetric inequality and some basic geometry
\[ A(\gamma_i) \leq A(\gamma'_i) + A(\gamma''_i) \leq a \ell (\gamma|_{[s_{i-1}, s_i]}) + C (\ell (\gamma|_{[s_{i-1}, s_i]}))^2. \]
Adding up this inequality for all $i = 1, \ldots, m$ yields
\[ A(\gamma) \leq A^* (s_0, \ldots, s_m) \leq a l + \sum_{i=1}^m C (\ell (\gamma|_{[s_{i-1}, s_i]}))^2. \]
The right hand side converges to $0$ as the mesh size of the subdivisions approaches zero.
\end{proof}

The following is our main regularity result.

\begin{Lemma} \label{Lem:regularityon1skeleton}
The map $f : V^{(1)} \to M$ has regularity $C^{1,1}$ on every edge $E \subset V^{(1)}$.
\end{Lemma}

\begin{proof}
Let $E \subset V^1$ and equip $E$ with the smooth parameterization of an interval.
We now establish the regularity of the map $f_E = f |_E : E \to M$ up to the endpoints of $E$.
Assume $\ell (f|_E) > 0$, since otherwise we are done.
After scaling the interval by which $E$ is parameterized, we may assume without loss of generality that $f_E$ is parameterized by arclength, i.e. that
\[ \ell (f_E |_{[s_1, s_2]}) = s_2 - s_1 \qquad \text{for every interval} \qquad [s_1, s_2] \subset E. \]

Let $\varepsilon > 0$ be smaller than the injectivity radius of $M$ and observe that whenever we choose exponential coordinates $(y_1, \ldots, y_n)$ around a point $p \in M$ then under these coordinates we have the following comparison with the Euclidean metric $g_{\textnormal{eucl}}$:
\begin{equation} \label{eq:gminusgeuclforregularity}
 | g - g_{\textnormal{eucl}} | < C_1 r^2
\end{equation}
for some uniform constant $C_1$ (here $r$ denotes the radial distance from $p$).
Assume moreover that $\varepsilon$ is chosen small enough such that $g$ is $2$-Bilipschitz to $g_{\textnormal{eucl}}$.

Consider three parameters $s_1, s_2, s_3 \in E$ such that $s_1 < s_2 < s_3 < s_1 + \frac1{10} \varepsilon$.
We set $x_i = f_E(s_i)$, $l = |s_3 - s_1| = \ell( f_E |_{[s_1, s_3]})$ as well as $d = \dist(x_1, x_3)$ and we denote by $\gamma$ a minimizing geodesic segment between $x_1$ and $x_3$.
Consider now the competitor map $f'$ which agrees with $f$ on $V^{(1)} \cup (E \setminus (s_1, s_2))$ and which maps the interval $[s_1, s_3]$ to the segment $\gamma$.

Let us first bound the area gain for such a competitor.
Denote by $\gamma^* : S^1 \to M$ the loop which consists of the curves $f_E |_{[s_1, s_3]}$ and $\gamma$.
Consider geodesic coordinates $(y_1, \ldots, y_n)$ around $x_1$ such that $\gamma$ can be parameterized by $(t, 0, \ldots, 0)$ and denote by $a$ the maximum of the euclidean norm of the $(y_2, \ldots, y_n)$-component of $f_E$ on $[s_1, s_3]$.
By Lemma \ref{Lem:isoperimetricdistancetoaxis} we have
\[ A(\gamma^*)  \leq 2 l a. \]
(Recall that $g$ is $2$-Bilipschitz to the euclidean metric.)
Let $F_1, \ldots, F_v$ be the faces which are adjacent to $E$.
Then for each $j = 1, \ldots, v$ we have
\[ A( f' |_{\partial F_j} ) \leq A( f |_{\partial F_j} ) + A(\gamma^*) \leq A( f|_{\partial F_j}) + 2la. \]
Moreover, $\ell (f') \leq \ell(f) - l + d$.
So by the inequality of Lemma \ref{Lem:existenceon1skeleton} we obtain
\begin{equation} \label{eq:lminusd}
 l - d \leq 2v \cdot l a .
\end{equation}

Let now $l'$ be the length of the segment parameterized by $f_E |_{[s_1, s_3]}$ with respect to the euclidean metric $g_{\textnormal{eucl}}$ in the coordinate system $(y_1, \ldots, y_n)$.
Then $\tfrac12 l' \leq l \leq 2l'$.
Moreover, we obtain the following improved bound on $l'$ using (\ref{eq:gminusgeuclforregularity}):
\begin{multline*}
 l = \int_{s_1}^{s_3} \sqrt{g(f'_E (s), f'_E(s))} ds \geq \int_{s_1}^{s_3} \sqrt{ (1 - C_1 (l')^2) g_{\textnormal{eucl}} (f'_E(s), f'_E(s))} ds \\
  \geq \sqrt{ 1- 4C_1 l^2} \; l'.
\end{multline*}
By basic trigonometric estimates with respect to the euclidean metric in the coordinate system $(y_1, \ldots, y_n)$ we obtain
\[ d^2 + 4 a^2 \leq (l')^2  . \]
So
\begin{equation} \label{eq:dandarelationprecise}
  (1-4C_1 l^2) (d^2 + 4 a^2) \leq l^2.
\end{equation}
Plugging in (\ref{eq:lminusd}) yields with $c = v^{-2}$
\[ (1-4C_1 l^2) ( l^2 d^2 + c (l-d)^2) \leq  l^4. \]
And hence for small enough $l$
\[ \tfrac{c}2 (l-d)^2 \leq l^2 (l-d)(l+d) + 4 C_1 l^4 d^2 \leq 2 l^3 (l-d) + 4 C_1 l^6. \]
This inequality implies that if $l-d \geq l^3$, then $\tfrac{c}2 (l-d) \leq 2 l^3 + 4 C_1 l^3$.
So in general there is a universal constant $C_2$ such that
\begin{equation} \label{eq:lminusdcubebound}
 l - d \leq C_2 l^3.
\end{equation}
In particular, if $l$ is smaller than some uniform constant, then
\[ \tfrac12 d \leq l \leq 2d. \]
We will in the following always assume that this bound holds whenever we compare the intrinsic and extrinsic distance between two close points on $f_E$.

Next, we plug (\ref{eq:lminusdcubebound}) back into (\ref{eq:dandarelationprecise}) and obtain a bound on $a$ for small $l$:
\[ a \leq \sqrt{\frac{(l-d) (l+d) + 4 C_1 l^2 d^2}{4 ( 1- 4 C_1 l^2)}} \leq \sqrt{ C_2 l^3 \cdot 2 l + 4 C_1 l^2 d^2} \leq C_3 l^2 \]
for some uniform constant $C_3$.
Now consider the point $x_2$ on $f_E ([s_1, s_3])$, set $l = \ell (f_E |_{[s_1,s_2]})$ and let $\alpha_1 \geq 0$ be the angle between the geodesic segment $\gamma$ from $x_1$ to $x_3$ and the geodesic segment $\gamma_1$ from $x_1$ to $x_2$.
Observe that the angle $\alpha$ between $\gamma$ and $\gamma_1$ is the same with respect to both $g$ and $g_{\textnormal{eucl}}$.
Moreover, by our previous conclusion, the length of $\gamma_1$ is bounded from below by $\frac12 l_1$.
So by a basic trigonometric we find that there are uniform constants $\varepsilon_0 > 0$ and $C_4 < \infty$ such that we have
\begin{equation} \label{eq:angleboundforregularity}
 \alpha \leq C_4 l \qquad \text{if} \qquad l_1 \geq \tfrac12 l \quad \text{and} \quad l < \varepsilon_0 .
\end{equation}

We can now establish the differentiability of $f_E$.
Let $s, s', s'' \in E$ such that $s < s' < s'' < s + \varepsilon_0$, set $x = f_E(s)$, $x' = f_E (s')$, $x'' = f_E (s'')$ and choose minimizing geodesic segments $\gamma', \gamma''$ between $x, x'$ and $x, x''$.
Let $\alpha \geq 0$ bet the angle between $\gamma', \gamma''$ at $x$.
For each $i \geq 1$ for which $s + 2^{-i} \in E$ we set $x_i = f_E ( s+ 2^{-i})$ and we choose a minimizing geodesic segment $\gamma_i$ between $x$ and $x_i$.
Choose moreover indices $i' \geq i'' \geq 1$ such that $2^{- i'} \leq s' - s< 2^{- i' + 1}$ and $2^{- i''} \leq s'' - s< 2^{- i'' + 1}$.
Then by (\ref{eq:angleboundforregularity})
\begin{multline*}
\alpha \leq \nangle_x (\gamma'', \gamma_{i''}) + \nangle_x (\gamma_{i''}, \gamma_{i'' + 1}) + \ldots + \nangle_x (\gamma_{i' - 2}, \gamma_{i' - 1}) + \nangle_x (\gamma_{i' - 1}, \gamma') \\
\leq C_4 (s''-s) + C_4 2^{-i''} + C_4 2^{-i'' - 1} + \ldots \\
 \leq C_4 (s'' - s) + 2 C_4 2^{-i''} \leq 3 C_4 (s'' - s).
\end{multline*}
Note also that by (\ref{eq:lminusdcubebound}) the quotients $\frac{\ell(\gamma')}{s'-s}$ and $\frac{\ell(\gamma'')}{s''-s}$ converge to $1$ as $s'' \to s$.
Altogether, this shows that the right-derivative of $f_E$ exists, has unit length and that
\begin{equation} \label{eq:anglebetweenfsandgamma}
 \nangle_x \big( \tfrac{d}{ds^+} f_E (s), \gamma'' \big) \leq 3 C_4  (s'' - s).
\end{equation}
The existence of the left-derivative together with the analogous inequality follows in the same way.
In order to show that the right and left-derivatives agree in the interior of $E$, it suffices to show for any $s \in \Int E$, that the angle between the geodesic segments between $f_E (s), f_E(s-s')$ and $f_E (s), f_E (s+s')$ goes to $\pi$ as $s' \to 0$.
This follows immediately from (\ref{eq:angleboundforregularity}) and the fact that the sum of the angles of small triangles in $M$ goes to $\pi$ as the circumference goes to $0$.

Finally, we establish the Lipschitz continuity of the derivative $f'_E(s)$.
Let $s_1, s_3 \in E$ such that $s_1 < s_3 < s_1 + \varepsilon_0$ and let $s_2 = \frac12 (s_1 + s_3)$ be the midpoint on $f_E$.
Let $\gamma$ and $\gamma_1$ be defined as before and let $\gamma_3$ be the geodesic segment between $x_2 = f_E(s_2)$ and $x_3 = f_E (s_3)$.
Using (\ref{eq:gminusgeuclforregularity}) it is not difficult to see that if we choose geodesic coordinates around $x_1$ or $x_3$, then we can compare angles at different points on $f_E ([s_1, s_3])$ up to an error of $O(|s_3 - s_1|^2)$.
So we can estimate using (\ref{eq:angleboundforregularity}) and (\ref{eq:anglebetweenfsandgamma})
\begin{multline*}
 \nangle ( f'_E(s_1), f'_E(s_3)) \leq \nangle (f'_E (s_1), \gamma_1) + \nangle (\gamma_1, \gamma) \\ + \nangle (\gamma, \gamma_3) + \nangle (\gamma_3, f'_E(s_3)) + O(|s_3 - s_1|^2) \\ \leq 3 C_4 |s_2 - s_1| + 2 C_4 |s_3 - s_1| + 3 C_4 |s_3 - s_2| + O(|s_3 - s_1|^2)  \leq C_5 |s_3 - s_1|
\end{multline*}
for some uniform constant $C_5$.
This finishes the proof.
\end{proof}

Now if for every face $F \subset V$ the map $f |_{\partial F}$ is injective (i.e. an embedding in a proper parameterization), then by solving the Plateau problem for each face (cf \cite{Mor}) we obtain an extension $\td{f} : V \to M$ of $f$ which is homotopic to $f_0$ and for which $\area \td{f} + \ell ( \td{f} |_{V^{(1)}} ) = A^{(1)} (f_0)$.
So in this case, the existence of the minimizer is ensured.
In general however, we need take into account the possibility that $f |_{\partial F}$ has self-intersections.
Note that there might be infinitely many such self-intersections and the set of self-intersections might even have positive $1$ dimensional Hausdorff measure.
This adds some technicalities to the following discussion.

\subsection{Results on self-intersections and the Plateau problem}
The following Lemma states that two intersecting curves agree up to order $2$ almost everywhere on their set of intersection.

\begin{Lemma} \label{Lem:speedandcurvatureagreeonintersection}
Let $\gamma : [0, l] \to M$ be a curve of regularity $C^{1,1}$ which is parameterized by arclength.
Then the geodesic curvature along $\gamma$ is defined almost everywhere, i.e. there is a vector field $\kappa : [0,l] \to TM$ along $\gamma$ (i.e. $\kappa (s) \in T_{\gamma(s)} M$ for all $s \in [0,l]$) and a null set $N \subset [0,l]$ such that at each $s \in [0,l] \setminus N$ the curve $\gamma$ is twice differentiable and the geodesic curvature at $s$ equals $\kappa (s)$.

Consider now two such curves $\gamma_1 : [0,l_1] \to M$, $\gamma_2 : [0,l_2] \to M$ with geodesic curvature vector fields $\kappa_1, \kappa_2$.
Assume additionally that $\gamma_1, \gamma_2$ are injective embeddings which are contained in a coordinate chart $(U, (x_1, \ldots, x_n))$ in such a way that there is a vector $v \in \IR^n$ with the property that $\langle \gamma'_i (s), v \rangle \neq 0$ with respect to the euclidean metric for all $s \in [0, l_i]$ and $i = 1,2$.

Let $X_1 = \{ s \in [0,l_1] \;\; : \;\; \gamma_1(s) \in \gamma_2([0,l_2]) \}$ and $X_2 = \{ s \in [0,l_2] \;\; : \;\; \gamma_2(s) \in \gamma_1([0,l_1]) \}$ be the parameter sets of self-intersections.
Then there is a continuously differentiable map $\varphi : [0,l_1] \to \IR$ whose derivative  vanishes nowhere such that $\varphi(X_1) = X_2$ and such that $\gamma_1(s) = \gamma_2 (\varphi(s))$ whenever $s \in X_1$.
Moreover, there are null sets $N_i \subset X_i$ such that $\varphi(N_1) = N_2$ and such that for all $s \in X_1 \setminus N_1$ we have $\varphi'(s) = \pm 1$, $\gamma'_1 (s) = \gamma'_2 (\varphi(s)) \varphi'(s)$ and $\kappa_1 (s) = \kappa_2 (\varphi(s))$.
\end{Lemma}

\begin{proof}
The first statement follows from the fact that a Lipschitz function is differentiable almost everywhere.
Observe that the geodesic curvature can be computed in terms of the first and second derivative of the curve in a local coordinate system.
So in particular, we only need to consider the case in which $M$ is euclidean space equipped with the cartesian coordinate system $(x_1, \ldots, x_n)$ for the second statement.

Let $\varphi : [0,l_1] \to \IR$ be the composition of the projection $s \mapsto \langle \gamma_1(s), v \rangle$ with the inverse of the projection $s \mapsto \langle \gamma_2(s), v \rangle$.
It is then clear that $\varphi(X_1) = X_2$ and $\gamma_1(s) = \gamma_2 (\varphi(s))$ whenever $s \in X_1$.
Moreover, $\varphi'(s) \neq 0$ for all $s \in [0, l_1]$.

Next, let $N'_i \subset [0,l_i]$ be the null sets from the first part outside of which $\kappa_i$ is equal to the geodesic curvature of $\gamma_i$.
Let moreover, $N^*_1 \subset X_1$ be the set of isolated points of $X_i$.
Clearly $N^*_1$ is a null set.
We now claim that the Lemma holds for $N_1 = X_1 \cap (N'_1 \cup \varphi^{-1} (N'_2) \cup N^*_1 )$ and $N_2 = \varphi(N_1)$.
Obviously, $N_1, N_2$ are null sets.
Let now $s \in X_1 \setminus N_1$.
Observe that for $s'$ close to $s$, we have
\[ \gamma_1 (s') = \gamma_1 (s) + (s'-s) \gamma'_1 (s) + \tfrac12 (s'-s)^2 \kappa_1(s) + o((s'-s)^2). \]
Similarly, for every $s''$ close to $\varphi(s)$
\[ \gamma_2 (s'') = \gamma_1 (s) + (s''- \varphi(s)) \gamma'_2 (\varphi(s)) + \tfrac12 (s''-\varphi(s))^2 \kappa_2(\varphi(s)) + o((s''- \varphi(s))^2). \]
Since $s \notin N^*_1$, there is a sequence of parameters $s'_k \to s$, $s'_k \neq s$ such that with $s''_k = \varphi(s''_k)$ we have $\gamma_1(s'_k) = \gamma_2(s''_k)$.
Due to the fact that $\varphi$ is continuously differentiable,
\[ s''_k - \varphi(s) = \varphi'(s) (s'_k - s) + o(s'_k - s). \]
So we obtain from the expansions for $\gamma_1, \gamma_2$ that
\begin{multline*}
 (s'_k - s) \gamma'_1(s) + o (s'_k - s) = \gamma_1 (s'_k) - \gamma_1 (s) \\
 = \gamma_2 (s''_k) - \gamma_2 (\varphi(s)) = \varphi'(s) (s'_k - s) \gamma'_2 (\varphi(s)) + o (s'_k - s).
\end{multline*}
This implies that $\gamma'_1(s) = \gamma'_2(\varphi(s)) \varphi'(s)$ and $\varphi'(s) = \pm 1$ follows from the fact that $|\gamma'_1(s)| = |\gamma'_2(s)| = 1$.

Next, we take the scalar product of the expansions for $\gamma_1, \gamma_2$ with an arbitrary vector $v^* \in \IR^n$ which is orthogonal to $\gamma'_1(s)$ and hence also to $\gamma'_2(\varphi(s))$.
Then
\begin{multline*}
\tfrac12 (s'_k - s)^2 \big\langle \kappa_1(s), v^* \big\rangle + o ((s'_k -s)^2) = \big\langle \gamma_1(s'_k) - \gamma_1(s), v^* \big\rangle \\
= \big\langle \gamma_2 (s''_k) - \gamma_1 (s), v^* \big\rangle 
= \tfrac12 (s'_k - s)^2 \big\langle \kappa_2(\varphi(s)), v^* \big\rangle + o((s'_k - s)^2).
\end{multline*}
So $\langle \kappa_1(s), v^* \rangle = \langle \kappa_2(\varphi(s)), v^* \rangle$.
Since $\kappa_1(s), \kappa_2(\varphi(s))$ are orthogonal to $\gamma'_1(s)$, we conclude that $\kappa_1(s) = \kappa_2(\varphi(s))$.
\end{proof}

In the remainder of this subsection, we state the solution of the Plateau problem for loops with (possibly infinitely many) self-intersections.
We will hereby always make use of the following terminology.

\begin{Definition}
Let $\gamma : S^1 \to M$ be a continuous and contractible loop.
A continuous map $f : D^2 \to M$ is called a \emph{solution to the Plateau problem for $\gamma$} if $f$ is smooth, harmonic and almost conformal on the interior of $D^2$ and if $\area f = A(\gamma)$ and if there is an orientation preserving homeomorphism $\varphi : S^1 \to S^1$ such that $f |_{S^1} = \gamma \circ \varphi$.
\end{Definition}

We will also need a variation of the Douglas condition.

\begin{Definition}[Douglas-type condition] \label{Def:DouglasCondition}
Let $\gamma : S^1 \to M$ be a piecewise $C^1$ immersion which is contractible in $M$.
We say that $\gamma$ \emph{satisfies the Douglas-type condition} if for any distinct pair of parameters $s, t \in S^1$, $s \neq t$ with $\gamma(s) = \gamma(t)$ the following is true:
Consider the loops $\gamma_1, \gamma_2$ which arise from restricting $\gamma$ to the arcs of $S^1$ between $s$ and $t$.
Then
\[ A(\gamma) < A(\gamma_1) + A(\gamma_2). \]
\end{Definition}

We can now state a slightly more general solution of the Plateau problem.

\begin{Proposition} \label{Prop:PlateauProblem}
Consider a loop $\gamma : S^1 \to M$ which is a piecewise $C^1$-immersion and which is contractible in $M$.
Assume first that $\gamma$ satisfies the Douglas-type condition.
Then the following holds.
\begin{enumerate}[label=(\alph*)]
\item There is a solution $f : D^2 \to M$ to Plateau problem for $\gamma$.
\item If $\gamma$ has regularity $C^{1,1}$ on $U \cap S^1$ for some open subset $U \subset D^2$ then for every $\alpha < 1$ the map $f$ (from assertion (a)) locally has regularity $C^{1, \alpha}$ on $U$.
Moreover, the restriction $f |_{S^1}$ has non-vanishing derivative on $U \cap S^1$ away from finitely many branch points.

Similarly, if $\gamma$ has regularity $C^{m, \alpha}$ for some $m \geq 2$ and $\alpha \in (0,1)$ on $U \cap S^1$, then $f$ locally has regularity $C^{m, \alpha}$ on $U$.
\item Assume that we have a sequence $\gamma_k : S^1 \to M$ of continuous maps which uniformly converge to $\gamma$ and consider solutions of the Plateau problem $f_k : D^2 \to M$ for each such $\gamma_k$.
Then there are conformal maps $\psi_k : D^2 \to D^2$ such that the maps $f_k \circ \psi_k : D^2 \to M$ subconverge uniformly on $D^2$ and smoothly on $\Int D^2$ to a map $f : D^2 \to M$ which solves the Plateau problem for $\gamma$.

Furthermore, if $\gamma$ has regularity $C^{1,1}$ on $U \cap S^1$ for some open subset $U \subset D^2$ and $\gamma_k$ locally converges to $\gamma$ on $U \cap S^1$ in the $C^{1,\alpha}$ sense for some $\alpha \in (0,1)$, then the sequence $f_k$ actually converges to $f$ on $U$ in the $C^{1, \alpha'}$ sense for every $\alpha' < \alpha$.
\end{enumerate}
Next assume that $\gamma$ does not necessarily satisfy the Douglas-type condition and let $p$ be the number of places where $\gamma$ is not differentiable (i.e. where the right and left-derivatives don't agree).
Then there are finitely or countably infinitely many loops $\gamma_1, \gamma_2, \ldots : S^1 \to M$ which are piecewise $C^1$-immersions and contractible in $M$ such that:
\begin{enumerate}[label=(\alph*), start=4]
\item The loops $\gamma_i$ satisfy the Douglas-type condition.
\item Each $\gamma_i$ is composed of finitely many subsegments of $\gamma$ in such a way that each such subsegment of $\gamma$ is used at most once for the entire sequence $\gamma_1, \gamma_2, \ldots$.
\item For each $i$ let $p_i$ be the number of places where $\gamma_i$ is not differentiable.
Then $p_i = 2$ for all but finitely many $i$ and
\[ \sum_i (p_i - 2) \leq  p - 2. \]
\item We have
\[ A(\gamma) = \sum_i A(\gamma_i). \]
\item For any set of solutions $f_1, f_2, \ldots : D^2 \to M$ to the Plateau problems for $\gamma_1, \gamma_2, \ldots$ and every $\delta > 0$ there is a map $f_\delta : D^2 \to M$ and an open subset $D_\delta \subset D^2$ such that the following holds:
$f_\delta |_{S^1} = \gamma$ and $f_\delta$ restricted to each connected component of $D_\delta$ is a diffeomorphic reparameterization of some $f_i$ restricted to an open subset of $D^2$ in such a way that every $i$ is used for at most one component of $D_\delta$.
Moreover
\[ \area f_\delta |_{D^2 \setminus D_\delta} < \delta \qquad \text{and} \qquad \area f_\delta < A(\gamma) + \delta. \]
\end{enumerate}
\end{Proposition}

\begin{proof}
Observe that the Douglas-type condition allows the following conclusion:
Whenever $f_k : D^2 \to M$ with $f_k |_{S^1} = \gamma$ and $\lim_{k \to \infty} \area f_k = A(\gamma)$  is a minimizing sequence and $\sigma_k \subset D^2$ is a sequence of embedded curves for which $\lim_{k \to \infty} \ell ( f_k |_{\sigma_k} ) = 0$, then the distance between the two endpoints of $\sigma_k$ goes to $0$ as $k \to \infty$.

Assertion (a) now follows directly using the methods of \cite{Mor} since the embeddednes of $\gamma$ was only used in this paper for the previous conclusion.
The first part of assertion (c) follows for the same reasons.

The proof of assertion (b) in the case in which $\gamma$ is $C^2$ on $U \cap S^1$ can be found in \cite{HH}.
We remark that in the case, in which $\gamma$ is only $C^{1,1}$ on $U \cap S^1$ and $g$ is locally flat on $U$, assertion (b) is a consequence of \cite{Kin}.
For our purposes, however it is enough to note that the methods of the proof of \cite{HH} carry over to the case in which $\gamma$ is only $C^{1,1}$ on $U \cap S^1$.
We briefly point out how this can be done:
The first step in \cite{HH} consists of the choice of a local coordinate system $(x_1, \ldots, x_n)$ in which $\gamma$ is locally mapped to the $x_n$-axis.
For the subsequent estimates, this coordinate system has to be of class $C^2$.
In the case in which $\gamma$ is only $C^{1,1}$ on $U \cap S^1$, we can choose a sequence of coordinate systems $(x^k_1, \ldots, x^k_n)$, which are uniformly bounded in the $C^2$ sense, and which converge to a coordinate system $(x^\infty_1, \ldots, x^\infty_n)$ of regularity $C^{1,1}$ in every $C^{1,\alpha}$ norm and in this coordinate system $\gamma$ is locally mapped to the $x_n$-axis.
The minimal surface equation in the coordinate system $(x^k_1, \ldots, x^k_n)$ implies an equation of the form $|\triangle y^k| \leq \beta |\nabla y^k|^2$ for $y^k = (x^k_1, \ldots, x^k_{n-1}) \circ f$ where $\beta$ can be chosen independently of $k$.
Moreover, $y^k$ restricted to $U \cap S^1$ converges to $0$ in every $C^{1,\alpha}$ norm as $k \to \infty$.
Let $U''' \Subset U'' \Subset U' \Subset U$ be an arbitrary compactly contained open subsets.
A closer look at the proof of the ``Hilfssatz'' in \cite{Hei} yields that for every $r > 0$ we have the estimate $|y^k| < C r$ on $U' \cap (D^2(1- r) \setminus D^2(1- 2r))$ if $k$ is large depending on $r$.
It follows then that $\Vert y^k \Vert_{C^1 (U'' \cap D^2 (1-r))} < C$ for every $r > 0$ and large $k$.
This implies $\Vert y^\infty \Vert_{C^1 (U'')} < C$ and hence $\Vert y^k \Vert_{C^1(U'')} < 2C$ for large $k$. 
Standard elliptic estimates applied to the equation $|\triangle y^k| < 4\beta C^2$ then yield that $\Vert y^k \Vert_{C^{1, \alpha}(U''')} < C'$ for large $k$.
The regularity of $x^k_n \circ f$ and the fact that branch points are isolated also follow similarly as in \cite{HH}.

The second  part of assertion (c) follows in a similar manner.
We just need to choose the local coordinate systems $(x^k_1, \ldots, x^k_n)$ such that both $(x^k_1, \ldots, x^k_n) \circ \gamma$ and $(x^k_1, \ldots, x^k_n) \circ \gamma_k$ locally converge to the $x_n$-axis in the $C^{1,\alpha}$ sense. 

Now consider the case in which $\gamma$ does not satisfy the Douglas-type condition.
Then the remaining assertions follow from the methods of Hass (\cite{Hass-Plateau}).
For completeness, we briefly recall his proof.

We will inductively construct a (finite or infinite) sequence of straight segments $\sigma_1, \sigma_2, \ldots \subset D^2$ between pairs of points $s,t \in S^1$ with $\gamma(s) = \gamma(t)$, such that any two distinct segments don't intersect in their interior and such that the following holds for all $k \geq 0$:
Consider the (unique) extension $\gamma_k : S^1 \cup \sigma_1 \cup \ldots \cup \sigma_k  \to M$ of the map $\gamma$ which is constant on each $\sigma_i$.
Then we impose the condition that the sum $A (\gamma_k |_{\partial \Omega})$ over all connected components $\Omega \subset \Int D^2 \setminus (\sigma_1 \cup \ldots \cup \sigma_k)$ is equal to $A(\gamma)$.
(Note that every such component is bounded by some of the $\sigma_i$ and some arcs of $S^1$.)

Having constructed segments $\sigma_1, \ldots, \sigma_k$, we will choose $\sigma_{k+1}$ as follows:
Consider all components $\Omega \subset \Int D^2 \setminus (\sigma_1 \cup \ldots \cup \sigma_k)$ such that $\gamma_k |_{\partial \Omega}$ does not satisfy the Douglas-type condition (or to be precise, such that the loop which is composed of the restriction of $\gamma$ to $S^1 \cap \partial \Omega$ does not satisfy the Douglas-type condition).
If there is no such $\Omega$, then we are done.
Otherwise we pick an $\Omega$ for which $\ell (\gamma |_{S^1 \cap \partial\Omega})$ is maximal.
By our assumption, we can find a straight segment $\sigma \subset D^2$ connecting two distinct parameters $s,t \in S^1 \cap \partial \Omega$ such that if we denote by $\Omega', \Omega''$ the two components of $\Omega \setminus \sigma'$, then
\begin{equation} \label{eq:OmegaOmegas}
 A (\gamma_k |_{\partial \Omega}) = A (\gamma_k |_{\partial \Omega'}) + A (\gamma_k |_{\partial \Omega''}).
\end{equation}
It follows that we are allowed to choose $\sigma_{k+1} = \sigma$ for any such $\sigma$.
Now pick $\sigma$ amongst all such straight segments such that $\min \{ \ell (\gamma |_{S^1 \cap \partial\Omega'}), \ell (\gamma |_{S^1 \cap \partial\Omega''}) \}$ is larger than $\frac12$ times the supremum of this quantity over all such $\sigma$ and set $\sigma_{k+1} = \sigma$.

Having constructed the sequence $\sigma_1, \sigma_2, \ldots$, we let $X \subset D^2$ be the closure of $\sigma_1 \cup \sigma_2 \cup \ldots$ and we let $\gamma_X : S^1 \cup X \to M$ be the obvious extension.
Then all components $\Omega \subset \Int D^2 \setminus X$ are bounded by finitely many straight segments and arcs of $S^1$.
We claim that $A(\gamma)$ is equal to the sum of $A(\gamma_X |_{\partial \Omega})$ over all such components:
Let $\Omega_1, \ldots, \Omega_N$ be arbitrary, pairwise distinct components of $\Int D^2 \setminus X$.
Then there is a $k_0$ such that for all $k > k_0$ these components lie in different components $\Omega_{1, k}, \ldots, \Omega_{N, k}$ of $\Int D^2 \setminus (\sigma_1 \cup \ldots \cup \sigma_k)$.
Moreover $\Omega_{j, k} \to \Omega_j$ as $k \to \infty$.
So $\lim_{k \to \infty} A (\gamma_X |_{\partial \Omega_{j, k}} ) = A (\gamma_X |_{\partial\Omega_j})$ for each $j = 1, \ldots, N$.
Since the choice of the $\Omega_j$ was arbitrary, this shows that the sum of $A (\gamma_X |_{\partial \Omega})$ over all connected components $\Omega \subset \Int D^2 \setminus X$ is not larger than $A(\gamma)$.
The other direction is clear.

Next, we show that for each component $\Omega \subset \Int D^2 \setminus X$, the loop $\gamma_X |_{\partial \Omega}$ satisfies the Douglas-type condition.
If not, then we could separate $\Omega$ into two non-empty components $\Omega', \Omega''$ along a straight line $\sigma$ between two parameters $s,t \in S^1$ for which $\gamma(s) = \gamma(t)$ such that (\ref{eq:OmegaOmegas}) holds for $\gamma_X$ instead of $\gamma_k$.
Choose a sequence $\Omega_k \subset \Int D^2 \setminus (\sigma_1 \cup \ldots \cup \sigma_k)$ such that $\Omega_1 \supset \Omega_2 \supset \ldots$ and such that $\Omega_k \to \Omega$ as $k \to \infty$.
Let moreover $\Omega'_k, \Omega''_k$ be the components of $\Omega_k \setminus \sigma$ such that $\Omega'_k \to \Omega'$ and $\Omega''_k \to \Omega''$.
Then $\lim_{ k \to \infty } A (\gamma_k |_{\partial \Omega_k}) = A (\gamma_X |_{\partial \Omega})$ and $\lim_{ k \to \infty } A (\gamma_k |_{\partial \Omega'_k}) = A (\gamma_X |_{\partial \Omega'})$ and $\lim_{ k \to \infty } A (\gamma_k |_{\partial \Omega''_k}) = A (\gamma_X |_{\partial \Omega''})$.
Moreover, for all $k \geq 1$
\begin{multline*}
 A (\gamma_1 |_{\partial \Omega'_1}) + A (\gamma_1 |_{\partial \Omega''_1}) \leq A (\gamma_k |_{\partial \Omega'_k}) + A(\gamma_k |_{\partial (\Omega'_1 \setminus \Omega'_k)}) + A (\gamma_k |_{\partial \Omega''_k}) + A(\gamma_k |_{\partial (\Omega''_1 \setminus \Omega''_k)}) \\
 = A (\gamma_k |_{\partial \Omega'_k}) + A (\gamma_k |_{\partial \Omega''_k}) + A(\gamma_k |_{\partial (\Omega_1 \setminus \Omega_k)}) \\
 = A (\gamma_k |_{\partial \Omega'_k}) + A (\gamma_k |_{\partial \Omega''_k}) + A(\gamma_k |_{\partial \Omega_1 }) - A(\gamma_k |_{\partial \Omega_k}).
\end{multline*}
Letting $k \to \infty$ yields
\[ A (\gamma_1 |_{\partial \Omega'_1}) + A (\gamma_1 |_{\partial \Omega''_1}) \leq A (\gamma_1 |_{\partial \Omega_1}). \]
Since the opposite inequality is trivially true, we must have equality.
This is however a contradiction, because by our construction of the sequence $\sigma_1, \sigma_2, \ldots$ we must have picked $\sigma$ earlier and hence $\sigma_k = \sigma$ for some $k$.

Assertions (d), (e) and (g) follow immediately.
By the fact that $\gamma$ is a piecewise immersion, we can deduce that all but finitely many components of $\Omega \subset \Int D^2 \setminus X$ are bounded by exactly two straight segments and two arcs.
Assertion (f) follows now easily.
Finally, the functions $f_\delta$ from assertion (h) can be constructed by parameterizing the solutions $f_i$ by the corresponding component of $\Int D^2 \setminus X$ and mollifying.
\end{proof}

The following variational property is a direct consequence of assertion (h) and will be used twice in this paper.

\begin{Lemma} \label{Lem:variationofA}
Consider a contractible, piecewise $C^1$-immersion $\gamma : S^1 \to M$, let $\gamma_i$ be the loops from the second part of Proposition \ref{Prop:PlateauProblem} and consider solutions $f_i : D^2 \to M$ to the Plateau problem for each $\gamma_i$.
Let $(g_t)_{t \in [0,\varepsilon)}$ be a smooth family of Riemannian metrics such that $g_0 = g$ (not necessarily the Ricci flow) and denote by $A_t (\gamma)$ the infimum over the areas of all spanning disks with respect to the metric $g_t$.
Then in the barrier sense
\[ \frac{d}{dt^+} \Big|_{t = 0} A_t ( \gamma ) \leq \sum_i \int_{D^2} \frac{d}{dt} \Big|_{t = 0} d{\vol}_{f^*_i (g_t)} \]
Here $d{\vol}_{f^*_i (g_t)}$ denotes the volume form of the pull-back metric $f^*_i (g_t)$.
\end{Lemma}

\begin{proof}
Due to the smoothness of the family $(g_t)$, we can find a constant $C < \infty$ such that for any two vectors $v, w \in T M$ based at the same point and every $t \in [0, \varepsilon/2)$ we have
\[ \big| g_t (v,w) - g_0 (v,w) - t \partial_t  g_0 (v,w) \big| \leq C t^2 |v|_0 |w|_0. \]
Let now $\delta > 0$ be a small constant and consider the map $f_\delta : D^2 \to M$ from Proposition \ref{Prop:PlateauProblem}(h).
It is then not difficult to see that there is a constant $C' < \infty$, which is independent of $\delta$, such that for small $t$
\[ \bigg| \area_t f_\delta - \area_0 f_\delta - t \int_{D^2} \frac{d}{dt} \Big|_{t=0} d {\vol}_{f_\delta^* (g_t)} \bigg| \leq C' t^2 \area_0 f_\delta. \]
So we find that
\[ A_t (  \gamma ) \leq \area_0 f_\delta + t \int_{D^2} \frac{d}{dt} \Big|_{t=0} d {\vol}_{f_\delta^* (g_t)}  + C' t^2 \area_0 f_\delta . \]
By the properties of $f_\delta$ and the fact that the integrand in the previous integral is bounded by a multiple of $d{\vol}_{f^*_\delta (g_t)}$ independently of $\delta$, it follows that for fixed $t$ and for $\delta \to 0$ the right hand side of the previous inequality goes to
\[ A_0 ( \gamma ) + t \sum_i \int_{D^2} \frac{d}{dt} \Big|_{t = 0} d{\vol}_{f^*_i (g_t)} + C' t^2 A_0 ( \gamma ). \]
This yields the desired barrier.
\end{proof}

\subsection{The structure of a minimizer along the 1-skeleton} \label{subsec:1skeletonstruc}
Consider now again the $C^{1,1}$ regular map $f : V^{(1)} \to M$ from subsection \ref{subsec:regularityon1skeleton}.
The goal of this subsection is to derive a variational identity in the spirit of (\ref{eq:simplenuiskappa}).
However, due to possible self-intersections of $f$, this undertaking becomes a quite delicate issue and it will be important to analyze the combinatorics of these self-intersections.
Note that, at least a priori, there could be infinitely many such self-intersections and the set of self-intersections can have positive measure (and possibly empty interior).
Our main result will be Lemma \ref{Lem:almosteverywherealong1skeleton}.
In fact, inequality (\ref{eq:pointwisenukappaOnBoundary}) of this Lemma is the only conclusion that will be needed subsequently.
At this point we recall that by definition $f |_{\partial V} = f_0 |_{\partial V}$ is a smooth embedding.
So no edge at the boundary has a self-intersection and two edges only intersect in their endpoints.

We denote by $F_1, \ldots, F_n$ the faces and by $E_1, \ldots, E_m$ the edges of $V$ in such a way that $E_1, \ldots, E_{m_0}$ are the edges of $\partial V$.
For every $k = 1, \ldots, m$ let $l_k$ be the length of $f |_{E_k}$ and let $\gamma_k : [0, l_k] \to M$ be a parameterization of $f |_{E_k}$ by arclength.
Since the maps $\gamma_k$ have regularity $C^{1,1}$ (see Lemma \ref{Lem:regularityon1skeleton}), we can find for each $k = 1, \ldots, n$ a vector field $\kappa_k : [0, l_k] \to TM$ along $\gamma_k$ (i.e. $\kappa_k(s) \in T_{\gamma_k(s)}M$) which equals the geodesic curvature of $\gamma_k$ almost everywhere (see Lemma \ref{Lem:speedandcurvatureagreeonintersection}).

Next, we apply Proposition \ref{Prop:PlateauProblem} for each loop $f |_{\partial F_j}$ ($j=1, \ldots, n$) and obtain loops $\gamma_{j,1}, \gamma_{j,2}, \ldots$ which satisfy assertions (d)--(h) of this Proposition.
Without loss of generality, we may assume that each $\gamma_{j,i}$ is parameterized by arclength, i.e. that $\gamma_{j,i} : S^1(l_{j,i}) \to M$ where $l_{j,i}$ is the length of $\gamma_{j,i}$.
As before, we choose vector fields $\kappa_{j,i} : S^1(l_{j,i}) \to TM$ along each $\gamma_{j,i}$ which represent the geodesic curvature almost everywhere.
Now, let $f_{j,i} : D^2 \to M$ be an arbitrary solution to the Plateau problem for each loop $\gamma_{j,i}$.
Proposition \ref{Prop:PlateauProblem}(b) yields that $f_{j,i}$ is $C^{1,\alpha}$ up to the boundary except at the finitely many points where $\gamma_{j,i}$ is not differentiable.
So we can choose unit vector fields $\nu_{j,i} : S^1(l_{j,i}) \to TM$ along each $\gamma_{j,i}$ which are outward pointing tangential to $f_{j,i}$ everywhere except at finitely many points.

For each edge $E_k$ and each adjacent face $F_j$ we can consider the collection of subsegments of the $\gamma_{j,i}$ which lie on $E_k$.
These subsegments are pairwise disjoint and are equipped with the vector fields $\nu_{j,i}$.
We can hence construct a vector field along $\gamma_k$ which is equal to each of the $\nu_{j,i}$ on the corresponding subsegment and zero everywhere else.
Doing this for all faces $F_j$ which are adjacent to $E_k$ yields vector fields $\nu^{(1)}_k, \ldots, \nu^{(v_k)}_k : [0, l_k] \to TM$ along $\gamma_k$ where $v_k$ is the valency of $E_k$.
Note that $|\nu^{(u)}_k| \leq 1$ for all $k =1, \ldots, m$ and $u = 1, \ldots, v_k$.

With this notation at hand we can derive the following variation formula.

\begin{Lemma} \label{Lem:variationofV}
For every continuous vector field $X \in C^0(M; TM)$ which vanishes on $f (\partial V \cap V^{(0)})$ we have
\begin{multline*}
\Bigg| \sum_{k = 1}^m  \int_0^{l_k} \Big\langle \sum_{u = 1}^{v_k} \nu_k^{(u)} (s) , X_{\gamma_k(s)} \Big\rangle ds + \sum_{k=m_0+1}^m 
\bigg( - \int_0^{l_k} \big\langle \kappa_k (s), X_{\gamma_k(s)} \big\rangle ds \\
 -\big\langle \gamma'_k (0), X_{\gamma_k (0)} \big\rangle + \big\langle \gamma'_k (l_k), X_{\gamma_k (l_k)} \big\rangle \bigg) \Bigg| 
 \leq \sum_{k=1}^{m_0} \int_0^{l_k} \big| X_{\gamma_k (s)} \big| ds.
\end{multline*}
\end{Lemma}

\begin{proof}
Let first $X \in C^\infty (M; TM)$ be a smooth vector field and consider the smooth flow $\Phi : M \to \IR \to M$, $\partial_t \Phi_t = X \circ \Phi_t$ of $X$.
Observe that $\Phi_t (x) = x$ for all $x \in f (\partial V \cap V^{(0)})$ and $t \in \IR$.
For each $t \in \IR$ let $f'_t : V^{(1)} \to M$ be the map which is equal to $\Phi_t \circ f |_{V^{(1)} \setminus \partial V}$ on $V^{(1)} \setminus \partial V$ and equal to $f |_{\partial V}$ on $\partial V$.
By Lemma \ref{Lem:existenceon1skeleton} for all $t \in \IR$
\[ A (f'_t |_{\partial F_1} ) + \ldots + A (f'_t |_{\partial F_n} ) + \ell (f'_t) \geq A^{(1)}(f_0) \]
where equality holds for $t = 0$.
So we obtain that in the barrier sense
\begin{equation} \label{eq:variationisnonnegative}
 \frac{d}{dt^+} \Big|_{t = 0} \big( A (f'_t |_{\partial F_1} ) + \ldots + A (f'_t |_{\partial F_n} ) + \ell (f'_t) \big) \geq 0.
\end{equation}

Next we compute the derivative of each term on the left hand side.
First note that for all $k = m_0 + 1, \ldots, m$
\begin{equation} \label{eq:variationoflength}
 \frac{d}{dt^+} \Big|_{t = 0} \ell(\Phi_t \circ \gamma_k ) = - \int_0^{l_k} \big\langle \kappa_k (s), X_{\gamma_k(s)} \big\rangle ds - \big\langle \gamma'_k (0), X_{\gamma_k(0)} \big\rangle + \big\langle \gamma'_k (l_k), X_{\gamma_k(l_k)} \big\rangle.
\end{equation}
Next, we estimate the derivatives of the area terms.
To do this note that for each sufficiently differentiable map $h : D^2 \to M$ the area of $\Phi_t \circ h$ is equal to the area of $h$ with respect to the metric $\Phi^*_t (g)$.
So we can use Lemma \ref{Lem:variationofA} to deduce that for each $j =1, \ldots, n$
\[ \frac{d}{dt^+} \Big|_{t = 0}  A ( \Phi_t \circ f |_{\partial F_j} ) \leq \sum_i \int_{S^1(l_{j,i})} \big\langle \nu_{j,i} (s), X_{\gamma_{j,i} (s)} \big\rangle ds. \]
Now consider for each $k = 1, \ldots, m_0$ the loop which is composed of $\gamma_k$ and $\Phi_t \circ \gamma_k$ (recall that the endpoints of $\gamma_k$ are left invariant by $\Phi_t$).
This loop spans the disk which comes from the map $[0, l_k] \times [0, t] \to M$ with $(s,t') \mapsto \Phi_{t'} (\gamma_k(s))$.
The area of this disk is bounded by $\int_0^{l_k} |X_{\gamma_k(s)} | ds + O(t^2)$.
So
\begin{multline*}
 \frac{d}{dt^+} \Big|_{t = 0} \big( A (f'_t |_{\partial F_1} ) + \ldots + A (f'_t |_{\partial F_n} ) \big) \\ \leq \sum_{j=1}^m \sum_i \int_{S^1(l_{j,i})} \big\langle \nu_{j,i} (s), X_{\gamma_{j,i} (s)} \big\rangle ds  + \sum_{k=1}^{m_0}  \int_0^{l_k} \big| X_{\gamma_k(s)} \big| ds.
\end{multline*}
Together with (\ref{eq:variationisnonnegative}) and (\ref{eq:variationoflength}) this yields
\begin{multline*}
 \sum_{j=1}^m \sum_i \int_{S^1(l_{j,i})} \big\langle \nu_{j,i} (s), X_{\gamma_{j,i} (s)} \big\rangle ds + \sum_{k = m_0 + 1}^m \bigg( - \int_0^{l_k} \big\langle \kappa_k (s), X_{\gamma_k(s)} \big\rangle ds \\
 -\big\langle \gamma'_k (0), X_{\gamma_k (0)} \big\rangle + \big\langle \gamma'_k (l_k), X_{\gamma_k (l_k)} \big\rangle \bigg) + \sum_{k=1}^{m_0}  \int_0^{l_k} \big| X_{\gamma_k(s)} \big| ds \geq 0.
\end{multline*}
Note that by a simple rearrangement
\[  \sum_{j=1}^m \sum_i \int_{S^1(l_{j,i})} \big\langle \nu_{j,i} (s), X_{\gamma_{j,i} (s)} \big\rangle ds = \sum_{k = 1}^m  \int_0^{l_k} \Big\langle \sum_{u = 1}^{v_k} \nu_k^{(u)} (s) , X_{\gamma_k(s)} \Big\rangle ds. \]
So our conclusions applied for $X$ and $-X$ show that the desired inequality holds for all smooth vector fields which vanish on $f ( \partial V \cap V^{(0)} )$.
By continuity it must also hold for all \emph{continuous} vector fields which vanish on $f ( \partial V \cap V^{(0)} )$.
\end{proof}

We can now use this inequality to derive the following identities.

\begin{Lemma} \label{Lem:almosteverywherealong1skeleton}
For every $x \in f(V^{(0)} \setminus \partial V)$ the (normalized) directional derivatives of $f$ at every vertex of $V^{(0)}$, which is mapped to $x$, in the direction of each adjacent edge add up to zero.

Moreover, for every $k = 1, \ldots, m$ and for almost all $s \in [0, l_k]$ the following holds:
If $\gamma_k(s) \notin f(\partial V)$, then
\begin{equation} \label{eq:pointwisenukappa}
 \sum_{k' = 1}^m \sum_{\substack{s' \in E_{k'}  \\ f (s') = f(s)}}  \sum_{u = 1}^{v_{k'}} \nu_{k'}^{(u)} (s') - |f^{-1} (f(s)) | \cdot \kappa_{k} (s) = 0.
\end{equation}
Otherwise
\begin{equation} \label{eq:pointwisenukappaOnBoundary}
 \bigg| \sum_{k' = 1}^m \sum_{\substack{s' \in E_{k'}  \\ f (s') = f(s)}}  \sum_{u = 1}^{v_{k'}} \nu_{k'}^{(u)} (s') - \big( |f^{-1} (f(s)) | -1 \big) \cdot \kappa_{k} (s) \bigg| \leq 1.
\end{equation}
Moreover,
\begin{equation} \label{eq:nutimeskappageq0}
 \sum_{j=1}^n \sum_i \int_{S^1 (l_{j, i})} \big\langle \nu_{j, i}(s) , \kappa_{j,i}(s) \big\rangle ds \geq - \sum_{k = 1}^{m_0} \int_0^{l_k} \big| \kappa_k (s) \big| ds.
\end{equation}
\end{Lemma}

\begin{proof}
Since all $\kappa_k$ and $\nu_{j, i}$ are uniformly bounded it follows immediately from the variation formula in Lemma \ref{Lem:variationofV} that for every (not necessarily continuous) vector field $X$ on $M$ which vanishes on $f(\partial V \cap V^{(0)})$
\[ \sum_{k =1}^m \Big( -\big\langle \gamma'_k (0), X_{\gamma_k (0)} \big\rangle + \big\langle \gamma'_k (l_k), X_{\gamma_k (l_k)} \big\rangle \Big) = 0. \]
This implies the very first part of the claim and simplifies the variation formula:
For every continuous vector field $X \in C^0 (M; TM)$ we have
\begin{multline} \label{eq:variationwithoutvertices}
 \Bigg| \sum_{k = 1}^m  \int_0^{l_k} \Big\langle \sum_{u = 1}^{v_k} \nu_k^{(u)} (s) , X_{\gamma_k(s)} \Big\rangle ds - \sum_{k=m_0+1}^m 
 \int_0^{l_k} \big\langle \kappa_k (s), X_{\gamma_k(s)} \big\rangle ds 
 \Bigg| \\
 \leq \sum_{k=1}^{m_0} \int_0^{l_k} \big| X_{\gamma_k (s)} \big| ds.
\end{multline}

Choose $N < \infty$ large enough such that the following holds:
Each curve $\gamma_k$ restricted to a subinterval of length $\frac1N l_k$ is embedded and whenever two curves $\gamma_{k_1}, \gamma_{k_2}$ restricted to subintervals of length $\frac1{N} l_{k_1}, \frac1{N} l_{k_2}$ intersect, then we are in the situation of Lemma \ref{Lem:speedandcurvatureagreeonintersection}, i.e. we can find a coordinate chart $(U, (x_1, \ldots, x_n))$ which contains these subsegments and in which we can find a vector $v \in \IR^n$ with the property that $\langle \gamma'_{k_1}, v \rangle, \langle \gamma'_{k_2}, v \rangle \neq 0$ on both subsegments with respect to the euclidean metric.
Consider now the index set $I = \{ 1, \ldots, m \} \times \{ 0, \ldots, N-1 \}$ and define for every $(k, e) \in I$ and every subset $I' \subset I$ with $(k, e) \in I'$ the domain
\begin{multline*}
 \mathcal{D}_{k, e, I'} = \big\{ s \in [\tfrac{e}{N} l_k, \tfrac{e+1}{N} l_k] \;\; : \;\; \gamma_k (s) \in \gamma_{k'} \big( [\tfrac{e'}{N} l_{k'}, \tfrac{e'+1}{N} l_{k'}] \big) \\ \text{if and only if $(k',e') \in I'$} \big\}.
\end{multline*}
Clearly, these sets are measurable and for all $(k, e) \in I$
\[ \mathop{\dot\bigcup}_{\substack{I' \subset I \\ (k, e) \in I'}} \mathcal{D}_{k, e, I'} = [\tfrac{e}{N} l_k, \tfrac{e+1}{N} l_k]. \]
Moreover, since $f |_{\partial V} = f_0 |_{\partial V}$ is injective, we find that $\mathcal{D}_{k,e,I'}$ is empty or finite whenever there are two distinct pairs $(k', e'), (k'', e'') \in I'$ for which $k', k'' \leq m_0$.

Consider now two pairs $(k_1, e_1), \linebreak[1] (k_2, e_2)$ and a subset $I' \subset I$ such that $(k_1, e_1), \linebreak[1] (k_2, e_2) \in I'$ and assume that $\mathcal{D}_{k_1, e_1, I'}$ (and hence also $\mathcal{D}_{k_2, e_2, I'}$) is non-empty.
We can now apply the second part of Lemma \ref{Lem:speedandcurvatureagreeonintersection} and obtain a continuously differentiable map $\varphi : [\frac{e_1}N, \frac{e_1+1}N] \to \IR$, whose derivative vanishes nowhere, for which the following holds:
$\varphi (\mathcal{D}_{k_1, e_1, I'}) = \mathcal{D}_{k_2, e_2, I'}$ and $\gamma_{k_1} (s) = \gamma_{k_2} (\varphi(s))$ for all $s \in \mathcal{D}_{k_1, e_1, I'}$.
Moreover, for almost every $s \in \mathcal{D}_{k_1, e_1, I'}$ we have $\varphi'(s) = \pm 1$ and $\kappa_{k_1} (s) = \kappa_{k_2} (\varphi(s))$.
So the following three identities hold for every continuous vector field $X \in C^0(M; TM)$
\begin{alignat}{1} 
 \int_{\mathcal{D}_{k_1, e_1, I'}} \big\langle \kappa_{k_1} (s), X_{\gamma_{k_1} (s)} \big\rangle ds &= \int_{\mathcal{D}_{k_2, e_2, I'}} \big\langle \kappa_{k_2} (s), X_{\gamma_{k_2} (s)} \big\rangle ds, \label{eq:kappaintoverdifferentDD} \displaybreak[2] \\
 \int_{\mathcal{D}_{k_1, e_1, I'}} \Big\langle \sum_{u = 1}^{v_{k_2}} \nu_{k_2}^{(u)} (\varphi(s)), X_{\gamma_{k_1}(s)} \Big\rangle ds &= \int_{\mathcal{D}_{k_2, e_2, I'}} \Big\langle \sum_{u = 1}^{v_{k_2}} \nu_{k_2}^{(u)} (s), X_{\gamma_{k_2}(s)}  \Big\rangle ds, \label{eq:nuintoverdifferentDD} \displaybreak[2] \\ 
 \int_{\mathcal{D}_{k_1, e_1, I'}} \Big\langle \sum_{u=1}^{v_{k_2}} \nu_{k_2}^{(u)} (\varphi(s)), \kappa_{k_1} (s) \Big\rangle ds &= \int_{\mathcal{D}_{k_2, e_2, I'}} \Big\langle \sum_{u=1}^{v_{k_2}} \nu_{k_2}^{(u)} (s), \kappa_{k_2} (s) \Big\rangle ds.  \label{eq:nukappaproductondifferentDD}
\end{alignat}

Next we express both sides of (\ref{eq:variationwithoutvertices}) as sums of integrals over the domains $\mathcal{D}_{k, e, I'}$.
\begin{multline*}
 \Bigg| \sum_{I' \subset I} \bigg( \sum_{(k,e) \in I'} \int_{\mathcal{D}_{k, e, I'}} \Big\langle \sum_{u = 1}^{v_k} \nu_k^{(u)} (s) , X_{\gamma_k(s)} \Big\rangle ds - \sum_{\substack{(k,e) \in I' \\ k > m_0}} \int_{\mathcal{D}_{k, e, I'}} \Big\langle \kappa_k (s), X_{\gamma_k(s)} \Big\rangle ds  \bigg) \Bigg| \\
 \leq \sum_{I' \subset I} \sum_{\substack{(k,e) \in I' \\ k \leq m_0}} \int_{\mathcal{D}_{k,e,I'}} \big| X_{\gamma_k(s)} \big| ds.
\end{multline*}
We will now group integrals whose values are the same.
To do this set $I_0 = \{ 1, \ldots, m_0 \} \times \{ 0, \ldots, N-1 \}$ and for each $\emptyset \neq I' \subset I$ choose a pair $(k_{I'}, e_{I'}) \in I'$ such that $(k_{I'}, e_{I'}) \in I_0$ whenever $I' \cap I_0 \neq \emptyset$.
Using (\ref{eq:kappaintoverdifferentDD}) and (\ref{eq:nuintoverdifferentDD}) we may then express the integrals over the domains $\mathcal{D}_{k, e, I'}$ in the last inequality in terms of integrals over the domains $\mathcal{D}_{k_{I'}, e_{I'}, I'}$.
This yields
\begin{multline} \label{eq:variationsplitupinDD}
\Bigg| \sum_{\emptyset \neq I' \subset I} \int_{\mathcal{D}_{k_{I'} ,e_{I'} , I'}} \Big\langle \sum_{k =1}^m  \sum_{\substack{s' \in E_k  \\ f (s') = f(s)}} \sum_{u = 1}^{v_k} \nu_k^{(u)} (s') - |I' \cap (I \setminus I_0)| \cdot \kappa_{k_{I'}} (s), X_{\gamma_{k_{I'}}(s)} \Big\rangle ds \Bigg|  \\
 \leq \sum_{\substack{\emptyset \neq I' \subset I \\ I' \cap I_0 \neq \emptyset}} \int_{\mathcal{D}_{k_{I'} ,e_{I'} , I'}} \big| X_{\gamma_k(s)} \big| ds.
\end{multline}
Note that all summands involving $\emptyset \neq I' \subset I$ for which $I' \cap I_0$ contains more than one element vanish since those consist of integrals over a finite set.
So for all remaining summands and all $(k,e) \in I' \subset I$ for almost every $s \in \mathcal{D}_{k,e, I'}$ the quantity $|I' \cap (I \setminus I_0)|$ is equal to $|f^{-1} (f (s))|$ if $\gamma_k (s) \notin f(\partial V)$ (or equivalently if $I' \cap I_0 = \emptyset$) or equal to $|f^{-1} (f (s))| - 1$ if $\gamma_k (s) \in f(\partial V)$ (or equivalently if $|I' \cap I_0| = 1$).
So the first factor in the scalar product on the left hand side of (\ref{eq:variationsplitupinDD}) is equal to the left hand side of equation (\ref{eq:pointwisenukappa}) or (\ref{eq:pointwisenukappaOnBoundary}), depending on $I'$.

We will now show by induction on $|I'|$ that for every $\emptyset \neq I' \subset I$ equation (\ref{eq:pointwisenukappa}) or (\ref{eq:pointwisenukappaOnBoundary}) holds for almost every $s \in \mathcal{D}_{k_{I'}, n_{I'}, I'}$.
Using the previous conclusions which related $\mathcal{D}_{k, e, I'}$ to $\mathcal{D}_{k_{I'}, e_{I'}, I'}$ for any other $(k,e) \in I'$ this will then imply the desired statement.
So let $\emptyset \neq I' \subset I$ and assume that for all $\emptyset \neq I'' \subsetneq I'$ equation (\ref{eq:pointwisenukappa}) or (\ref{eq:pointwisenukappaOnBoundary}) holds for almost every $s \in \mathcal{D}_{k_{I''}, n_{I''}, I''}$.
This implies that the terms involving subsets $I''$ in the sums on both sides of the inequality (\ref{eq:variationsplitupinDD}) vanish whenever $\emptyset \neq I'' \subsetneq I$ and $I'' \cap I_0 = \emptyset$.

Consider now some $s_0 \in \mathcal{D}_{k_{I'}, e_{I'}, I'}$.
Then we can find an open neighborhood $U \subset M$ around $\gamma_{k_{I'} } (s_0)$ such that $\gamma_k([\frac{e}{N} l_k, \frac{e+1}{N} l_k]) \cap U \neq \emptyset$ if and only if $(k, e) \in I'$.
So as long as $X \in C^0 ( M ; TM)$ is supported in $U$, the summands in (\ref{eq:variationsplitupinDD}) involving $\emptyset \neq I'' \subset I$ with $\emptyset \neq I'' \not\subset I'$ vanish.
Then the only summands which are not a priori zero are the summand involving the subset $I'$ and all proper subsets $I'' \subsetneq I'$ for which $| I'' \cap I_0 | = 1$.

Consider first the case in which $I' \cap I_0 = \emptyset$.
Then the previous conclusion implies that only the summand involving $I'$ on the left hand side of (\ref{eq:variationsplitupinDD}) is not a priori zero and that the right hand side of this equation is zero.
So
\[ \int_{\mathcal{D}_{k_{I'} ,e_{I'} , l'}} \Big\langle \sum_{k =1}^m  \sum_{\substack{s' \in E_k  \\ f (s') = f(s)}} \sum_{u = 1}^{v_k} \nu_k^{(u)} (s') - |f^{-1} (f(s))| \cdot \kappa_{k_{I'}} (s), X_{\gamma_{k_{I'}}(s)} \Big\rangle ds = 0 \]
for all $X \in C^0(M; TM)$ which are supported in $U$.
Since $\gamma_{k_{I'}(s)}$ restricted to $[\frac{e_{I'}}N l_{k_{I'}}, \frac{e_{I'} + 1}N l_{k_{I'}}]$ is an embedding, this implies that
\[ \int_{\mathcal{D}_{k_{I'} ,e_{I'} , l'}} \Big\langle \sum_{k =1}^m  \sum_{\substack{s' \in E_k  \\ f (s') = f(s)}} \sum_{u = 1}^{v_k} \nu_k^{(u)} (s') - |f^{-1} (f(s))| \cdot \kappa_k (s), X (s) \Big\rangle ds = 0 \]
for every compactly supported continuous vector function $X \in C^0 ( \gamma_{k_{I'}}^{-1} (U) \cap \linebreak[1] [\frac{e_{I'}}N l_{k_{I'}}, \linebreak[1] \frac{e_{I'} + 1}N l_{k_{I'}}])$.
So (\ref{eq:pointwisenukappa}) holds almost everywhere on $\mathcal{D}_{k_{I'} ,e_{I'} , l'} \cap \gamma_{k_{I'}}^{-1} (U) \cap \linebreak[1] [\frac{e_{I'}}N l_{k_{I'}}, \linebreak[1] \frac{e_{I'} + 1}N l_{k_{I'}}]$.
Since $s_0$ was chosen arbitrarily, this shows that (\ref{eq:pointwisenukappa}) holds for almost every $s \in \mathcal{D}_{k_{I'} ,e_{I'} , l'}$ which finishes the induction in the first case.

Next consider the case in which $I' \cap I_0 = \{ (k_{I'}, e_{I'}) \}$.
Then for every non-zero summand in (\ref{eq:variationsplitupinDD}) involving $I''$ we have $(k_{I''}, e_{I''}) = (k_{I'}, e_{I'}) =: (k_0, e_0)$.
Since the union of all domains $\mathcal{D}_{k_0, e_0, I''}$ for which $(k_0, e_0) \in I'$ is equal to the interval $[\frac{e_0}N l_{k_0}, \frac{e_0 + 1}{N} l_{k_0}]$, inequality (\ref{eq:variationsplitupinDD}) implies that
\begin{multline*}
\Bigg| \int_{\frac{e_0}N l_{k_0}}^{\frac{e_0+1}N l_{k_0}}  \Big\langle \sum_{k =1}^m  \sum_{\substack{s' \in E_k  \\ f (s') = f(s)}} \sum_{u = 1}^{v_k} \nu_k^{(u)} (s') - \big( | f^{-1} (f(s)) \big| - 1\big) \cdot \kappa_{k_{I'}} (s), X_{\gamma_{k_{I'}}(s)} \Big\rangle ds \Bigg| \\
 \leq \int_{\frac{e_0}N l_{k_0}}^{\frac{e_0+1}N l_{k_0}} \big| X_{\gamma_k(s)} \big| ds
\end{multline*}
for all $X \in C^0(M; TM)$ which are supported in $U$.
As in the first case, we conclude that
\begin{multline*}
\Bigg| \int_{\frac{e_0}N l_{k_0}}^{\frac{e_0+1}N l_{k_0}}  \Big\langle \sum_{k =1}^m  \sum_{\substack{s' \in E_k  \\ f (s') = f(s)}} \sum_{u = 1}^{v_k} \nu_k^{(u)} (s') - \big( | f^{-1} (f(s)) \big| - 1\big) \cdot \kappa_{k_{I'}} (s), X (s) \Big\rangle ds \Bigg| \\
 \leq \int_{\frac{e_0}N l_{k_0}}^{\frac{e_0+1}N l_{k_0}} \big| X(s) \big| ds
\end{multline*}
for every compactly supported continuous vector function $X \in C^0 ( \gamma_{k_0}^{-1} (U) \cap \linebreak[1] [\frac{e_0}N l_{k_0}, \linebreak[1] \frac{e_0 + 1}N l_{k_0}])$.
This implies that (\ref{eq:pointwisenukappaOnBoundary}) holds for almost all $s \in \mathcal{D}_{k_0, e_0, I'} \subset [\frac{e_0}N l_{k_0}, \frac{e_0 + 1}{N} l_{k_0}]$ and finishes the induction in the second case.

Finally, we prove (\ref{eq:nutimeskappageq0}).
Observe that by a simple rearrangement we have
\[ \sum_{j=1}^n \sum_i \int_{S^1(l_{j, i})} \langle \nu_{j,i} (s), \kappa_{j,i} (s) \rangle ds = \sum_{k=1}^m \sum_{u=1}^{v_k} \int_0^{l_k} \big\langle \nu_k^{(u)} (s), \kappa_k (s) \big\rangle ds. \]
Using (\ref{eq:nukappaproductondifferentDD}) we conclude that
\begin{multline*}
\sum_{j=1}^n \sum_i \int_{S^1 (l_{j, i})} \big\langle \nu_{j, i}(s) , \kappa_{j,i}(s) \big\rangle ds = 
 \sum_{k=1}^m \sum_{u=1}^{v_k} \int_0^{l_k} \big\langle \nu_k^{(u)} (s), \kappa_k (s) \big\rangle ds \\ 
 = \sum_{I' \subset I} \sum_{(k,e) \in I'} \int_{\mathcal{D}_{k, e, I'}} \Big\langle \sum_{u = 1}^{v_k} \nu_k^{(u)} (s) , \kappa_k(s) \Big\rangle ds \\
 = \sum_{\emptyset \neq I' \subset I} \int_{\mathcal{D}_{k_{I'}, e_{I'}, I'}} \Big\langle \sum_{k = 1}^m \sum_{\substack{s' \in E_k  \\ f (s') = f(s)}} \sum_{u = 1}^{v_k} \nu_k^{(u)} (s'), \kappa_{k_{I'}} (s) \Big\rangle ds.
\end{multline*}
We now apply (\ref{eq:pointwisenukappa}) to all summands for which $I' \cap I_0 = \emptyset$ and (\ref{eq:pointwisenukappaOnBoundary}) to all summands for which $I' \cap I_0 \neq \emptyset$.
Then we obtain that the right hand side of the previous equation is bounded from below by
\begin{multline*}
 \sum_{\substack{\emptyset \neq I' \subset I \\ I' \cap I_0 = \emptyset}} \int_{\mathcal{D}_{k_{I'}, n_{I'}, I'}} |I'| \cdot \big\langle \kappa_{k_{I'}} (s), \kappa_{k_{I'}} (s) \big\rangle ds   \\
+ \sum_{\substack{\emptyset \neq I' \subset I \\ I' \cap I_0 \neq \emptyset}} \int_{\mathcal{D}_{k_{I'}, n_{I'}, I'}} \Big( (|I'|-1) \cdot 
\langle \kappa_{k_{I'}} (s), \kappa_{k_{I'}} (s) \big\rangle - \big| \kappa_{k_{I'}} (s) \big| \Big) ds \displaybreak[1] \\
\geq - \sum_{\substack{\emptyset \neq I' \subset I \\ I' \cap I_0 \neq \emptyset}} \int_{\mathcal{D}_{k_{I'}, n_{I'}, I'}} \big| \kappa_{k_{I'}} (s) \big| ds
= - \sum_{k = 1}^{m_0} \int_0^{l_k} \big| \kappa_k (s) \big| ds .
\end{multline*}
This establishes the claim.
\end{proof}

\subsection{Summary}
We conclude this section by summarizing the important results which are needed in section \ref{sec:areaevolutionunderRF}.

\begin{Proposition} \label{Prop:existenceofVminimizer}
Consider a compact Riemannian manifold $(M,g)$ with $\pi_2(M) \linebreak[0] = 0$.
Let $V$ be a finite simplicial complex whose faces are $F_1, \ldots, F_n$ and $f_0 : V \to M$ a continuous map such that $f_0 |_{\partial V}$ is a smooth embedding.
Furthermore, let $\gamma_k : [0, l_k] \to M$, $(k = 1, \ldots, m_0)$ be arclength parameterizations of $f$ restricted to the edges of $\partial V$ and $\kappa_k : [0, l_k] \to TM$ the geodesic curvature of $\gamma_k$

Then the following is true:
\begin{enumerate}[label=(\alph*)]
\item There is a map $f : V^{(1)} \to M$ which restricted to every edge $E \subset V^{(1)}$ is a $C^{1,1}$-immersion such that $f$ is homotopic to $f_0 |_{V^{(1)}}$ relative to $\partial V$ and
\[ A( f|_{\partial F_j} ) + \ldots + A( f|_{\partial F_n} ) + \ell (f)  = A^{(1)} (f_0). \]
\item Consider for each $j = 1, \ldots, n$ the loop $f |_{\partial F_j}$ and apply Proposition \ref{Prop:PlateauProblem} to obtain the loops $\gamma_{j,i} : S^1 (l_{j,i}) \to M$.
Let $p_{j,i}$ be the (finitely many) number of places where $\gamma_{j,i}$ is not differentiable.
Then
\[  \sum_i (p_{j,i} - 2) \leq  1. \]
\item For each loop $\gamma_{j, i}$ the geodesic curvature $\kappa_{j,i} : S^1 (l_{j,i}) \to TM$ is defined almost everywhere.
Let now $f_{j,i} : D^2 \to M$ be arbitrary solutions to the Plateau problem for $\gamma_{j,i}$ and let $\nu_{j,i} : S^1 (l_{j,i}) \to TM$ be unit vector fields along $\gamma_{j,i}$ which are outward pointing tangential to $f_{j,i}$.
Then
\[ \sum_{j=1}^n \sum_i \int_{S^1 (l_{j, i})} \big\langle \nu_{j, i}(s) , \kappa_{j,i}(s) \big\rangle ds \geq - \sum_{k = 1}^{m_0} \int_0^{l_k} \big| \kappa_k (s) \big| ds. \]
\end{enumerate}
\end{Proposition}

\begin{proof}
Note that $f |_{\partial F_j}$ is differentiable everywhere except possibly at its three corners.
So assertion (b) follows from Proposition \ref{Prop:PlateauProblem}(f).
Assertions (a), (c) are just restatements of facts proved earlier.
\end{proof}

\begin{Remark} \label{Rmk:Alambda}
For any $\lambda > 0$ consider the quantity
\[ A^{(\lambda)} (f_0) := \inf \big\{ \area (f') + \lambda \ell( f' |_{V^{(1)}} )  \;\; : \;\; f' \simeq f_0 \;\; \text{relative to $\partial V$} \big\}. \]
Then all assertions of Proposition \ref{Prop:existenceofVminimizer} hold with $A^{(1)}$ replaced by $A^{(\lambda)}$ (in assertion (a) we obviously have to insert the factor $\lambda$ in front of $\ell (f)$).
This follows simply by rescaling the metric $g$ by a factor of $\lambda$.
\end{Remark}

\section{Area evolution under Ricci flow} \label{sec:areaevolutionunderRF}
\subsection{Overview}
Let in this section $M$ be a closed $3$-manifold with $\pi_2 (M) = 0$.
Consider a finite simplicial complex $V$ whose faces are denoted by $F_1, \ldots, F_n$ and a continuous map $f_0 : V \to M$ such that $f_0 |_{\partial V}$ is a smooth embedding.

Consider a Ricci flow $(g_t)_{t \in [T_1, T_2]}$ on $M$ such that $\scal_t \geq - \frac{3}{2t}$ on $M$ for all $t \in [T_1, T_2]$.
The goal of this section is to study the evolution of the time dependent quantity
\[ A_t (f_0) : = \inf \big\{ \area_t f' \;\; : \;\; f' \simeq f_0 \;\; \text{relative to $\partial V$} \big\} \]
as introduced in section \ref{sec:existenceofvminimizers}.
We now explain our strategy in this section.
Assume first that for some time $t_0 \in [T_1, T_2]$ there is an embedded minimizer $f : V \to M$ in the homotopy class of $f_0$ (relative to $\partial V$), i.e. $\area_{t_0} f = A_{t_0} (f_0)$.
Then by a simple variational argument, we conclude that at every edge $E \subset V^{(1)} \setminus \partial V$ the unit vector fields $\nu_E^{(1)}, \ldots, \nu_E^{(v_E)}$ along $f |_E$, which are orthogonal to $f |_E$ and outward pointing tangential to the $v_E$ faces which are adjacent to $E$, satisfy the following identity
\begin{equation} \label{eq:nuadduptozero}
 \nu_E^{(1)} + \ldots + \nu_E^{(v_E)} = 0.
\end{equation}
We can then use Hamilton's method (see also \cite[Lemmas 7.2]{Bamler-longtime-II} and \cite{Ham}) to compute the time derivative of the area of the minimal disk $f |_{F_j}$ for every $j = 1, \ldots, n$
\begin{equation} \label{eq:derivativeonfaceillustration}
 \frac{d}{dt} \Big|_{t = t_0}  \area_t (f |_{F_j} ) \leq \frac3{4t_0} \area_{t_0} (f |_{F_j} ) + \pi  - \int_{\partial F_j} \big\langle \nu_{\partial F_j}, \kappa_{\partial F_j} \big\rangle.
\end{equation}
Here $\nu_{\partial F_j}$ is the unit vector field which is normal to $f |_{\partial F_j}$ and outward pointing tangential to $f |_{F_j}$ and $\kappa_{\partial F_j}$ is the geodesic curvature of $f |_{\partial F_j}$.
Now we add up these inequalities for $j = 1, \ldots, n$.
The sum of the integrals on the right hand side can be rearranged and grouped into integrals over each edge of $\partial V$.
By (\ref{eq:nuadduptozero}) the integrals over each edge $E \subset V^{(1)} \setminus \partial V$ cancel each other out and we are left with the integrals over edges $E \subset \partial V$.
So
\[ \frac{d}{dt} \Big|_{t = t_0}  \area_t f \leq \frac3{4t_0} \area_{t_0} (f |_{F_j} ) + \pi n  + \sum_{E \subset \partial V} \int_E |\kappa_E |. \]
This implies that in the barrier sense
\begin{equation} \label{eq:wantthisbarrierbound}
 \frac{d}{dt^+} \Big|_{t = t_0}  A_t (f_0) \leq \frac3{4t_0} \area_{t_0} (f |_{F_j} ) + \pi n   + \sum_{E \subset \partial V} \int_E |\kappa_{E} |.
\end{equation}

Unfortunately, as mentioned in section \ref{sec:existenceofvminimizers}, an existence theory for such a minimizer $f$ is hard to come by.
We will however be able to establish the bound (\ref{eq:wantthisbarrierbound}) without the knowledge of this existence using the following trick.
For every $\lambda > 0$ consider the quantity
\[ A_t^{(\lambda)} (f_0) := \inf \big\{ \area_t (f') + \lambda \ell_t ( f' |_{V^{(1)}} )  \;\; : \;\; f' \simeq f_0 \;\; \text{relative to $\partial V$} \big\} \]
as introduced in Remark \ref{Rmk:Alambda}.
It is not hard to see that for each $t \in [T_1, T_2]$
\begin{equation} \label{eq:AlambdaandA0}
 A_t^{(\lambda)} (f_0) \geq A_t (f_0) \qquad \text{and} \qquad \lim_{\lambda \to 0} A_t^{(\lambda)} (f_0) = A_t (f_0).
\end{equation}
Now the existence theory for a minimizer of $A_t^{(\lambda)}(f_0)$ becomes far easier and has been carried out in section \ref{sec:existenceofvminimizers}.
Assume for the purpose of clarity that for some time $t_0$ there is an embedded, smooth minimizer $f : V \to M$ for the corresponding minimization problem, i.e. $\area_{t_0} f + \lambda \ell_{t_0} (f |_{V^{(1)}}) = A^{(\lambda)}_{t_0} (f_0)$.
Then identity (\ref{eq:nuadduptozero}) becomes (compare with (\ref{eq:simplenuiskappa}))
\[ \nu^{(1)}_E +  \ldots + \nu^{(v_E)}_E = \lambda \kappa_E. \]
So when adding up inequality (\ref{eq:derivativeonfaceillustration}) for all $j =1, \ldots, n$ and grouping the integrals on the right hand side by edge, we find that luckily the extra term which arises due to this modified identity has the right sign:
\begin{multline*}
 \frac{d}{dt} \Big|_{t = t_0}  \area_t f \leq \frac3{4t_0} \area_{t_0} (f |_{F_j} ) + \pi n  
  + \sum_{E \subset \partial V} \int_E |\kappa_E | 
 - \sum_{E \subset V^{(1)} \setminus \partial V} \int_E \big\langle \lambda \kappa_E, \kappa_E \big\rangle \\
\leq  \frac3{4t_0} \area_{t_0} (f |_{F_j} ) + \pi n 
  + \sum_{E \subset \partial V} \int_E |\kappa_E |.
\end{multline*}
Now choose a function $\lambda : [T_1, T_2] \to (0,1)$ such that $\lambda' (t) < - K_t \lambda(t)$ where $K_t$ is a bound on the Ricci curvature at time $t$.
This is always possible.
Then we can check that
\[ \frac{d}{dt} \Big|_{t = t_0} A^{(\lambda(t))}_t (f_0) \leq \frac3{4t_0} \area_{t_0} (f |_{F_j} ) + \pi n - \sum_{E \subset \partial V} \int_E |\kappa_{\partial F_j} |. \]
Since $\lambda(t)$ can be chosen arbitrarily small, we are able to conclude (\ref{eq:wantthisbarrierbound}) using (\ref{eq:AlambdaandA0}).

Note that this is a simplified picture of the arguments that will be presented in the next subsection.
The main difficulty that needs to be overcome stems from the fact that $f : V \to M$ is in general only defined on the $1$-skeleton and not smooth there and that $f$ might have self-intersections.

\subsection{Main part}
In the following Lemma we deduce an important bound on a curvature integral over a minimal disk with smooth boundary.
The statement and its proof are similar to parts of \cite[Lemmas 7.1, 7.2]{Bamler-longtime-II}.

\begin{Lemma} \label{Lem:integralboundonsmoothmindisk}
Let $f : D^2 \to M$ be a smooth, harmonic, almost conformal map and set $\gamma = f_{\partial D^2}$.
Denote by $\kappa : S^1 = \partial D^2 \to TM$ the geodesic curvature of $\gamma$ and by $\nu : S^1 \to TM$ the unit vector field along $\gamma$ which is orthogonal to $\gamma$ and outward pointing tangential to $f$ away from possible branch points.
Then
\[ \int_{D^2} \sec^M (df) d{\vol}_{f^*(g)} \geq 2\pi + \int_{S^1} \big\langle \nu(s), \kappa(s) \big\rangle \cdot |\gamma'(s)| ds. \]
Here $\sec^M(df)$ denotes the sectional curvature of $M$ in the direction of the image of $df$.
Note that the integrand on the left hand side is well-defined since the volume form vanishes whenever $df$ is not injective.
\end{Lemma}

\begin{proof}
In order to avoid issues arising from possible branch points (especially on the boundary of $\Sigma$), we employ the following trick (compare with \cite{PerelmanIII}):
Denote by $g_{\textnormal{eucl}}$ the euclidean metric on $D^2$ and consider for every $\varepsilon > 0$ the Riemannian manifold $(D_\varepsilon = D^2, \varepsilon g_{\textnormal{eucl}})$.
The identity map $h_\varepsilon : D^2 \to (D^2, \varepsilon g_{\textnormal{eucl}})$ is trivially a harmonic and conformal diffeomorphism and hence the map $f_\varepsilon = (f, h_\varepsilon) : D^2 \to M \times D_\varepsilon$ is a harmonic and conformal \emph{embedding}.
Denote its image by $\Sigma_\varepsilon = f_\varepsilon(D^2) \subset M \times D_\varepsilon$.
Since the sectional curvatures on the target manifold are bounded, we have
\[ \lim_{\varepsilon \to 0} \int_{\Sigma_\varepsilon} \sec^{M \times D_\varepsilon} ( T \Sigma_\varepsilon) d {\vol} =  \int_{\Sigma}  \sec^M ( df) d {\vol}_{f^*(g)}, \]
where $d{\vol}$ on the left hand side denotes the induced volume form and the integrand denotes the function on $\Sigma_\varepsilon$ which assigns to each point the (ambient) sectional curvature of $M \times D_\varepsilon$ in the direction of its tangent space.

Since $\Sigma_\varepsilon$ is a minimal surface, its interior sectional curvatures are not larger than the corresponding ambient ones.
So we obtain together with Gau\ss-Bonnet
\[ \int_{\Sigma_\varepsilon} \sec^{M \times D_\varepsilon} ( T \Sigma_\varepsilon) d {\vol} \geq \int_{\Sigma_\varepsilon} \sec^{\Sigma_\varepsilon} d{\vol} = 2\pi + \int_{\partial \Sigma_\varepsilon} \kappa^{\Sigma_\varepsilon}_{\partial \Sigma_\varepsilon} d s. \]
Here $\kappa^{\Sigma_\varepsilon}_{\partial \Sigma_\varepsilon}$ denotes the geodesic curvature of the boundary circle viewed as a curve within $\Sigma_\varepsilon$.
We now estimate the last integral.
Let $\gamma_\varepsilon = (\gamma_\varepsilon^M, \gamma_\varepsilon^{D_\varepsilon}) : S^1(l_\varepsilon) \to M \times D_\varepsilon$ be an arclength parameterization in such a way that $\gamma_\varepsilon^M$ converges to a unit speed parameterization $\gamma_0 : S^1(l) \to M$ of $\gamma$ as $\varepsilon \to 0$.
Furthermore, let $\nu_\varepsilon = (\nu_\varepsilon^M, \nu_\varepsilon^{D_\varepsilon}): S^1(l_\varepsilon) \to T(M \times D_\varepsilon)$ be the unit normal field along $\gamma_\varepsilon$ which is outward pointing tangent to $\Sigma_\varepsilon$.

It is not difficult to see that due to conformality,
\[ \nu_\varepsilon^M (s) = |(\gamma_\varepsilon^M)'(s)| \cdot \nu (\gamma_\varepsilon^{D_\varepsilon}(s))  \qquad \text{and} \qquad |\nu_\varepsilon^{D_\varepsilon} (s)| = |(\gamma_\varepsilon^{D_\varepsilon})'(s)| . \]
We can compute
\begin{multline*}
 \int_{\partial \Sigma_\varepsilon} \kappa_{\partial \Sigma_\varepsilon}^{\Sigma_\varepsilon} d s =   \int_{S^1(l_\varepsilon)} \Big\langle \nu_\varepsilon (s), \frac{D}{ds} \frac{d}{ds}  \gamma_\varepsilon (s) \Big\rangle ds \\
 = \int_{S^1(l_\varepsilon)} \Big\langle \nu^M_\varepsilon(s), \frac{D}{ds}  \frac{d}{ds}  \gamma^M_\varepsilon (s)  \Big\rangle ds 
+ \int_{S^1(l_\varepsilon)} \Big\langle  \nu^{D_\varepsilon}_\varepsilon(s), \frac{D}{ds}  \frac{d}{ds}  \gamma^{D_\varepsilon}_\varepsilon (s) \Big\rangle ds.
\end{multline*}
It is not difficult to see that the first integral on the right hand side is converges to $\int_{S^1} \big\langle \nu(s), \kappa(s) \big\rangle \cdot |\gamma'(s)| ds$ as $\varepsilon \to 0$.
The absolute value of the second integral is bounded by
\[  \int_{S^1(l_\varepsilon)} |(\gamma_\varepsilon^{D_\varepsilon})'(s)| \cdot \Big| \frac{D}{ds}  \frac{d}{ds}  \gamma^{D_\varepsilon}_\varepsilon (s) \Big| ds, \]
which goes to zero as $\varepsilon \to 0$.
This implies the claim.
\end{proof}

Next, we extend the bound of Lemma \ref{Lem:integralboundonsmoothmindisk} to minimal disks which are bounded by not necessarily embedded, piecewise $C^{1,1}$ loops which satisfy the Douglas-type condition.

\begin{Lemma} \label{Lem:integralcurvboundnotsmooth}
Let $\gamma : S^1 \to M$ be a continuous loop which is a piecewise $C^{1,1}$-immersion and let $\theta_1, \ldots, \theta_p$ be the angles between the right and left derivative of $\gamma$ at the points where $\gamma$ is not differentiable.
(Observe that $\theta_i = 0$ means that both derivatives agree).
Assume that $\gamma$ satisfies the Douglas-type condition (see Definition \ref{Def:DouglasCondition}).
Then there is a solution to the Plateau problem $f : D^2 \to M$ for $\gamma$ which has the following property:

The function $f$ is $C^{1,\alpha}$ up to the boundary away from finitely many points.
Let $\nu : S^1 \to TM$ be the unit vector field along $\gamma$ which is orthogonal to $\gamma$ and outward pointing tangential to $f$ away from possibly finitely many points and let $\kappa : S^1 \to TM$ be almost everywhere equal to the the geodesic curvature of $\gamma$.
Then
\[ \int_{D^2} \sec^M (df) d{\vol}_{f^*(g)} \geq 2\pi  - \theta_1 - \ldots - \theta_p + \int_{S^1} \big\langle \nu(s), \kappa(s) \big\rangle \cdot |\gamma'(s)| ds. \]
\end{Lemma}

\begin{proof}
The proof uses an approximation method.

Let $s_1, \ldots, s_p \in S^1$ be the places where $\gamma$ is not differentiable and choose a small constant $\varepsilon > 0$.
It is not difficult to see that there is a function $\phi : (0,1) \to (0,1)$ with $\lim_{t \to 0} \phi(t) = 0$ (which may depend on $(M,g)$ and $\gamma$) such that:
We can replace $\gamma$ in a small neighborhood of each $s_i$ by a small arc of length $\leq ( \theta_i + \phi(\varepsilon)) \varepsilon$ and geodesic curvature bounded by $\varepsilon^{-1}$ such that the resulting curve $\gamma^* : S^1 \to M$ is a $C^1$-immersion.
It then follows that if $\kappa^* : S^1 \to M$ is almost everywhere equal to the geodesic curvature of $\gamma^*$, we have
\[ \int_{S^1} \big| \kappa^*(s) - \kappa (s) \big| \cdot |\gamma^{* \prime} (s)| ds \leq \theta_1 + \ldots + \theta_p + p \phi(\varepsilon) + p C \varepsilon. \]
Here $C$ is a $C^{1,1}$ bound on $\gamma$.
Next, we mollify $\gamma^*$ to obtain a smooth immersion $\gamma^{**} : S^1 \to M$ such that if $\kappa^{**} : S^1 \to M$ is the geodesic curvature of $\gamma^{**}$, we have
\[ \int_{S^1} \big| \kappa^{**}(s) - \kappa (s) \big| \cdot |\gamma^{** \prime} (s)| ds \leq \theta_1 + \ldots + \theta_p + p \phi(\varepsilon) + pC \varepsilon + \varepsilon. \]
Finally, we can perturb $\gamma^{**}$ to a smooth embedding $\gamma^{***} : S^1 \to M$ whose geodesic curvature $\kappa^{***} : S^1 \to M$ satisfies
\[ \int_{S^1} \big| \kappa^{***}(s) - \kappa (s) \big| \cdot |\gamma^{*** \prime} (s)| ds \leq \theta_1 + \ldots + \theta_p + p \phi(\varepsilon) + pC \varepsilon + 2\varepsilon. \]

These constructions have shown that can find a sequence $\gamma_1, \gamma_2, \ldots : S^1 \to M$ of smoothly embedded loops which uniformly converge to $\gamma$ and which locally converge on $S^1 \setminus \{ s_1, \ldots, s_q \}$ to $\gamma$ in the $C^{1,\alpha}$ sense such that the geodesic curvatures $\kappa_k : S^1 \to TM$ satisfy
\begin{equation} \label{eq:limsupboundedbythetas}
 \limsup_{k \to \infty} \int_{S^1} \big| \kappa_k (s) - \kappa (s) \big| \cdot |\gamma'_k(s) | ds \leq \theta_1 + \ldots + \theta_q . \end{equation}
Let now $f_1, f_2, \ldots : D^2 \to M$ be solutions of the Plateau problem for these loops.
By Proposition \ref{Prop:PlateauProblem}(b) the maps $f_k$ are smooth up to the boundary.
Moreover, by Proposition \ref{Prop:PlateauProblem}(c) we conclude that, after passing to a subsequence and a possible conformal reparameterization, the maps $f_k : D^2 \to M$ converge uniformly on $D^2$ and smoothly on $\Int D^2$ to a map $f : D^2 \to M$ which solves the Plateau problem for $\gamma$.
By Proposition \ref{Prop:PlateauProblem}(b) the map $f$ has local regularity $C^{1, \alpha}$ up to the boundary away from finitely many points for all $\alpha < 1$.
So by Proposition \ref{Prop:PlateauProblem}(c), the convergence $f_k \to f$ is locally in $C^{1, \alpha}$ away from finitely many points.

We now conclude first that
\begin{equation} \label{eq:interiorcurvboundconvergence}
 \lim_{k \to \infty} \int_{D^2} \sec^M (df_k) d{\vol}_{f^*_k (g)} = \int_{D^2} \sec^M (df) d{\vol}_{f^*(g)}.
\end{equation}
Moreover, if we denote by $\nu_k : S^1 \to M$ the unit normal vectors to $\gamma_k$ which are outward tangential to $f_k$, we obtain that (recall that $|\gamma'_k(s)|$ and $\kappa (s)$ are uniformly bounded).
\begin{equation} \label{eq:limitwithnukbutkappa}
 \lim_{k \to \infty} \int_{S^1} \big\langle \nu_k (s), \kappa(s) \big\rangle \cdot |\gamma'_k(s)| ds= \int_{S^1} \big\langle \nu (s), \kappa(s) \big\rangle \cdot |\gamma'(s)| ds.
\end{equation}
Now note that
\begin{multline*}
 \Big| \int_{S^1} \big\langle \nu_k (s), \kappa_k (s) \big\rangle \cdot |\gamma'_k(s)| ds - \int_{S^1} \big\langle \nu_k (s), \kappa (s) \big\rangle \cdot |\gamma'_k (s)| ds \Big| \\
 \leq \int_{S^1} \big| \kappa_k (s) - \kappa (s) \big| \cdot |\gamma'_k(s) |.
\end{multline*}
Together with (\ref{eq:limsupboundedbythetas}) and (\ref{eq:limitwithnukbutkappa}) this implies 
\[ \liminf_{k \to \infty} \int_{S^1} \big\langle \nu_k (s), \kappa_k (s) \big\rangle \cdot |\gamma'_k(s)| ds \geq  - \theta_1 - \ldots - \theta_q + \int_{S^1} \big\langle \nu (s), \kappa(s) \big\rangle \cdot |\gamma'(s)| ds. \]

Finally, applying Lemma \ref{Lem:integralboundonsmoothmindisk} for each $f_k$ we obtain together with (\ref{eq:interiorcurvboundconvergence}) and the previous estimate that
\begin{multline*}
 \int_{D^2} \sec^M (df) d{\vol}_{f^*(g)} = \lim_{k \to \infty} \int_{D^2} \sec^M (df_k) d{\vol}_{f^*_k (g)} \\
 \geq 2\pi + \liminf_{k \to \infty} \int_{S^1} \big\langle \nu_k (s), \kappa_k (s) \big\rangle \cdot |\gamma'_k(s)| ds \\
 \geq 2\pi - \theta_1 - \ldots - \theta_p + \int_{S^1} \big\langle \nu (s), \kappa(s) \big\rangle \cdot |\gamma'(s)| ds. \qedhere
\end{multline*}
\end{proof}

We can now apply the previous bound together with the results of Proposition \ref{Prop:existenceofVminimizer} and to control the time derivative of the quantity $A_t^{(\lambda)}$.
We remark that the proof of this Lemma is again similar to parts of \cite[Lemmas 7.1, 7.2]{Bamler-longtime-II}.

\begin{Lemma} \label{Lem:evolutionofAlambdaunderRF}
Let $0 < T_1 < T_2 < \infty$ and $(g_t)_{t \in [T_1,T_2]}$ be a smooth solution of the Ricci flow on $M$ on which $\scal_t \geq - \frac{3}{2t}$ for all $t \in [T_1, T_2]$.
Assume that the Ricci curvature of $g_t$ is bounded by some constant $K < \infty$ for all $t \in [T_1, T_2]$.

Let moreover $V$ be a finite simplicial complex whose faces are denoted by $F_1, \ldots, F_n$ and $f_0 : V \to M$ a continuous map such that $f_0 |_{\partial V}$ is a smooth embedding.
At every time $t \in [T_1, T_2]$ let $\gamma_{k, t} : [0, l_{k, t}] \to M$, $(k = 1, \ldots, m_0$) be time-$t$ arclength parameterizations of $f$ restricted to the edges of $\partial V$ and $\kappa_{k, t} : [0, l_{k, t}] \to TM$ the geodesic curvature of each $\gamma_{k, t}$ at time $t$.

Now let $\lambda : [T_1, T_2] \to (0, \infty)$ be a continuously differentiable function such that $\lambda'(t) \leq - K \lambda(t)$ for all $t \in [T_1, T_2]$.
Then we can bound the evolution of the quantity $A^{(\lambda(t))}_t (f_0)$ as follows.
For every $t \in [T_1, T_2)$ we have in the barrier sense:
\[ \frac{d}{dt^+} A^{(\lambda(t))}_t (f_0) \leq \frac{3}{4t} A^{(\lambda(t))}_t (f_0) +  \pi n + \sum_{k=1}^{m_0} \int_0^{l_{k, t}} \big| \kappa_{k, t} (s) \big|_t ds . \]
\end{Lemma}

\begin{proof}
Let $t_0 \in [T_1, T_2]$.
We first apply Proposition \ref{Prop:existenceofVminimizer} (see also Remark \ref{Rmk:Alambda}) at time $t_0$ and obtain a $C^{1,1}$-map $f : V^{(1)} \to M$ which is homotopic to $f_0 |_{V^{(1)}}$ relative $\partial V$ and for which
\[  \sum_{j=1}^n A_{t_0} (f |_{\partial F_j} ) + \lambda(t_0) \ell_{t_0} (f) = A^{(\lambda(t_0))}_{t_0} (f_0). \]
Consider for each $j = 1, \ldots, n$ the loop $f|_{\partial F_j}$ and apply Proposition \ref{Prop:PlateauProblem} to obtain the loops $\gamma_{j,i} : S^1 (l_{j,i}) \to M$.
As in Proposition \ref{Prop:existenceofVminimizer}(b) let $p_{j,i}$ be the number of places where $\gamma_{j,i}$ is not differentiable and let $\kappa_{j,i} : S^1(l_{j,i}) \to TM$ be the geodesic curvature along $\gamma_{j,i}$.
Recall that each $\gamma_{j,i}$ satisfies the Douglas-type condition and that for each $j =1, \ldots, n$
\[ \sum_i A_{t_0} (\gamma_{j,i} ) = A_{t_0} (\gamma_j) \qquad \text{and} \qquad \sum_i (p_{j,i} - 2) \leq 1. \]

Next, we apply Lemma \ref{Lem:integralcurvboundnotsmooth} at time $t_0$ to obtain a solution to the Plateau problem $f_{j,i} : D^2 \to M$ for each $\gamma_{j,i}$ such that for the unit normal vector field $\nu_{j,i} : S^1(l_{j,i}) \to TM$ which is outward pointing tangential to $f_{j,i}$ we have
\[ \int_{D^2} \sec_{t_0}^M (df_{j,i}) d{\vol}_{f_{j,i}^*(g_{t_0})} \geq \pi (2 - p_{j,i}) + \int_{S^1(l_{j,i})} \big\langle \nu_{j,i} (s), \kappa_{j,i} (s) \big\rangle_{t_0}  ds. \]

We can now apply Lemma \ref{Lem:variationofA} and conclude that in the barrier sense for all for all $j = 1, \ldots, n$
\begin{multline*}
 \frac{d}{dt^+} \Big|_{t = t_0} A_t ( f |_{\partial F_j} ) \leq \sum_i \int_{D^2} \frac{d}{dt} \Big|_{t = t_0} d{\vol}_{f^*_{j,i} (g_t)} \\
 = - \sum_i \int_{D^2} \tr_{f^*_{j,i} (g_{t_0})} (\Ric_{t_0} (df_{j,i}, df_{j,i})) d{\vol}_{f^*_{j,i} (g_{t_0})} \displaybreak[1] \\
  =  - \sum_i \bigg( \frac12 \int_{D^2} \big( \scal_{t_0} \circ f_{j,i} \big) d{\vol}_{f^*_{j,i} (g_{t_0}) } + \int_{D_2} \sec^M_{t_0} (df_{j,i}) d{\vol}_{f^*_{j,i} (g_{t_0}) } \bigg) \\
 \leq \frac{3}{4t_0} \sum_i A_{t_0} (\gamma_{j,i}) + \sum_i \pi (p_{j,i} - 2) - \sum_i \int_{S^1(l_{j,i})} \big\langle \nu_{j,i} (s), \kappa_{j,i}(s) \big\rangle ds \\
 = \frac{3}{4t_0} A_{t_0} (\gamma_j) + \pi  - \sum_i \int_{S^1(l_{j,i})} \big\langle \nu_{j,i} (s), \kappa_{j,i}(s) \big\rangle ds.
\end{multline*}
Now Proposition \ref{Prop:existenceofVminimizer}(c) implies that if we sum this inequality over all $j = 1, \ldots, n$, then the integral term can be estimated by a boundary integral:
\[ \frac{d}{dt^+} \Big|_{t = t_0} \sum_{j = 1}^n A_t ( f|_{\partial F_j} ) \leq \frac{3}{4t_0} \sum_{j=1}^n A_{t_0} (\gamma_j) + \pi n + \sum_{k = 1}^{m_0} \int_0^{l_{k, t_0}} \big| \kappa_{k, t_0} (s) \big|_{t_0} ds. \]

It remains to estimate the distortion of the length of $f$.
Since the Ricci curvature is bounded by $K$ on $[T_1, T_2]$, we find
\[ \frac{d}{dt}  \Big|_{t = t_0} \big( \lambda(t) \ell_t (f) \big) \leq - K \lambda(t_0) \ell_{t_0} (f) + \lambda(t_0) \cdot K \ell_{t_0} (f) \leq 0. \]
Finally, observe that for all $t \geq t_0$ we have by Lemma \ref{Lem:existenceon1skeleton}
\[ A^{(\lambda(t))}_t (f_0) \leq \sum_{j = 1}^n A_t ( f|_{\partial F_j} ) + \lambda(t) \ell_t (f). \]
The equality is strict for $t = t_0$ and the time derivative of the right hand side is bounded by exactly the desired term in the barrier sense.
This finishes the proof of the Lemma.
\end{proof}

Letting the parameter $\lambda$ go to zero yields the following estimate which does not require a global curvature bound.

\begin{Lemma} \label{Lem:evolofAunderRF}
Let $0 < T_1 < T_2 \leq \infty$ and $(g_t)_{t \in [T_1,T_2)}$ be a smooth solution of the Ricci flow on $M$ on which $\scal_t \geq - \frac{3}{2t}$ for all $t \in [T_1, T_2]$.

Let moreover $V$ be a finite simplicial complex whose faces are denoted by $F_1, \ldots, F_n$ and $f_0 : V \to M$ a continuous map such that $f_0 |_{\partial V}$ is a smooth immersion.
At every time $t \in [T_1, T_2)$ let $\gamma_{k, t} : [0, l_{k, t}] \to M$, $(k = 1, \ldots, m_0$) be time-$t$ arclength parameterizations of $f_0$ restricted to the edges of $\partial V$ and $\kappa_k : [0, l_{k, t}] \to TM$ the geodesic curvature of each $\gamma_{k, t}$ at time $t$.

Then we can bound the evolution of $A_t (f_0)$ as follows in the barrier sense:
\[ \frac{d}{dt^+} A_t (f_0) \leq \frac{3}{4t} A_t (f_0) +  \pi n + \sum_{k=1}^{m_0} \int_0^{l_{k, t}} \big| \kappa_{k, t} (s) \big|_t ds . \]
\end{Lemma}

\begin{proof}
Note that by a perturbation argument we only need to consider the case in which $f_0 |_{\partial V}$ is an embedding.
Moreover, we can without loss of generality restrict to a time-interval on which the Ricci curvature is bounded by some constant $K < \infty$.
For brevity set
\[ R_t = \pi n + \sum_{k=1}^{m_0} \int_0^{l_{k, t}} \big| \kappa_{k, t} (s) \big|_t ds. \]
Note that $R_t$ is continuous with respect to $t$.
Let $\varepsilon > 0$ be a small constant and apply Lemma \ref{Lem:evolutionofAlambdaunderRF} with $\lambda(t) = \varepsilon \exp(-Kt)$.
We obtain
\[ \frac{d}{dt^+} A^{(\varepsilon \exp(-Kt))}_t (f_0) \leq \frac{3}{4t} A^{(\varepsilon \exp(-Kt))}_t (f_0) + R_t. \]

Let now $t_0 \in [T_1, T_2)$ and consider the solution of the differential equation
\[ \frac{d}{dt} F_{t_0, \varepsilon} (t) = \frac{3}{4t} F_{t_0, \varepsilon} (t) + R_t \qquad \text{and} \qquad F_{t_0, \varepsilon} (t_0) = A_{t_0}^{(\varepsilon \exp (-Kt_0))} (f_0). \]
It follows that
\[  A_t^{(\varepsilon \exp(-K t))} (f_0) \leq F_{t_0, \varepsilon} (t) \qquad \text{for all} \qquad t \geq t_0. \]
Letting $\varepsilon \to 0$ and using the fact that that $\lim_{\lambda \to 0} A^{(\lambda)}_t (f_0) = A_t(f_0)$ yields
\[ A_t (f_0) \leq F_{t_0, 0} (t) \qquad \text{for all} \qquad t \geq t_0 \]
where $F_{t_0, 0}$ satisfies the differential equation
\[ \frac{d}{dt} F_{t_0, 0} (t) = \frac{3}{4t} F_{t_0, 0} (t) + R_t \qquad \text{and} \qquad F_{t_0, 0} (t_0) = A_{t_0} (f_0). \]
So $F_{t_0, 0} (t)$ is a barrier for $A_t (f_0)$ with the required properties.
\end{proof}

We can finally summarize our findings.
Note that the following Proposition is in some way a generalization of \cite[Lemma 7.2]{Bamler-longtime-II} for simplicial complexes.

\begin{Proposition} \label{Prop:areaevolutioninMM}
Let $\MM$ be a Ricci flow with surgery with precise cutoff defined on a time-interval $[T_1, T_2)$ (where $0 < T_1 < T_2 \leq \infty)$ and assume that $\pi_2 (\MM(t)) = 0$ for all $t \in [T_1, T_2)$.
Consider a finite simplicial complex $V$ whose faces are denoted by $F_1, \ldots, F_n$.

Let $f_0 : V \to \MM (T_1)$ be a continuous map such that $f_{0,0} = f_0 |_{\partial V}$ is a smooth immersion.
Consider a smooth family of immersions $f_{0, t} : \partial V \to \MM(t)$ parameterized by time which extend $f_{0,0}$ and which don't meet any surgery points.
Assume moreover that there is a constant $\Gamma < \infty$ such that for each $t \in [T_1, T_2)$ the following is true:
Let $\gamma_{k, t} : [0, l_{k, t}] \to \MM(t)$, $(k = 1, \ldots, m_0$) be time-$t$ arclength parameterizations of $f_{0,t}$ restricted to the edges of $\partial V$ and $\kappa_k : [0, l_{k, t}] \to T\MM(t)$ the geodesic curvature of each $\gamma_{k, t}$ at time $t$.
Then
\[ \sum_{k=1}^{m_0} \int_0^{l_{k, t}} \big( \big| \kappa_{k, t} (s) \big|_t + \big| \partial_t \gamma_{k,t}^\perp (s) \big|_t \big) ds \leq \Gamma. \]
Here $\partial_t \gamma_{k,t}^\perp (s)$ is the component of $\partial_t \gamma_{k,t} (s)$ which is perpendicular to $\gamma_{k,t}$.

For every time $t \in [T_1, T_2)$ denote by $A(t)$ the infimum over the areas of all piecewise smooth maps $f : V \to \MM(t_0)$ such that $f |_{\partial V} = f_{0,t}$ and such that there is a homotopy between $f_0$ and $f$ in space-time which restricts to $f_{0,t'}$ on $\partial V$.

Then the quantity 
\[ t^{1/4} \big( t^{-1} A(t) - 4 \pi n - 4 \Gamma \big) \]
is monotonically non-increasing on $[T_1, T_2)$ and if $T_2 = \infty$, we have
\[ \limsup_{t \to \infty} t^{-1} A(t) \leq  4 \pi n + 4 \Gamma. \]
\end{Proposition}

\begin{proof}
Note that the property of having precise cutoff implies that the metric $g(t)$ has $t^{-1}$-positive curvature which in turn entails that $\scal_t \geq -\frac{3}{2t}$ (see \cite[Definitions 2.10, 2.11(1)]{Bamler-longtime-II}).
Also, by a mollifier argument the infimum $A(t)$ can be taken over all maps which are only continuous and continuously differentiable when restricted to $V \setminus V^{(1)}$ and $V^{(1)}$ as well as bounded in $W^{1,2}$ on each face of $V$.

So the monotonicity of the desired quantity away from surgery times follows directly from Lemma \ref{Lem:evolofAunderRF} together with a variational estimate due to the fact that $f_{0,t}$ can move in time (\`a la Lemma \ref{Lem:variationofV}).
By \cite[Definition 2.11(5)]{Bamler-longtime-II} the the value of $A(t)$ cannot increase under a surgery, i.e. the function $A(t)$ is lower semi-continuous.
\end{proof}

\section{Construction and analysis of simplicial complexes in $M$} \label{sec:constructanalysispolygonalcomplex}
\subsection{Setup and statement of the results} \label{subsec:CombinatorialSetup}
In this section, we construct the simplicial complex $V$ which we will use in section \ref{sec:mainproof}.
We moreover analyze the intersections of images of $V$ with solid tori in $M$.
The results of this section are rather topological, we will however need to make use of some combinatorial geometric arguments in the proofs.

Let $M$ be a closed, orientable, irreducible $3$-manifold which is not a spherical space form.
Consider a geometric decomposition of $T_1, \ldots, T_m \subset M$ of $M$, i.e. the components of $M \setminus (T_1 \cup \ldots \cup T_m)$ are either hyperbolic or Seifert (see \cite[Definition 1.1]{Bamler-longtime-II} for more details).
We will assume from now on that the decomposition has been chosen such that no two hyperbolic components are adjacent to one another.
This can always be achieved by adding a parallel torus next to a torus between two hyperbolic components and hence adding another Seifert piece $\approx T^2 \times (0,1)$.
Let $M_{\textnormal{hyp}}$ be the union of the closures of all hyperbolic pieces of this decomposition and $M_{\textnormal{Seif}}$ the union of the closures of all Seifert pieces.
Then $M = M_{\textnormal{hyp}} \cup M_{\textnormal{Seif}}$ and $M_{\textnormal{hyp}} \cap M_{\textnormal{Seif}} = \partial M_{\textnormal{hyp}} = \partial M_{\textnormal{Seif}}$ is a disjoint union of embedded, incompressible $2$-tori.

The goal of this section is to establish the following Proposition.
In this Proposition, we need to distinguish the cases in which $M$ is covered by a $T^2$-bundle over a circle (i.e. in which $M$ is the quotient of a $3$-torus, the Heisenberg manifold or the Solvmanifold) and in which it is not.
It is not known to the author whether part (a) of the Proposition actually holds in both cases.

\begin{Proposition} \label{Prop:maincombinatorialresult}
There is a finite simplicial complex $V$ and a constant $C < \infty$ such that the following holds:
\begin{enumerate}[label=(\alph*)]
\item In the case in which $M$ is not covered by a $T^2$-bundle over a circle there is a map
\[ f_0 : V \to M \qquad \text{with} \qquad f_0 (\partial V) \subset \partial M_{\textnormal{Seif}} \]
which is a smooth immersion on $\partial V$ such that the following holds:
Let $S \subset \Int M_{\textnormal{Seif}}$, $S \approx S^1 \times D^2$ be an embedded solid torus whose fundamental group injects into the fundamental group of $M$ (i.e. $S$ is incompressible in $M$).
Let moreover $f : V \to M$ be a piecewise smooth map which is homotopic to $f_0$ relative $\partial V$ and $g$ a Riemannian metric on $M$.
Then we can find a compact, smooth domain $\Sigma \subset \IR^2$ and a smooth map $h : \Sigma \to S$ such that $h(\partial \Sigma) \subset \partial S$ and such that $h$ restricted to the interior boundary circles of $\Sigma$ of  is contractible in $\partial S$ and $h$ restricted to the exterior boundary circle of $\Sigma$ is non-contractible in $\partial S$ and such that
\[ \area  h < C \area f. \]
\item In the case in which $M$ is covered by a $T^2$-bundle over a circle the following holds:
$\partial V = 0$ and there are continuous maps 
\[ f_1, f_2, \ldots : V \to M \]
such that for every $n \geq 1$, every map $f'_n : V \to M$ which is homotopic to $f_n$ and every embedded loop $\sigma \subset M$ with the property that all non-trivial multiples of $\sigma$ are non-contractible in $M$, the map $f'_n$ intersects $\sigma$ at least $n$ times, i.e. ${f'}_n^{-1} (\sigma)$ contains at least $n$ points.
\end{enumerate}
\end{Proposition}

We will first establish part (a) of the Proposition in subsections \ref{subsec:reducifnoT2bundle}--\ref{subsec:proofofeasiermaincombinatorialresult} and then part (b) in subsection \ref{subsec:CaseT2bundleoverS1}.

\subsection{Reduction in the case in which $M$ is not covered by a $T^2$-bundle over a circle} \label{subsec:reducifnoT2bundle}
Assume in this subsection that $M$ is not covered by a $T^2$-bundle over a circle.
In order to establish part (a) of Proposition \ref{Prop:maincombinatorialresult}, it suffices to construct a simplicial complex $V$ and a map $f_0 : V \to M$ with the desired properties for every component $M' \subset M_{\textnormal{Seif}}$, i.e. $f_0 (\partial V) \subset \partial M'$ and check that the inequality involving the areas holds for every solid torus $S \subset M'$ and every homotope $f$ of $f_0$ .
We will hence from now on fix a single component $M' \subset M_{\textnormal{Seif}}$.

The next Lemma ensures that we can pass to a finite cover of $M'$ and simplify the structure of $M'$.
This simplification is not really needed in the following analysis, but it makes its presentation more comprehensible.

\begin{Lemma} \label{Lem:ProductstructinFiniteCovering}
Under the assumptions of this subsection there is a finite cover $\widehat\pi' : \widehat{M}' \to M'$ such that the following holds:
There is a Seifert decomposition $\widehat{T}_1, \ldots, \widehat{T}_m \subset \widehat{M}'$ such that the components of $\Int \widehat{M}' \setminus (\widehat{T}_1 \cup \ldots \cup \widehat{T}_m)$ are diffeomorphic to the interiors of manifolds $\widehat{M}_j = \Sigma_j \times S^1$ for $j= 1, \ldots, k$, where each $\Sigma_j$ is a compact orientable surface (possibly with boundary).
The diffeomorphisms can be chosen in such a way that they can be smoothly extended to the boundary tori.

Moreover, we are in one of the following cases:
\begin{enumerate}[label=(\Alph*)]
\item $\widehat{M}'$ is diffeomorphic to $T^2 \times I$ and $m = 0$, $k = 1$.
\item $\widehat{M}'$ is closed and diffeomorphic to an $S^1$-bundle over a closed, orientable surface $\Sigma$ with $\chi(\Sigma) < 0$.
In particular, $m = k = 1$ and the surface $\Sigma$ arises from $\Sigma_1$ by gluing together its two boundary circles.
\item $\Sigma_j$ has at least one boundary component and $\chi(\Sigma_j) < 0$ for all $j = 1, \ldots, k$ and at each torus $T_i$ the fibers coming from the $S^1$-fibration induced from either side are not homotopic to one another.
\end{enumerate}
\end{Lemma}

\begin{proof}
The arguments in this proof are similar to those in \cite{Luecke-Wu-93}.

Let $T_1, \ldots, T_m \subset M'$ be a Seifert decomposition of $M'$, i.e. $T_1$, \ldots, $T_m$ are pairwise disjoint, embedded, incompressible $2$-tori such that the components of $\Int M' \setminus (T_1 \cup \ldots \cup T_m)$ are diffeomorphic to the interiors of compact Seifert spaces $M'_1, \ldots, M'_m$ of $M' \setminus (T_1 \cup \ldots \cup T_m)$ whose quotient spaces are compact orbifolds $O_1, \ldots, O_m$ (possibly with boundary) whose singularities are of cone type.

We first analyze the $2$-orbifolds $O_1, \ldots, O_m$.
By \cite[Lemma 3.8(c)]{Bamler-longtime-II} and the fact that $M$ is aspherical we conclude that each $O_j$ is good, i.e. its interior is diffeomorphic to an isometric quotient of $S^2, \IR^2$ or $\IH^2$ (observe that otherwise we would be able to cover $M$ by two solid tori).
By the same argument and the fact that every orbifold covering of $O_j$ induces a covering of $M_j$, it follows that $O_j$ can also not be a quotient of $S^2$.
So each $O_j$ is an isometric quotient of $\IR^2$ or $\IH^2$.

If $\Int O_j$ is diffeomorphic to an isometric quotient of $\IR^2$, then there is a finite covering $\widehat{O}_j \to O_j$ such that $\widehat{O}_j$ is diffeomorphic to a torus or an annulus.
Let $\widehat{M}'_j \to M'_j$ be the induced covering.
In the first case $m = j = 1$ and $\widehat{M} = \widehat{M}' = \widehat{M}'_j$ carries an $S^1$-fibration over $T^2$.
Since $T^2$ fibers over a circle, this would however imply that $\widehat{M}$ fibers over a circle with $T^2$-fibers, in contradiction to our assumptions.
So $\widehat{O}_j$ is diffeomorphic to an annulus and $\widehat{M}'_j \approx T^2 \times I$.
We will note the following fact which we will use later in the proof:
For every natural number $N \geq 1$, the covering $\widehat{O}_j \to O_j$ can be chosen such that its restriction to every boundary component of $\widehat{O}_j$ is an $N$-fold covering over a circle.
We can moreover pass to a covering $\widehat{M}_j \to \widehat{M}'_j$, $\widehat{M}_j \approx T^2 \times I$ such that the composition $\widehat{M}_j \to \widehat{M}'_j \to M'_j$ over each boundary torus of $M'_j$ is an $N^2$-fold covering of $n_j := 1$ or $n_j := 2$ tori over a torus which is induced by the sublattice $N \IZ^2 \subset \IZ^2$.

If $\Int O_j$ is diffeomorphic to an isometric quotient of $\IH^2$, then by an argument from \cite[Lemma 4.1]{Luecke-Wu-93} for every large enough $N \geq 2$ we can find a finite orbifold covering $\widehat{O}_j \to O_j$ such that $\widehat{O}_j$ is a manifold and such that the covering map restricted to each boundary component of $\widehat{O}_j$ is an $N$-fold covering of the circle.
Consider the induced covering $\widehat{M}'_j \to M'_j$ where $\widehat{M}'_j$ is an $S^1$-bundle over $\Int \widehat{O}_j$.
If $\widehat{O}_j$ is closed, then we are in case (B) of the Lemma, so assume in the following that none of the $\widehat{O}_j$ is closed.
The $S^1$-fibration on each $\widehat{M}'_j$ can then be trivialized, i.e. $\widehat{M}'_j = \widehat{O}_j \times S^1$.
We can hence pass to a further $N$-fold covering $\widehat{M}_j \to \widehat{M}'_j$ using an $N$-fold covering of the $S^1$-factor.
Then for some $n_j \geq 1$ the composition $\widehat{M}_j \to \widehat{M}'_j \to M'_j$ over each boundary torus of $M'_j$ is the disjoint union of $n_j$ many $N^2$-fold coverings over the torus, induced by a sublattice $N \IZ^2 \subset \IZ^2$.

Now choose $N$ large enough such that the construction of the last two paragraphs can be carried out for every $j = 1, \ldots, m$.
Observe that the coverings over every $T_i$ coming from the coverings over the two adjacent $M'_j$ consist of equivalent pieces.
Let $N_0 = n_1 \cdots n_k$ and consider $\frac{N_0}{n_j}$ many disjoint copies of $\widehat{M}_j$ for each $j = 1, \ldots, k$.
It is then not difficult to see that these copies can be glued together along their boundary to obtain a covering $\widehat{M}' \to M'$.
The Seifert decomposition on $M'$ induces a Seifert decomposition $\widehat{T}'_1, \ldots, \widehat{T}'_{m'}$ of $\widehat{M}'$ all of whose pieces are products.

We are now almost done.
As a final step we successively remove tori $\widehat{T}'_i$ which are adjacent to Seifert components $\approx T^2 \times (0,1)$.
Since $\widehat{M}$ cannot be a $T^2$-bundle over a circle, these Seifert components can never be adjacent to such a torus $\widehat{T}'_i$ from both sides.
At the end of this process, we are either left with a single piece $\approx T^2 \times I$ and we are in case (A) of the Lemma or none of the Seifert pieces are diffeomorphic to $T^2 \times I$.
In the latter case we also remove tori $\widehat{T}'_i$ for which the $S^1$-fibers coming from either side are homotopic to one another.
This will either result in two distinct Seifert components getting joined together or in identifying two boundary tori of a single Seifert component.
If at any point in this process the new Seifert component is closed, then we undo the last step and we are in case (B).
Otherwise, we are in case (C).
\end{proof}

We will now show that Proposition \ref{Prop:maincombinatorialresult}(a) is implied by the following Proposition.

\begin{Proposition} \label{Prop:easiermaincombinatorialresultCaseb}
Let $M_0$ be an arbitrary $3$-manifold with $\pi_2 (M_0) = 0$ and $M \subset M_0$ be an embedded, connected, orientable, compact $3$-manifold with incompressible toroidal boundary components such that the fundamental group of $M$ injects into the fundamental group of $M_0$.

Assume that $M$ satisfies one of the following conditions:
\begin{enumerate}[label=(\Alph*)]
\item $M \approx T^2 \times I$.
\item $M$ is the total space of an $S^1$-bundle over a closed, orientable surface $\Sigma$ with $\chi (\Sigma) < 0$.
\item $M$ admits a Seifert decomposition $T_1, \ldots, T_m \subset M$ such that the components of $\Int M \setminus (T_1 \cup \ldots \cup T_m)$ are diffeomorphic to the interiors of $M_j = \Sigma_j \times S^1$ for $j = 1, \ldots, k$, where each $\Sigma_j$ is a compact orientable surface with at least one boundary component and $\chi (\Sigma_j) < 0$.
The diffeomorphisms can be chosen in such a way that they can be smoothly extended to the boundary tori.
Moreover, at each $T_i$ the fibers of the $S^1$-fibrations induced from the manifold $M_j$ on either side are not homotopic to one another.
\end{enumerate}
Then there is a constant $C < \infty$, a simplicial complex $V$ and a continuous map
\[ f_0 : V \to M \qquad \text{with} \qquad f_0 (\partial V) \subset \partial M \]
which is a smooth immersion on $\partial V$ such that the following holds:

Let $S \subset \Int M$, $S \approx S^1 \times D^2$ be an embedded solid torus whose fundamental group injects into the fundamental group of $M$ (i.e. $S$ is incompressible in $M$).
Let moreover $f : V \to M_0$ be a piecewise smooth map which is homotopic to $f_0$ relative $\partial V$ in $M_0$ and $g$ a Riemannian metric on $M_0$.
Then we can find a compact, smooth domain $\Sigma \subset \IR^2$ and a smooth map $h : \Sigma \to S$ such that $h(\partial \Sigma) \subset \partial S$ and such that $h$ restricted to the interior boundary circles of $\Sigma$ of  is contractible in $\partial S$ and $h$ restricted to the exterior boundary circle of $\Sigma$ is non-contractible in $\partial S$ and such that
\[ \area  h < C \area f. \]
\end{Proposition}

\begin{proof}[Proof that Proposition \ref{Prop:easiermaincombinatorialresultCaseb} implies Proposition \ref{Prop:maincombinatorialresult}(b)]
Let $M = M_{\textnormal{hyp}} \cup M_{\textnormal{Seif}}$ be a closed, orientable, irreducible manifold as defined in subsection \ref{subsec:CombinatorialSetup} and $M'$ a component of $M_{\textnormal{Seif}}$.
By van Kampen's Theorem the fundamental group of $M'$ injects into that of $M$.
Consider now the finite covering $\widehat\pi' : \widehat{M}' \to M'$ from Lemma \ref{Lem:ProductstructinFiniteCovering}.
Choose $p \in \widehat{M}'$ and consider the push forward $\widehat\pi'_* (\pi_1 (\widehat{M}', p))$ inside $\pi_1 ( M, \widehat\pi'_* (p))$.
This subgroup induces a covering $\widehat\pi : \widehat{M} \to M$ which can be seen as an extension of $\widehat\pi' : \widehat{M}' \to M'$.
Still, the fundamental group of $\widehat{M}'$ injects into that of $\widehat{M}$.

The cases (A)--(C) of Lemma \ref{Lem:ProductstructinFiniteCovering} for $\widehat{M}'$ correspond to the conditions (A)--(C) in Proposition \ref{Prop:easiermaincombinatorialresultCaseb}.
So we can apply Proposition \ref{Prop:easiermaincombinatorialresultCaseb}  for $M \leftarrow \widehat{M}'$, $M_0 \leftarrow \widehat{M}$ and obtain a simplicial complex $V$ and a map $\widehat{f}_0 : V \to \widehat{M}'$ (observe that $\pi_2(\widehat{M}) = \pi_2(M) = 0$ by \cite[Proposition 3.3]{Bamler-longtime-II} and by the fact that $M$ is irreducible).
Set $f_0 = \widehat\pi \circ \widehat{f}_0 : V \to M$.
Then we can lift any homotopy between $f_0$ and a map $f : V \to M$ to a homotopy between $\widehat{f}_0$ and $\widehat{f} : V \to \widehat{M}$ such that $f = \widehat\pi \circ \widehat{f}$.
Consider now an incompressible solid torus $S \subset M'$ and choose a component $\widehat{S} \subset \widehat\pi^{-1} (S) \cap \widehat{M}'$.
It is easy to see that $\widehat{S}$ is a solid torus as well which is incompressible in $\widehat{M}'$.
So Proposition \ref{Prop:easiermaincombinatorialresultCaseb} provides a compact, smooth domain $\Sigma \subset \IR^2$ and a map $\widehat{h} : \Sigma \to \widehat{M}$ such that $\widehat{h}$ restricted to the exterior boundary circle of $\Sigma$ is non-contractible in $\partial \widehat{S}$, but $\widehat{h}$ restricted to the other boundary circles is contractible in $\partial \widehat{S}$.
Moreover, for $h = \widehat\pi \circ \widehat{h}$ we have
\[ \area h  = \area \widehat{h} < C \area \widehat{f} = C \area f. \]
Clearly, $h$ has the desired topological properties.
\end{proof}

In the following four subsections, we will frequently refer to the conditions (A)--(C).
We first finish off the case in which $M$ satisfies condition (A).

\begin{Proposition} \label{Prop:CasAiseasy}
Proposition \ref{Prop:easiermaincombinatorialresultCaseb} holds if $M$ satisfies condition (A).
\end{Proposition}

\begin{proof}
We argue as in \cite{Bamler-longtime-I}.
Observe that $M \approx T^2 \times I \approx S^1 \times S^1 \times I$.
Denote by $A_1, A_2$ the two embedded annuli of the form
\[ \{ \textnormal{pt} \} \times S^1 \times I, \;\; S^1 \times \{ \textnormal{pt} \} \times I \subset M. \]
Let $V$ be their disjoint union and $f_0 : V \to M$ be the inclusion map.
Then every non-contractible loop $\sigma \subset \Int M$ has non-zero intersection number with one of the maps $f_0 |_{A_1}$ or $f_0 |_{A_2}$.

Consider now the solid torus $S \subset M$ and let $\sigma \subset \Int S$ be a non-contractible curve inside $S$ (and hence also inside $M$).
Choose $i \in \{ 1, 2 \}$ such that $f_0 |_{A_i}$ has non-zero intersection number with $\sigma$.
Then so does $f |_{A_i}$.
Let $f' : A_i \to M$ be a small perturbation of $f |_{A_i}$ which is transversal to $\partial S$ and for which $\area f' < 2 \area f$.
Still, $f'$ has non-zero intersection number with $\sigma$.

Denote the components of $f^{\prime -1} (S)$ by $Q_1, \ldots, Q_p \subset A_i \approx S^1 \times I$.
The sum of the intersection numbers of $f' |_{Q_j}$ with $\sigma$ is non-zero.
Moreover, by the choice of $i$ none of these components $Q_j$ can contain a circle which is non-contractible in $A_i$.
So each $Q_j$ is contained in a closed disk $Q'_j \subset A_i$ with $\partial Q'_j \subset \partial Q_j$ which arises from filling in all its interior boundary circles.
Note that any two such disks, $Q'_{j_1}, Q'_{j_2}$ are either disjoint or one is contained in the other.
By a maximality argument, we can choose $j \in \{ 1, \ldots, p \}$ such that the intersection number of $f' |_{Q_j}$ with $\sigma$ is non-zero, but such that $Q'_j$ does not contain any other $Q_{j'}$ with the same property.
It then follows easily that $f'$ has to have zero intersection number with $\sigma$ on every component of $Q'_j \setminus Q_j$.
Hence, $f'$ restricted to every circle of $\partial Q_j \setminus \partial Q'_j$ is contractible in $\partial S$ and $f'$ restricted to $\partial Q'_j$ is non-contractible.
So if we choose $\Sigma = Q_j \subset Q'_j \approx D^2 \subset \IR^2$ and $h = f'' |_{Q_j}$, then the desired properties are fulfilled and $\area h < \area f' < 2 \area f'$.
\end{proof}

It remains to prove Proposition \ref{Prop:easiermaincombinatorialresultCaseb} in the cases in which $M$ satisfies condition (B) or (C).
Its proof in these two cases will be carried out in subsection \ref{subsec:proofofeasiermaincombinatorialresult}.
The proof makes use of a simplicial complex $V$ which will be constructed and analyzed in the following subsection and relies on a certain combinatorial convexity estimate on $V$ which will be derived in subsection \ref{subsec:combconvexincaseC} for case (C) and in subsection \ref{subsec:combconvexincaseB} for case (B).

\subsection{Combinatorial geometry of $\td{M}$ if $M$ satisfies condition (B) or (C)} \label{subsec:constructionofV}
In this subsection we will set up the proof of Proposition \ref{Prop:easiermaincombinatorialresultCaseb}.
In particular, we will construct the simplicial complex $V$ and introduce the tools that will be needed in the following two subsections.

Assume that $M$ satisfies condition (B) or (C) in Proposition \ref{Prop:easiermaincombinatorialresultCaseb}, i.e. $M$ is a compact, connected, orientable $3$-manifold with incompressible toroidal boundary components.
If $M$ satisfies condition (C), we fix the Seifert decomposition $T_1, \ldots, T_m$ of $M$ as well as the identifications of the components of $\Int M \setminus (T_1 \cup \ldots \cup T_m)$ with the interiors of the products $M_j \approx \Sigma_j \times S^1$ ($j = 1, \ldots, k$).
Here $\Sigma_1, \ldots, \Sigma_m$ are compact surfaces with at least one boundary component and negative Euler characteristic.
If $M$ satisfies condition (B), then we set $m = k = 1$ and we can find a torus $T_1 \subset M$ such that $M \setminus T_1$ is diffeomorphic to the interior to the product $\Sigma_1 \times S^1$, where $\Sigma_1$ is a compact, orientable surface with two boundary circles which can be obtained from $\Sigma$ by cutting along a non-separating, embedded loop.
Moreover, $\chi(\Sigma_1) = \chi(\Sigma) < 0$.
In either case, we assume that the diffeomorphisms which identify the interior of each $M_j$ with the corresponding component of $\Int M \setminus (T_1 \cup \ldots \cup T_m)$ can be continued smoothly up to the boundary tori.
If $M$ satisfies condition (C), then the fibrations coming from either side of each torus $T_i$ are assumed to be non-homotopic to one another and in case (B) we assume that the fibration on $M_1$ has been chosen such that both fibrations agree.

We will mainly be working in the universal covering $\td{M}$ of $M$.
Let $\pi : \td{M} \to M$ be the covering projection.

\begin{Definition}[chambers]
The closures $K \subset \td{M}$ of components of the preimages of components of $M \setminus (T_1 \cup \ldots \cup T_m)$ under $\pi$ are called \emph{chambers} and the set of chambers is denoted by $\mathcal{K}$.
\end{Definition}

\begin{Definition}[walls]
The components $W$ of $\partial \td{M}$ and of the preimages $\pi^{-1} (T_i)$, $i = 1, \ldots, m$ are called \emph{walls} and the set of walls is denoted by $\mathcal{W}$.
We say that two distinct chambers $K_1, K_2 \in \mathcal{K}$ are \emph{adjacent} if they share a common wall.
\end{Definition}

By van Kampen's Theorem every chamber $K \in \mathcal{K}$ can be viewed as the universal cover of $M_{j_K}$ for a unique $j_K \in \{1, \ldots, k\}$.
So $K \approx \td\Sigma_{j_K} \times \IR$.
The boundary of $K$ is a disjoint union of walls which cover exactly the tori $T_i$ and the boundary tori of $M$ which are adjacent to $M_{j_K}$, and these tori are in one-to-one correspondence with the boundary circles of $\Sigma_{j_K}$.
Moreover, every wall is diffeomorphic to $\IR^2$.
For later purposes, we will replace the $j$-index by $K$ and write for example $M_K = M_{j_K}$ and $\Sigma_K = \Sigma_{j_K}$.
Note that $K$ does not intersect any wall in its interior.
So the complement of the union of all walls in $\td{M}$ is equal to the union of the interiors of all chambers.

\begin{Lemma} \label{Lem:KKistree}
Every wall $W \in \mathcal{W}$, $W \not\subset \partial \td{M}$ separates $\td{M}$ into two components.
So every two distinct chambers $K_1, K_2 \in \mathcal{K}$ can only intersect in at most one wall $W = K_1 \cap K_2$ and the adjacency graph of $\mathcal{K}$ is a tree. 
\end{Lemma}

\begin{proof}
If $W \in \mathcal{W}$ did not separate $\td{M}$, then we could find a loop $\gamma \subset \td{M}$ which intersects $W$ transversally and exactly once, i.e. its intersection number with $W$ is $1$.
However $\gamma \subset \td{M}$ must be contractible.
\end{proof}

On each torus $T_i$ and boundary torus of $M$ we fix an affine structure and a point $e_i \in T_i$ for the remainder of this subsection.
These affine structures induce an affine structure on all walls $W \in \mathcal{W}$.
We can assume that the product structures on each $M_j \approx \Sigma_j \times S^1$ are chosen such that the circle fibers on each boundary component $M_j$ coming from the $S^1$-factor or the boundary circle of the $\Sigma_j$ are geodesic circles in the corresponding torus $T_i$.

Now, for each $j = 1, \ldots, k$ we choose an embedded section $S_j \subset M_j \approx \Sigma_j \times S^1$ of the form $\Sigma_j \times \{ \text{pt} \}$.
Next, we choose embedded and pairwise disjoint curves inside each $\Sigma_j$, whose endpoints lie in the boundary of $\Sigma_j$ and which cut the interior of $\Sigma_j$ into a topological ball, i.e. a fundamental domain.
Denote their union by $C^*_j \subset \Sigma_j$ and set $C_j = C^*_j \times S^1$.
Let now
\[ V = T_1 \cup \ldots \cup T_m \cup S_1 \cup \ldots \cup S_k \cup C_1 \cup \ldots \cup C_k. \]
By construction $V$ can be seen as an embedded, finite simplicial complex such that $\partial V \subset \partial M$.
Its $1$-skeleton $V^{(1)}$ is the union of $\partial S_j$, $\partial C_j$ and $C_j \cap S_j$ for all $j = 1, \ldots, k$.
All its vertices $V^{(0)}$ lie on $T_1 \cup \ldots \cup T_m \cup \partial M$.
The complement $\Int M \setminus V$ is a disjoint union of $k$ topological balls $\approx (\Sigma_1 \setminus C^*_1) \times (0,1), \ldots, (\Sigma_k \setminus C^*_k) \times (0,1)$.

Consider now the universal covering $\pi : \td{M} \to M$ and set $\td{V} = \pi^{-1} (V) \subset \td{M}$.
Then $\td{V}$ is an infinite simplicial complex with $\partial \td{V} \subset \partial \td{M}$ and the components of $\Int \td{M} \setminus \td{V}$ are topological balls on which $\pi$ is injective.
Their boundary is diffeomorphic to a polyhedral $2$-sphere.

\begin{Definition}[cells]
The closure $Q$ of any component component of $\td{M} \setminus \td{V}$ is called a \emph{cell} and the set of cells is denoted by $\mathcal{Q}$.
Two cells are called \emph{adjacent} if their intersection contains a point of $\td{V} \setminus \td{V}^{(1)}$.
\end{Definition}

So every chamber $K \in \mathcal{K}$ is equal to the union of cells $Q \subset K$.
Identify $K$ with $\td\Sigma_K \times \IR$ as before.
The structure of $\td{V}$ in $K$ can then be understood as follows:
Let $\td{C}^*_K$ be the preimage of $C^*_K$ under the universal covering map $\td\Sigma_K \to \Sigma_K$.
Then $\td{V} \cap K$ is equal to the union of $\pi^{-1} (C_K) \cap K \approx \td{C}^*_K \times \IR$ with $\pi^{-1} (S_K) \cap K$ and $\partial K$.
So the arrangement of the cells $Q \subset K$ is reflected by the following identity
\begin{equation} \label{eq:arrangementQinK}
\bigcup_{Q \in \mathcal{Q}, \; Q \subset K} \Int Q = (\Int \td\Sigma_K \setminus \td{C}^*_K) \times (\IR \setminus \IZ).
\end{equation}
We will always refer to the first factor in this cartesian product as the \emph{horizontal} direction and to the second factor as the \emph{vertical} direction.
In the next definition we group cells which share the same horizontal coordinates.

\begin{Definition}[columns]
Consider a chamber $K \in \mathcal{K}$ and choose the identification $K \cong \td\Sigma_K \times \IR$ as in the last paragraph.
Then the closure $E$ of each component of $(\Int \td\Sigma_K \setminus \td{C}^*_K) \times \IR$ is called a \emph{column}.
The set of columns of $K$ is denoted by $\mathcal{E}_K$.

We say that two columns $E_1, E_2 \in \mathcal{E}_K$ are \emph{adjacent} if they intersect.
An ordered tuple $(E_0, \ldots, E_n)$ of columns for which $E_i$ is adjacent to $E_{i+1}$ is called a \emph{chain between $E_1$ and $E_n$} and $n$ is called its \emph{length}.
It is called \emph{minimal} if its length is minimal amongst all chains between the same columns.
\end{Definition}

So each chamber $K \in \mathcal{K}$ is equal to the union of all its columns $E \in \mathcal{E}_K$ and every such column $E$ consists of cells $Q \subset E$ which are arranged in a linear manner.
Next, we define distance functions with respect to the horizontal and vertical direction in (\ref{eq:arrangementQinK}).

\begin{Definition}[horizontal and vertical distance within a chamber]
Let $K \in \mathcal{K}$ be a chamber and $E_1, E_2 \in \mathcal{E}_K$ two columns.
We define their \emph{horizontal distance  $\dist^H_K (E_1, E_2)$ (within $K$)} to be the minimal length of a chain between $E_1$ and $E_2$.
For two cells $Q_1, Q_2 \subset K$ with $Q_1 \subset E_1$ and $Q_2 \subset E_2$ we define the \emph{horizontal distance $\dist^H_K (Q_1, Q_2) = \dist^H_K (E_1, E_2)$ (within $K$)}.
We say that $Q_1, Q_2$ are \emph{vertically aligned (within $K$)} if $\dist^H_K (Q_1, Q_2) = 0$, i.e. if $Q_1, Q_2$ lie in the same column.

For two cells $Q_1, Q_2 \subset K$ we define the \emph{vertical distance $\dist^V_K (Q_1, Q_2)$ (within $K$)} by the minimal number of times that a curve $\gamma : [0,1] \to K$ with $\gamma(0) \in \Int Q_1$ and $\gamma(1) \in \Int Q_2$ intersects $\pi^{-1} (S_K)$, i.e. the number if integers between the second coordinates in (\ref{eq:arrangementQinK}) of both cells.
We say that $Q_1, Q_2$ are \emph{horizontally aligned (within $K$)} if $\dist^H_K (Q_1, Q_2) = 0$.
\end{Definition}

Obviously, both distance functions satisfy the triangle inequality.
Two cells $Q_1, Q_2 \subset K$ are adjacent if and only if $\dist^H_K(Q_1, Q_2) + \dist^V_K(Q_1, Q_2) = 1$.
And they are disjoint if and only if this sum is $\geq 2$ and not both summands are equal to $1$.

\begin{Lemma} \label{Lem:EEinKisatree}
Assume that $M$ satisfies condition (B) or (C).
Consider a chamber $K \in \mathcal{K}$.
Then the set of columns $\mathcal{E}_K$ together with the adjacency relation describes a tree with constant valency $\geq 4$.
So between every two columns $E_1, E_2 \in \mathcal{E}_K$, there is a unique minimal chain between $E_1, E_2$ and a chain between $E_1, E_2$ is the minimal one if and only if it  contains each column not more than once.
Moreover, for every three columns $E_1, E_2, E_3 \in \mathcal{E}_K$ there is a unique column $E^* \in \mathcal{E}_K$ which lies on all three minimizing chains between every pair of $E_1, E_2, E_3$.

Finally, for every two columns $E_1, E_2 \in \mathcal{E}_K$ with $\dist^H_K (E_1, E_2) \geq 2$ there is at most one wall $W \in \mathcal{W}$ which is adjacent to both $E_1$ and $E_2$.
\end{Lemma}

\begin{proof}
By a simple intersection number argument as in the proof of Lemma \ref{Lem:KKistree}, a loop in $\td\Sigma_K$ cannot cross a component of $\td{C}^*_K \subset \td\Sigma_K$ exactly most once.
This establishes the tree property.

Now assume that there are two distinct boundary components $B_1, B_2 \subset \partial \td\Sigma_K$ which are adjacent to two distinct components $U_1, U_2 \subset \td\Sigma_K \setminus \td{C}^*_K$ at the same time.
Since $\td\Sigma_K$ is simply connected, the closure of the set $B_1 \cup B_2 \cup U_1 \cup U_2$ separates $\td\Sigma_K$ into two open components $A_1, A_2$ one of which, say $A_1$, has compact closure.
So $A_1$ only contains finitely many components of $\td\Sigma_K \setminus \td{C}^*_K$ and all these components are only adjacent to each other or to $U_1$ or $U_2$.
This however contradicts the tree property.
\end{proof}

In the following we want to understand the adjacency structure of $\mathcal{Q}$ on $\td{M}$.
As a first step we analyze its structure near walls.

\begin{Lemma} \label{Lem:QQestimatesatW}
There is a constant $C_0 < \infty$ such that the following holds:

Let $W \in \mathcal{W}$, $W \not\subset \partial M$ be a wall and let $K, K' \in \mathcal{K}$ be the chambers which are adjacent to $W$ from either side.
Then the columns $E \in \mathcal{E}_K$, $E' \in \mathcal{E}_{K'}$ intersect $W$ in affine strips $E \cap W$, $E' \cap W$ (i.e. domains bounded by two parallel straight lines).
In case in which $M$ satisfies condition (B), these strips are all parallel and if $M$ satisfies condition (C), each pair of strips coming from $K$ and $K'$ are not parallel to one another; so they intersect in a non-empty compact set.

We furthermore have the the following estimates between the horizontal and vertical distance functions in $K$ and $K'$:
\begin{enumerate}[label=(\alph*)]
\item Assume that $M$ satisfies condition (B) or (C) and let $Q_1, Q_2 \subset K$ be cells which are adjacent to a common cell $Q' \subset K$.
Then
\[ \dist^H_K (Q_1, Q_2), \; \dist^V_K (Q_1, Q_2) < C_0. \]
\item Assume that $M$ satisfies condition (C) and let $Q_1, Q_2 \subset K$, $Q'_1, Q'_2 \subset K'$ be cells such that $Q_1, Q'_1$ and $Q_2, Q'_2$ are adjacent and such that $Q'_1, Q'_2$ are vertically aligned.
Then
\[ \dist^V_K (Q_1, Q_2), \; \dist^V_{K'}(Q'_1, Q'_2) < C_0 \dist^H_K (Q_1, Q_2) + C_0. \]
\item Assume that $M$ satisfies condition (C) and consider four cells $Q_1$, $Q_2$, $Q_3$, $Q_4 \subset K$.
Assume that $Q_1, Q_2$ and $Q_3, Q_4$ are vertically aligned and assume that there are columns $E'_1, E'_2 \in \mathcal{E}_{K'}$ such that $Q_1, Q_4$ are adjacent to some cells in $E'_1$ and $Q_2, Q_3$ are adjacent to some cells in $E'_2$.
Then
\[ \big| \dist^V_K (Q_1, Q_2) - \dist^V_K (Q_3, Q_4) \big| < C_0. \]
\item Assume that $M$ satisfies condition (C) and consider cells $Q_1, Q_2 \subset K$ and $Q'_1, Q'_2 \subset K'$ such that $Q_1, Q'_1$ and $Q_2, Q'_2$ are adjacent and that $\dist^H_K (Q_1, Q_2), \dist^H_{K'} (Q'_1, Q'_2) \leq 3$.
Then
\[ \dist^V_K (Q_1, Q_2), \; \dist^V_{K'} (Q'_1, Q'_2) < C_0. \]
\item Assume that $M$ satisfies condition (B) and consider cells $Q_1, Q_2 \subset K$ and $Q'_1, Q'_2 \subset K'$ such that $Q_1, Q'_1$ and $Q_2, Q'_2$ are adjacent.
Then
\[ \dist^V_K (Q_1, Q_2) < \dist^V_{K'} (Q'_1, Q'_2) + C_0 \dist^H_K (Q_1, Q_2) + C_0. \]
\item Assume that $M$ satisfies condition (B) or (C) and consider cells $Q_1, Q_2 \subset K$ and $Q'_1, Q'_2 \subset K'$ such that $Q_1, Q'_1$ and $Q_2, Q'_2$ are adjacent.
Then
\[ \qquad\quad\; \dist^H_K (Q_1, Q_2), \; \dist^V_K(Q_1, Q_2) < C_0 \dist^H_{K'} (Q'_1, Q'_2) + C_0 \dist^V_{K'} (Q'_1, Q'_2) + C_0. \]
\item Assume that $M$ satisfies condition (B) or (C) and consider cells $Q_1, \linebreak[1] Q_2, \linebreak[1] Q_3, \linebreak[1] Q_4 \subset K$ and $Q'_1, Q'_2, Q'_3, Q'_4 \subset K'$ such that $Q_i$ and $Q'_i$ are adjacent for all $i = 1, \ldots, 4$.
Assume moreover that $\dist^H_K(Q_1, Q_2) = \dist^H_K (Q_3, Q_4)$ and $\dist^V_K(Q_1, Q_2) = \dist^V_K (Q_3, Q_4)$ in an oriented sense, i.e. the cells $Q_1, Q_2, Q_3, Q_4$ form a ``parallelogram'' along $W$.
Then
\[ \qquad\quad \big| \dist^H_{K'} (Q'_1, Q'_2) - \dist^H_{K'} (Q'_3, Q'_4) \big|, \; \big| \dist^V_{K'} (Q'_1, Q'_2) - \dist^V_{K'} (Q'_3, Q'_4) \big| < C_0. \]
\end{enumerate}
\end{Lemma}

\begin{proof}
The pattern by which the cells of $K$ and $K'$ are arranged along $W$ is doubly periodic.
So we can introduce euclidean coordinates $(x_1, x_2) : W \to \IR^2$ such that for every two cells $Q_1, Q_2 \subset K$ and points $p_1 \in Q_1$, $p_2 \in Q_2$ we have $|\dist^H_K (Q_1, Q_2) - |x_1(p_1) - x_1(p_2)| | < C$ and $|\dist^V_K (Q_1, Q_2) - |x_2(p_1)-x_2(p_2)| | < C$ for some constant $C$.
Similarly, we can find euclidean coordinates $(x'_1, x'_2) : W \to \IR^2$ with the analogous behavior for cells $K'$ such that the origins of $(x_1, x_2)$ and $(x'_1, x'_2)$ agree.
The transformation matrix $A \in \IR^{2 \times 2}$ with $A (x_1, x_2) = (x'_1, x'_2)$ is invertible.
In case (C) we have $A_{12} \neq 0$ and in case (B) we have $A_{12} = 0, A_{11} \neq 0$ and $A_{22} = 1$.
All assertions of the Lemma now follow from the corresponding statements for these two coordinate systems.
\end{proof}

Next, consider a smooth curve $\gamma : [0,1] \to \td{M}$.

\begin{Definition}[general position]
We say that $\gamma$ is in \emph{general position} if its endpoints $\gamma(0), \gamma(1) \not\in \td{V}$ and if $\gamma$ intersects $\td{V}$ transversally and only in $\td{V} \setminus \td{V}^{(1)}$.
If $Q_1, Q_2 \in \mathcal{Q}$ are two cells with $\gamma(0) \in Q_1$ and $\gamma(1) \in Q_2$, then we say that $\gamma$ \emph{connects $Q_1$ with $Q_2$}.
\end{Definition}

Let $\eta, H > 0$ be constants whose value will be determined later in subsection \ref{subsec:proofofeasiermaincombinatorialresult}.
In the course of the following three subsections, we will need to assume that $\eta$ is small enough and $H$ is large enough to make certain arguments work out.

\begin{Definition}[length and distance]
The \emph{(combinatorial) length} $|\gamma|$ of a curve $\gamma : [0,1] \to \td{M}$ in general position is defined as
\begin{multline*}
 |\gamma| = \eta \; \big(\textnormal{number of intersections of $\gamma$ with $\pi^{-1} (S_1 \cup \ldots \cup S_k)$} \big) \\
 + H \; \big(\textnormal{number of intersections of $\gamma$ with $\pi^{-1} (T_1 \cup \ldots \cup T_m)$} \big) \\
  + \big(\textnormal{number of intersections of $\gamma$ with $\pi^{-1} (C_1 \cup \ldots \cup C_k)$} \big) 
\end{multline*}
The \emph{(combinatorial) distance} $\dist (Q_1, Q_2)$ between two cells $Q_1, Q_2 \in \mathcal{Q}$ is the minimal combinatorial length of all curves in general position between $Q_1$ and $Q_2$.
A curve $\gamma : [0,1] \to \td{M}$ in general position is said to be \emph{(combinatorially) minimizing} if its length is equal to the combinatorial distance between the two cells that contain its endpoints.
\end{Definition}

Observe that $(\mathcal{Q}, \dist)$ is a metric space.
On a side note, it is an interesting coincidence that this metric space approximates the conjectured geometric behavior of the Ricci flow metric $t^{-1} g_t$ lifted to the universal cover $\td{M}$.

Our main characterization of combinatorially minimizing curves will be stated in Proposition \ref{Prop:characterizationminimizinggamma} in case (C) and in Proposition \ref{Prop:characteriztionminimizinggammaCaseB} in case (B).
We conclude this subsection by pointing out three basic properties of combinatorially minimizing curves.

\begin{Lemma} \label{Lem:subsegmentofminimizing}
If $\gamma : [0,1] \to \td{M}$ is combinatorially minimizing, then so is every orientation preserving or reversing reparameterization and every subsegment of $\gamma$ whose endpoints don't lie in $\td{V}$.
\end{Lemma}

\begin{proof}
Obvious.
\end{proof}

\begin{Lemma}
For any cell $Q \in \mathcal{Q}$, the preimage $\gamma^{-1} (Q)$ under a combinatorially minimizing curve $\gamma : [0,1] \to \td{M}$ is a closed interval, i.e. $\gamma$ does not reenter $Q$ after exiting it.
\end{Lemma}

\begin{proof}
Otherwise we could replace $\gamma$ by a shorter curve.
\end{proof}

\begin{Lemma} \label{Lem:gammainsidesingleK}
Assume that $\gamma : [0,1] \to \td{M}$ is combinatorially minimizing and stays within some chamber $K \in \mathcal{K}$.
Let $E_0, \ldots, E_n \in \mathcal{E}_K$ be the columns that $\gamma$ intersects in that order.
Then $(E_0, \ldots, E_n)$ is a minimal chain in $\mathcal{E}_K$.

Moreover, $\gamma$ intersects each component of $\pi^{-1} (S_1 \cup \ldots \cup S_k)$ at most once.
So if the endpoints of $\gamma$ lie in cells $Q_1, Q_2 \in \mathcal{Q}$, then
\[ |\gamma|  = \dist (Q_1, Q_2) = \dist^H_K (Q_1,Q_2) + \eta \dist^V_K(Q_1, Q_2). \]
Finally, for any two cells $Q_1, Q_2 \subset K$ we have $\dist(Q_1, Q_2) \leq \dist^H_K (Q_1, Q_2) + \eta \dist^V_K (Q_1, Q_2)$.
\end{Lemma}

\begin{proof}
This follows easily from the cell structure of $K$.
\end{proof}

\subsection{A combinatorial convexity estimate if $M$ satisfies condition (C)} \label{subsec:combconvexincaseC}
In this subsection we assume that $M$ satisfies condition (C) in Proposition \ref{Prop:easiermaincombinatorialresultCaseb}.
We will analyze the combinatorial distance function on $\mathcal{Q}$ in this case.
The main result in this section will be the combinatorial convexity estimate in Proposition \ref{Prop:combinatorialconvexityseveralSeifert}.

\begin{Lemma} \label{Lem:gammainsidetwoKKs}
There are constants $\eta^* > 0$ and $H^* < \infty$ such that if $\eta \leq \eta^*$ and $H \geq H^*$, the following holds:

Consider two chambers $K, K' \in \mathcal{K}$ which are adjacent to a common wall $W = K \cap K'$ from either side and assume that $\gamma : [0,1] \to K \cup K'$ is combinatorially minimizing.
Then $\gamma$ intersects $W$ at most twice.
\begin{enumerate}[label=(\alph*)]
\item If $\gamma$ intersects $W$ exactly once, then there is a unique column $E \in \mathcal{E}_K$ in $K$ which is both adjacent to $W$ and intersects $\gamma$.
The same is true in $K'$.
\item If $\gamma$ intersects $W$ exactly twice and $\gamma(0), \gamma(1) \in K$, then there is a unique column $E' \in \mathcal{E}_{K'}$ such that $\gamma$ is contained in $K \cup E'$.
Moreover there are exactly two columns $E_1, E_2 \in \mathcal{E}_K$ which are adjacent to $W$ and which intersect $\gamma$ and we have $\dist^H_K (E_1, E_2) > H$.
\item If $\gamma$ does not intersect $W$, but intersects two columns $E_1, E_2 \in \mathcal{E}_K$ which are both adjacent to $W$, then $\dist^H_K (E_1, E_2) < 3H$.
\end{enumerate}
\end{Lemma}

\begin{proof}
We first establish assertion (a).
Assume without loss of generality that $\gamma(0) \in K$ and $\gamma(1) \in K'$.
Let $Q \subset K$ be the last cell that $\gamma$ intersects inside $K$ and $Q' \subset K'$ the first cell in $K'$.
So $Q, Q'$ are adjacent.
Let $E \in \mathcal{E}_K$ be the column which contains $Q$ and $E'$ the column which contains $Q'$.
Assume that contrary to the assertion there is another column $E_1 \neq E \in \mathcal{E}_K$ that is adjacent to $W$ and intersects $\gamma$.
Choose a cell $Q_1 \subset E_1$ that intersects $\gamma$.
Then by Lemma \ref{Lem:gammainsidesingleK}
\[ \dist(Q_1, Q') = \dist(Q_1, Q) + H = \dist^H_K (E_1, E) + \eta \dist^V_K (Q_1, Q) + H. \]
Let $Q_2 \subset E_1$ be the cell which is horizontally aligned with $Q$.
It follows by the triangle inequality and by previous equation that
\[ \dist(Q_2, Q') = \dist^H_K (E_1, E) + H. \]
Since $E_1 \cap W$ and $E' \cap W$ are non-parallel strips in $W$, we can find cells $Q_3 \subset E_1$ and $Q'_3 \subset E'$ which are adjacent to each other and by Lemma \ref{Lem:QQestimatesatW}(b) we can estimate
\[ \dist^V_K (Q_2, Q_3), \; \dist^V_{K'} (Q', Q'_3) < C_0 \dist^H_K (E_1, E) + C_0. \]
We then conclude
\begin{multline*}
 \dist^H_K (E_1, E) + H = \dist(Q_2, Q') \leq \dist(Q_2, Q_3) + \dist(Q_3, Q'_3) + \dist(Q'_3, Q') \\
 < 2 \eta \big( C_0 \dist^H_K (E_1, E) + C_0 \big) + H.
\end{multline*}
For $\eta < (4 C_0)^{-1}$ this implies $\dist^H_K (E_1, E) < 1$ and hence $E_1 = E$.

Next, we show assertion (b).
Let $E_1, E_2 \in \mathcal{E}_K$ be the the columns that $\gamma$ intersects right before intersecting $W$ for the first time and right after intersecting $W$ for the second time.
Let $E'_1, E'_2 \in \mathcal{E}_{K'}$ be the first and last columns that $\gamma$ intersects inside $K'$.
Assertion (a) applied to the subsegments of $\gamma$ between $\gamma(0)$ and $E'_2$ and between $E'_1$ and $\gamma(1)$, yields that $E' := E'_1 = E'_2$.
Since the subsegment of $\gamma$ which is contained in $K'$ has both of its endpoints in $E'$, it has to be fully contained in it.
Moreover, assertion (a) implies that there are no other columns than $E_1, E_2$ in $K$ which are adjacent to $W$ and which intersect $\gamma$.

It remains to show the inequality on the horizontal distance between $E_1, E_2$.
We do this by comparing the intrinsic and extrinsic distance between these two columns.
Choose $Q_1 \subset E_1$ and $Q'_1 \subset E'$ such that $\gamma$ crosses $W$ between $Q_1$ and $Q'_1$ for the first time and pick $Q_2 \subset E_2$ and $Q'_2 \subset E'$ accordingly.
So $Q_1, Q'_1$ and $Q_2, Q'_2$ are adjacent.
Lemma \ref{Lem:QQestimatesatW}(b) provides the bound $\dist^V_K (Q_1, Q_2) < C_0 \dist^H_K (Q_1, Q_2) + C_0$.
So
\begin{multline*}
 2H \leq \dist(Q_1, Q_2) \leq \dist^H_K (Q_1, Q_2) + \eta \dist^V_K (Q_1, Q_2) \\
  < (1 + \eta C_0) \dist^H_K (Q_1, Q_2) + \eta C_0 .
\end{multline*}
The desired inequality follows for $\eta < (2C_0)^{-1}$ and $H >2$.

We can now show that $\gamma$ intersects $W$ at most twice.
Assume not.
After passing to a subsegment and possibly reversing the orientation, we may assume that $\gamma$ intersects $W$ exactly three times and that $\gamma(0) \in K$, $\gamma(1) \in K'$.
By assertion (b) applied to subsegments of $\gamma$ which intersects $W$ exactly twice, we find that there are columns $E_1, E_3 \in \mathcal{E}_K$ and $E'_2, E'_4 \in \mathcal{E}_{K'}$ which are all adjacent to $W$ such that $\gamma$ crosses $W$ first between $E_1$ and $E'_2$, then between $E'_2$ and $E_3$ and finally between $E_3$ and $E'_4$.
Choose cells $Q_1 \subset E_1$, $Q'_1, Q'_2 \subset E'_2$, $Q_2, Q_3 \subset E_3$, $Q'_3 \subset E'_4$ such that $\gamma$ crosses $W$ first between $Q_1$ and $Q'_1$, then between $Q'_2$ and $Q_2$ and finally between $Q_3$ and $Q'_3$.
So
\begin{multline*}
 \dist(Q_1, Q'_3) = \dist(Q_1, Q'_1) + \dist(Q'_1, Q'_2) + \dist(Q'_2, Q_2) \\ + \dist(Q_2, Q_3) + \dist(Q_3, Q'_3) = 3H + \eta \dist^V_{K'} (Q'_1, Q'_2) + \eta \dist^V_K (Q_2, Q_3).
\end{multline*}
Since $E_1 \cap W$ and $E'_4 \cap W$ are non-parallel strips in $W$, we can find cells $Q^* \subset E_1$ and $Q^{*\prime} \subset E'_4$ which are adjacent to each other.
By Lemma \ref{Lem:QQestimatesatW}(c)
\begin{multline*}
 \dist^V_K (Q_1, Q^*) < \dist^V_K (Q_2, Q_3) + C_0 \qquad \text{and} \\ \quad
\dist^V_{K'} (Q^{*\prime}, Q'_3) < \dist^V_{K'} (Q'_1, Q'_2) + C_0.
\end{multline*}
So
\begin{multline*}
\dist (Q_1, Q'_3) \leq \dist(Q_1, Q^*) + \dist(Q^*, Q^{* \prime}) + \dist(Q^{* \prime}, Q'_3) \\
< \eta \dist^V_K (Q_2, Q_3) + \eta C_0 + H + \eta \dist^V_{K'} (Q'_1, Q'_2) + \eta C_0 \\
= \dist(Q_1, Q'_3) - 2H + 2 \eta C_0.
\end{multline*}
We obtain a contradiction for $\eta C_0 < H$.

Finally, we show assertion (c).
Assume now that $\gamma$ does not intersect $W$ and choose cells $Q_1 \subset E_1$ and $Q_2 \subset E_2$ which intersect $\gamma$.
Since $\gamma$ stays within $K$ we have
\[ \dist(Q_1, Q_2) = \dist^H_K (E_1, E_2) + \eta \dist^V_K (Q_1, Q_2). \]
Let $Q_3 \subset E_2$ be the cell which is horizontally aligned with $Q_1$.
By the triangle inequality and by previous equation
\[ \dist(Q_1, Q_3) = \dist^H_K (E_1, E_2). \]
Let $Q'_1 \subset K'$ be a cell which is adjacent to $Q_1$ and let $E' \in \mathcal{E}_{K'}$ be the column that contains $Q'_1$.
Since $E' \cap W$ and $E_2 \cap W$ are non-parallel strips, we can find cells $Q'_2 \subset E'$ and $Q''_2 \subset E_2$ which are adjacent.
By Lemma \ref{Lem:QQestimatesatW}(b), we have $\dist^V_K(Q''_2, Q_3), \dist^V_{K'} (Q'_1, Q'_2) < C_0 \dist^H_K (E_1, E_2) + C_0$.
So
\begin{multline*}
 \dist^H_K (E_1, E_2) = \dist(Q_1, Q_3) \leq \dist(Q_1, Q'_1) + \dist(Q'_1, Q'_2) \\+ \dist(Q'_2, Q''_2) + \dist(Q''_2, Q_3) 
 < 2 H + 2 \eta C_0 \dist^H_K (E_1, E_2) + 2 \eta C_0.
\end{multline*}
The desired inequality follows for $2 \eta C_0 < \frac1{10}$ and $H > 1$.
\end{proof}

The next Proposition provides an accurate characterization of the behavior of a minimizing curve.

\begin{Proposition} \label{Prop:characterizationminimizinggamma}
Assume that $M$ satisfies condition (C).
There are constants $\eta^* > 0$ and $H^* < \infty$ such that if $\eta \leq \eta^*$ and $H \geq H^*$, the following holds:

Consider a combinatorially minimizing curve $\gamma : [0,1] \to \td{M}$.
Then
\begin{enumerate}[label=(\alph*)]
\item For every chamber $K \in \mathcal{K}$ and every column $E \in \mathcal{E}_K$, the preimage $\gamma^{-1} (E)$ is a connected interval, i.e. $\gamma$ does not exit and reenter $E$.
\item $\gamma$ intersects every wall $W \in \mathcal{W}$ at most twice.
Assume that $K, K' \in \mathcal{K}$ are two chambers which are adjacent to a wall $W \in \mathcal{W}$ from either side.
Then
\begin{enumerate}[label=(b\arabic*)]
\item If $\gamma$ intersects $W$ exactly once, then there is a unique column $E \in \mathcal{E}_K$ which is both adjacent to $W$ and intersects $\gamma$.
Moreover, for every column $E^* \in \mathcal{E}_K$ which also intersects $\gamma$, the minimal chain between $E$ and $E^*$ intersects $W$ in at most two columns.
\item If $\gamma$ intersects $W$ twice and its endpoints lie on the same side of $W$ as $K$, then within both intersections it stays inside a column $E' \in \mathcal{E}_{K'}$ adjacent to $W$.
Moreover, there are exactly two columns $E_1, E_2 \in \mathcal{E}_K$ which intersect $\gamma$ in this order and which are adjacent to $W$ and we have $\dist^H_K (E_1, E_2) > H$.
The curve $\gamma$ leaves $K$ through $W$ right after $E_1$ and reenters $K$ through $W$ right before $E_2$.
\item If $\gamma$ does not intersect $W$, but intersects two columns $E_1, E_2 \in \mathcal{E}_K$ which are both adjacent to $W$, then $\dist^H_K (E_1, E_2) < 3H$.
\end{enumerate}

\item Consider a chamber $K \in \mathcal{K}$ and let $E_1, \ldots, E_n \in \mathcal{E}_K$ be the columns of $K$ that $\gamma$ intersects in that order.
Then there are columns $E^*_1, \ldots, E^*_k \in \mathcal{E}_K$ such that:
\begin{enumerate}[label=(c\arabic*)]
\item $E^*_1 = E_1$ and $E^*_n = E_n$
\item $\dist^H_K (E^*_i, E_i) \leq 1$ for all $i = 1, \ldots, n$.
\item $E^*_1, \ldots, E^*_n$ are pairwise distinct and lie on the minimal chain between $E_1$ and $E_n$ in that order.
\item If $E^*_i \neq E_i$, then there are two walls $W, W' \subset \partial K$ which both intersect $\gamma$ twice such that $\gamma$ exits $K$ through $W'$ right after $E_{i-1}$, enters $K$ through $W'$ right before $E_i$, exits $K$ through $W$ right after $E_i$ and enters $K$ through $W$ right before $E_{i+1}$.
In particular $E_i$ does not lie on the minimal chain between $E_1, E_n$ and $E_i$ is not adjacent to $E_{i-1}$ or $E_{i+1}$.
\item If $E_i, E_{i+1}$ are adjacent, then $\gamma$ stays within $E_i \cup E_{i+1}$ between $E_i$ and $E_{i+1}$.
\item If $E_i, E_{i+1}$ are not adjacent, then there is a wall $W \subset \partial K$ such that $\gamma$ exits $K$ through $W$ right after $E_i$ and enters $K$ through $W$ right before $E_{i+1}$.
The columns $E_i, E^*_i, E^*_{i+1}, E_{i+1}$ lie in that order (some of these columns might be the same) on a minimal chain which runs along $W$.
\item If $i_1 < i_2$ and $E_{i_1}, E_{i_2}$ are adjacent to a common wall $W \subset \partial K$, then either $(E_{i_1}, \ldots, E_{i_2})$ form a minimal chain or $i_2 = i_1 + 1$ and $\gamma$ intersects $W$ right after $E_{i_1}$ and right before $E_{i_2}$.
\end{enumerate}
\end{enumerate}
\end{Proposition}

\begin{proof}
The proof uses induction on the combinatorial length $|\gamma|$ of $\gamma$.
The case $|\gamma| = 0$ is obvious, so assume that $|\gamma| > 0$ and that all assertions of the Proposition hold for all combinatorially shorter minimizing curves.

Let $W \in \mathcal{W}$ be a wall and $K, K' \in \mathcal{K}$ the chambers which are adjacent to $W$ from either side.
We first check the first statement of assertion (b).
Assume that $\gamma$ intersects $W$ three times or more.
Then by assertion (b2) of the induction hypothesis applied to every subsegment of $\gamma$ which intersects $W$ exactly twice, we obtain that $\gamma$ stays within $K \cup K'$ between its first and last intersection with $W$.
This however contradicts Lemma \ref{Lem:gammainsidetwoKKs}.

Assertion (b2) follows similarly.
Assume that $\gamma$ intersects $W$ twice and that both endpoints lie on the same side of $W$ as $K$.
By assertion (b1) of the induction hypothesis applied to each subsegment of $\gamma$ which intersects $W$ exactly once, we obtain again that $\gamma$ stays within $K \cup K'$ between its first and second intersection with $W$.
The rest follows with Lemma \ref{Lem:gammainsidetwoKKs}(b).

For assertion (b1), observe that the complete assertion (c) holds for $\gamma$ in the case in which $\gamma$ crosses a wall exactly once, because in this case assertion (c) is only concerned with proper subsegments of $\gamma$.
So consider the columns $E_1, \ldots, E_n, E^*_1, \ldots, E^*_n \in \mathcal{E}_K$.
Without loss of generality, we may assume that $\gamma(1)$ lies on the same side of $W$ as $K$.
This implies that $E_1$ is adjacent to $W$.
If $E_i$ for some $2 \leq i \leq n$ was adjacent to $W$ as well, then by assertion (c7) the curve $\gamma$ must be contained in $K$ in between $E_1$ and $E_i$.
This is however impossible by Lemma \ref{Lem:gammainsidetwoKKs}(a) applied to the subsegment of $\gamma$ between the last column in $K'$ and $E_i$.
So the first part of (b1) holds.
Consider now the minimal chain between $E_1$ and some $E_i$ and assume that three of its columns are adjacent to $W$.
Those columns need to be the first three columns in this chain.
We can assume that $E_i = E_n$, because otherwise we could pass to a subsegment of $\gamma$.
Since by what we have already shown, $\gamma$ cannot intersect any column which is adjacent to $W$ other than $E_1$, it cannot happen that $E_2$ is adjacent to $E_1$ (compare with assertions (c4), (c5)).
So by assertion (c6) there is a wall $W' \subset \partial K$, $W' \neq W$ that is intersected twice by $\gamma$ in between $E_1$ and $E_2$.
So if $H$ is sufficiently large, assertion (b2) implies that the first three columns on the minimal chain between $E_1$ and $E^*_2$, i.e. the first three columns on the minimal chain between $E_1$ and $E_n$ are adjacent to both $W'$ and $W$.
This however contradicts Lemma \ref{Lem:EEinKisatree} and finishes the proof of assertion (b1).

We no show assertions (c1)--(c7), (a) and (b3).
Observe that by the induction hypothesis, it suffices to restrict our attention to the case in which $\gamma(0), \gamma(1) \in K$.
Consider now the columns $E_1, \ldots, E_n$ as defined in the proposition.
If $n \leq 2$, we are done with the help of assertion (b) for $E^*_1 = E_1$ and $E^*_2 = E_2$, assuming $H > 2$.
So assume that $n \geq 3$.
Assertion (c5) and the first part of (c6) follows immediately by passing to the subsegment between $E_i, E_{i+1}$ and using the induction hypothesis.
We will now distinguish the cases on when $E_{n-1}$ lies on the minimal chain between $E_1, E_n$ or not and establish assertions (c1)--(c4) and the second part of assertion (c6) in each case.
Based on these assertions we then conclude assertion (c7) in both cases.

Consider first the case in which $E_{n-1}$ lies on the minimal chain between $E_1$ and $E_n$.
Then we can apply the induction hypothesis to the subsegment of $\gamma$ between $E_1$ and $E_{n-1}$ and obtain the columns $E^*_1, \ldots, E^*_{n-1}$ on the minimal chain between $E_1$ and $E_{n-1}$.
Moreover, we set $E^*_n = E_n$.
Assertions (c1)--(c6) follow immediately.

Next consider the case in which $E_{n-1}$ does not lie on the minimal chain between $E_1$ and $E_n$.
Define $E^*_1, \ldots, E^*_{n-2}$ using the induction hypothesis applied to the subsegment of $\gamma$ between $E_1$ and $E_{n-1}$.

Assume first that $E_n$ and $E_{n-1}$ are adjacent.
Then $E_n$ must lie on the minimal chain between $E_1$ and $E_{n-1}$ (by our assumption and the tree property).
So $E_{n-2}$ cannot be adjacent to $E_{n-1}$, because that would imply by assertion (c4) of the induction hypothesis that $E_{n-2} = E_n$ and it is elementary that $\gamma$ cannot reenter a column without exiting $K$.
This means (by assertion (c6) of the induction hypothesis) that there is a wall $W \subset \partial K$ which intersects $\gamma$ twice and which is adjacent to $E_{n-2}, E^*_{n-2}, E_{n-1}$ and hence also $E_n$.
This contradicts assertion (b2).

So $E_n$ and $E_{n-1}$ are not adjacent and by assertion (b2) both columns are adjacent to a wall $W \subset \partial K$ such that $\gamma$ intersects $W$ right after $E_{n-1}$ and right before $E_n$.
By the tree property of $\mathcal{E}_K$ there is a column $E^* \in \mathcal{E}_K$ which lies on the three minimal chains between $E_{n-1}, E_n$ and $E_1, E_{n-1}$ and $E_1, E_n$.
So $E^*$ is adjacent to $W$ and horizontally lies between $E_{n-1}, E_n$.
By our earlier assumption $E^* \neq E_{n-1}$.
Assertion (b1) applied to a subsegment of $\gamma$ implies that $\dist^H_K (E^*, E_{n-1}) \leq 1$; so $E^*$ is adjacent to $E_{n-2}$.
If $E_{n-2}$ was adjacent to $E_{n-1}$, then $E_{n-2} = E^*$ contradicting assertion (b2).
So by assertion (c6) of the induction hypothesis $E_{n-2}, E^*_{n-2}, E_{n-1}$ are adjacent to a wall $W' \subset \partial K$ such that $\gamma$ intersects $W'$ twice between $E_{n-2}, E_{n-1}$.
This implies $W \neq W'$.
Set $E^*_{n-1} = E^*$ and $E^*_n = E_n$.
Assertions (c1)--(c3) follow immediately.
Assertion (c4) and the second part of (c6) hold with the walls $W, W'$ that we have just defined.

We now establish assertion (c7) in the general case (i.e. independently on whether $E_{n-1}$ lies on the minimal chain between $E_1$ and $E_n$ or not).
Assume that $(E_{i_1}, \ldots, E_{i_2})$ does not form a minimal chain.
Then $\gamma$ has to leave $K$ in between $E_{i_1}$ and $E_{i_2}$, i.e. by assertion (c5) there is a $j \in \{ i_1, \ldots, i_2 -1 \}$ such that $E_j, E_{j+1}$ are not adjacent and hence by assertions (c6) $\gamma$ has to intersect a wall $W' \subset \partial K$ in between $E_j$ and $E_{j+1}$.
The columns on the minimal chain between $E^*_j, E^*_{j+1}$ are adjacent to both $W$ and $W'$ and for $H > 10$ there are at least $3$ such columns.
So $W = W'$ and by assertion (b2) we must have $E_{i_1} = E^*_j$, $E_{i_2} = E^*_{j+1}$.

Finally, assertion (a) is a direct consequence of assertion (c7) and assertion (b3) follows from assertion (c7) and Lemma \ref{Lem:gammainsidetwoKKs}(c).
\end{proof}

Next, we analyze the relative behavior of two combinatorially minimizing curves.

\begin{Lemma} \label{Lem:distg1g2afterW}
There are constants $\eta^* > 0$ and $H^* < \infty$ such that if $\eta \leq \eta^*$ and $H \geq H^*$, then the following holds:

Let $\gamma_1, \gamma_2 : [0,1] \to \td{M}$ be two combinatorially minimizing curves and consider a wall $W \in \mathcal{W}$ which is adjacent to two chambers $K, K' \in \mathcal{K}$ on either side.
Assume that $\gamma_1, \gamma_2$ intersect $W$ exactly once and that $\gamma_1(0), \gamma_2(0)$ lie in a common chamber on the same side of $W$ as $K$.
If that chamber is $K$, we additionally require that the cells which contain these points are vertically aligned.
Similarly, assume that $\gamma_1(1), \gamma_2(1)$ lie in a common chamber on the same side of $W$ as $K'$.
If that chamber is $K'$, we also require that the cells which contain these points are vertically aligned.

Let $Q_1, Q_2 \subset K$ be the cells which $\gamma_1, \gamma_2$ intersect right before crossing $W$ and let $Q'_1, Q'_2 \subset K'$ be the cells which $\gamma_1, \gamma_2$ intersect right after crossing $W$.
Then every pair of the cells $Q_1, Q_2, Q'_1, Q'_2$ has combinatorial distance bounded by $4$ or  $4 + H$ depending on whether they lie on the same side of $W$ or not.
\end{Lemma}

\begin{proof}
Let $E_1, E_2 \in \mathcal{E}_K$ be the columns that contain $Q_1, Q_2$.
We first show that $\dist^H_K (E_1, E_2) \leq 3$ (in fact, we can show that $\dist^H_K (E_1, E_2) \leq 1$, but we don't need this result here).

Define $E^*_1, E^*_2 \in \mathcal{E}_K$ to be the first columns in $K$ that are intersected by $\gamma_1, \gamma_2$.
In the case $\gamma_1 (0), \gamma_2 (0) \in K$ we have $E^*_1 = E^*_2$.
So in either case, we can find a wall $W^* \subset \partial K$ with $W^* \neq W$ which is adjacent to both $E^*_1, E^*_2$.
Consider the minimal chain between $E_1, E^*_1$ and let $E^{**}_1$ be the last column on that chain that is adjacent to $W$.
Define $E^{**}_2$ accordingly.
By Proposition \ref{Prop:characterizationminimizinggamma}(b1) $\dist^H_K (E_1, E^{**}_1), \dist^H_K(E_2, E^{**}_2) \leq 1$.
We need to show that $\dist^H_K (E^{**}_1, E^{**}_2) \leq 1$.

If $W, W^*$ are adjacent to a common column, then both $E^{**}_1, E^{**}_2$ have to be adjacent to $W^*$ since in that case a minimal chain between $E_1, E^*_1$ first runs along $W$ and then along $W^*$.
Hence in that case $\dist^H_K (E^{**}_1, E^{**}_2) \leq 1$ by Lemma \ref{Lem:EEinKisatree}.
If $W, W^*$ are not adjacent to a common column, we follow the minimal chain between $E^{**}_1, E^*_1$, then the minimal chain between $E^*_1, E^*_2$ (along $W^*$) and finally the minimal chain between $E^*_2, E^{**}_2$, to obtain a chain which connects $E^{**}_1$ with $E^{**}_2$ and which intersects $W$ only in its first and last column.
By the tree property of $\mathcal{E}_K$ this chain covers the minimal chain between $E^{**}_1, E^{**}_2$ and hence it has to include all columns along $W$ between $E_1^{**}, E_2^{**}$.
So $\dist^H_K (E^{**}_1, E^{**}_2) \leq 1$.

It follows that $\dist^H_K (Q_1, Q_2) = \dist^H_K (E_1, E_2) \leq 3$.
Analogously $\dist^H_K (Q'_1, Q'_2) \leq 3$.
It now follows from Lemma \ref{Lem:QQestimatesatW}(d) that $\dist^V_K (Q_1, Q_2), \dist^V_K (Q'_1, Q'_2) < C_0$.
This establishes the claim for $\eta < C_0^{-1}$.
\end{proof}

\begin{Lemma} \label{Lem:interchangeorderinK}
There are constants $\eta^* > 0$ and $H^*, C_1 < \infty$ such that if $\eta \leq \eta^*$ and $H \geq H^*$, the following holds: 

Let $K \in \mathcal{K}$ be a chamber and $Q_1, \ov{Q}_1, Q_2, \ov{Q}_2 \subset K$ be cells such that $Q_1, \ov{Q}_1$ and $Q_2, \ov{Q}_2$ are vertically aligned in $K$.
Assume that the vertical order of $Q_1, \ov{Q}_1$ is opposite to the one of $Q_2, \ov{Q}_2$ (i.e. $Q_1$ is ``above'' $\ov{Q}_1$ and $Q_2$ is ``below'' $\ov{Q}_2$ or the other way round).
Let $\gamma, \ov\gamma : [0,1] \to \td{M}$ be minimizing curves from $Q_1$ to $Q_2$ and from $\ov{Q}_1$ to $\ov{Q}_2$.
Then we can find cells $Q', \ov{Q}' \subset K$ such that $Q'$ intersects $\gamma$, $\ov{Q}'$ intersects $\ov\gamma$ and such that $\dist^H_K (Q', \ov{Q}') < 3H$, $\dist^V_K (Q', \ov{Q}') < C_1 H$ and $\dist ( Q', \ov{Q}' ) < 4H$.
\end{Lemma}

\begin{proof}
Note that the last inequality follows from the first two inequalities if $\eta^* < C_1^{-1}H^*$.

Let $E_0, E_\omega \in \mathcal{E}_K$ be the columns that contain $Q_1, \ov{Q}_1$ and $Q_2, \ov{Q}_2$.
We first invoke Proposition \ref{Prop:characterizationminimizinggamma}(c) on $\gamma$ to obtain columns $E_1, \ldots, E_n$, $E^*_1, \ldots, E^*_n \in \mathcal{E}_K$ with $E_1 = E^*_1 = E_0$ and $E_n = E^*_n = E_\omega$.
Then $E^*_1, \ldots, E^*_n$ lie on the minimal chain $L$ between $E_0$ and $E_\omega$.
Let $S \subset L \cup E_1 \cup \ldots \cup E_n$ be the union of all cells in $K$ which intersect $\gamma$ and all cells in $L \cup E_1 \cup \ldots \cup E_n$ and which are adjacent to cells outside $K$ that intersect $\gamma$.
Then $Q_1, Q_2 \subset S$ and it is not difficult to see that these two cells lie in the same connected component of $S$.
Based on the set $S$ we construct another set $S' \subset L$ in the following way: $S'$ is the union of $S \cap L$ with all cells in each $E^*_i$ which are horizontally aligned with a cell in $S \cap E_i$.
Then again $Q_1, Q_2 \subset S'$ and both cells lie in the same connected component of $S'$.
Similarly, we can invoke Proposition \ref{Prop:characterizationminimizinggamma}(c) on $\ov\gamma$, obtaining columns $\ov{E}_1, \ldots, \ov{E}_{\ov{n}} \in \mathcal{E}_K$ and $\ov{E}^*_1, \ldots, \ov{E}^*_{\ov{n}}$ on $L$ and we can define $\ov{S}$ and $\ov{S}'$ in the same way.
So $\ov{Q}_1, \ov{Q}_2 \subset \ov{S}'$ and both cells lie in the same connected component of $\ov{S}'$.
Since the cells on $L$ are arranged on a rectangular lattice and the cells $Q_1, \ov{Q}_1$ and $Q_2, \ov{Q}_2$ lie on opposite sides of $L$ and have opposite vertical order, we conclude that the sets $\ov{S}$ and $\ov{S}'$ have to intersect.
Let $Q^\circ \subset S' \cap \ov{S}'$ be a cell in the intersection and $E^\circ \subset L$ the column containing $Q^\circ$.
So we can find cells $Q, \ov{Q} \in \mathcal{Q}$ which intersect $\gamma, \ov\gamma$ such that the following holds:
Either $Q \subset L$ and $Q^\circ = Q$, or $Q \subset (E_1 \cup \ldots \cup E_n) \setminus \Int L$ and $Q$ is adjacent and horizontally aligned with $Q^\circ$, or $Q \not\subset K$ and $Q$ is either adjacent to $Q^\circ$ or $Q^\circ \subset E^*_i$ for some $i \in \{ 1, \ldots, n \}$ for which $E^*_i \neq E_i$ and $Q$ is adjacent to a cell in $E_i$ which is adjacent to $Q^\circ$ and horizontally aligned with it.
In the first two cases we set $Q' := Q$.
In the third case we will define $Q'$ later.
So if $Q \subset K$, then $Q'$ intersects $\gamma$ and $\dist^H_K (Q', Q^\circ) \leq 1$ and $\dist^V_K (Q', Q^\circ) = 0$.
The analogous characterization holds for $\ov{Q}$ and we define $\ov{Q}'$ in the same way if $\ov{Q} \subset K$.

Next, we consider the case in which $\dist^H_K (E^\circ, E^*_i) \leq 1$ for some $i \in \{ 1, \ldots, n \}$, and we establish the existence of a cell $Q' \subset K$ which intersects $\gamma$ and which is within bounded distance from $Q^\circ$.
If $Q \subset K$, then we are done by the previous paragraph.
So assume that $Q \not\subset K$.
Let $K' \in \mathcal{K}$ be the chamber which contains $Q$ and let $W = K \cap K' \in \mathcal{W}$ be the wall between $K$ and $K'$.
So $\gamma$ intersects $W$ twice and $E^\circ$ is adjacent to $W$.
Choose $i' \in \{1, \ldots, n-1 \}$ such that $\gamma$ intersects $W$ between $E_{i'}, E_{i'+1}$.
If $E^\circ$ lies between $E^*_{i'}, E^*_{i'+1}$, then $\dist^H_K (E^\circ, E^*_{i'}) \leq 1$ or $\dist^H_K (E^\circ, E^*_{i'+1}) \leq 1$, by our initial assumption.
If $E^\circ$ lies on $L$ not between $E^*_{i'}, E^*_{i'+1}$, then we can conclude by applying Proposition \ref{Prop:characterizationminimizinggamma}(b1) to subsegments of $\gamma$ that intersect $W$ exactly once, that we also have $\dist^H_K (E^\circ, E^*_{i'}) \leq 1$ or $\dist^H_K (E^\circ, E^*_{i'+1}) \leq 1$.
So we may assume that $i = i'$ or $i = i'+1$.
Let now $Q' \subset E_i$ be the cell that is intersected by $\gamma$ right before or right after $W$.
Then $\dist^H_K (Q^\circ, Q') \leq 2$ and by Lemma \ref{Lem:QQestimatesatW}(b) we get $\dist^V_K (Q^\circ, Q') < 3C_0 + 1$.

Combining the previous conclusion with the analogous conclusion for $\ov\gamma$ and the triangle inequality yields the desired result in the case in which there are indices $i \in \{ 1, \ldots, n \}$ and $\ov{i} \in \{ 1, \ldots, \ov{n} \}$ such that $\dist^H_K (E^\circ, E^*_i) \leq 1$ and $\dist^H_K (E^\circ, \ov{E}^*_{\ov{i}}) \leq 1$.
So, after possibly interchanging the roles of $\gamma$ and $\ov\gamma$, it remains to consider the case in which there is an index $i \in \{ 1, \ldots, n-1 \}$ such that $E^\circ$ lies strictly in between $E^*_i, E^*_{i+1}$ and such that $E^\circ$ is not adjacent to either of these columns.
We will henceforth always assume that.
Let $W \subset \partial K$ be the wall that $\gamma$ intersects between $E_i, E_{i+1}$ and let $K' \in \mathcal{K}$ be the chamber on the other side.
Then $E_i, E^*_i, E^\circ, E^*_{i+1}, E_{i+1}$ are arranged along $W$ in that order and by Lemma \ref{Lem:EEinKisatree} we must have $Q \subset K'$; let $E \in \mathcal{E}_{K'}$ be the column that contains $Q$.
Finally, let $Q' \subset E_i$ be the cell that $\gamma$ intersects right before $W$.

Consider the columns on $L$ between $\ov{E}^*_{\ov{i}}, \ov{E}^*_{\ov{i}+1}$ for each $\ov{i} = 1, \ldots, \ov{n}-1$.
If for some $\ov{i}$ there are at least $3$ such columns which are also in between $E^*_i$ and $E^*_{i+1}$, we must have $\dist^H_K (\ov{E}^*_{\ov{i}}, \ov{E}^*_{\ov{i}+1}) > 1$ and all columns between $\ov{E}^*_{\ov{i}}$ and $\ov{E}^*_{\ov{i}+1}$ have to be adjacent to $W$ by Lemma \ref{Lem:EEinKisatree}.
However, this can only happen for at most one index $\ov{i}$.
So either there is no such $\ov{i}$ and hence all columns which are strictly in between $E^*_i$ and $E^*_{i+1}$ are contained in $\ov{E}^*_1 \cup \ldots \cup \ov{E}^*_{\ov{n}}$ or there is exactly one such $\ov{i}$ and all columns which are strictly in between $E^*_i$ and $E^*_{i+1}$ lie in between $\ov{E}^*_{\ov{i}}$ and $\ov{E}^*_{\ov{i}+1}$.
In the first case, $\ov\gamma$ intersects all columns which are strictly in between $E^*_i$ and $E^*_{i+1}$.
In the second case, we can apply the same argument reversing the roles of $\gamma$ and $\ov\gamma$ to conclude that there is no other index $i' \in \{ 1, \ldots, n-1 \}$, $i' \neq i$ such that there are more than $2$ columns which are between $E^*_{i'}, E^*_{i'+1}$ and $\ov{E}^*_{\ov{i}}, \ov{E}^*_{\ov{i}+1}$.
This implies that $\dist^H_K (E^*_i, \ov{E}^*_{\ov{i}}), \dist^H_K (E^*_{i+1}, \ov{E}^*_{\ov{i}+1}) \leq 1$ in the second case.

In the first case, we use Proposition \ref{Prop:characterizationminimizinggamma}(b3) and (c7) to find that $\ov\gamma$ intersects less than $3H$ columns which are adjacent to $W$.
So $\dist^H_K (\ov{E}^*_{\ov{i}}, \ov{E}^*_{\ov{i}+1}) \leq 3H$.
Since $\ov\gamma$ intersects $E^\circ$ we have $\ov{Q}' := \ov{Q} = Q^\circ$.
Hence $\dist^H_K (Q', \ov{Q}') < 3H$ and Lemma \ref{Lem:QQestimatesatW}(b) yields $\dist^V_K (Q', \ov{Q}') < 3C_0H + C_0$ and we are done.

In the second case, $E^\circ$ lies strictly in between $\ov{E}^*_{\ov{i}}, \ov{E}^*_{\ov{i}+1}$.
So $\ov{Q} \not\subset K$ and $\ov\gamma$ intersects $W$ between $\ov{E}_{\ov{i}}, \ov{E}_{\ov{i}+1}$ (by Lemma \ref{Lem:EEinKisatree}).
Let $\ov{K}' \in \mathcal{K}$ be the chamber that contains $\ov{Q}$ and $W = K \cap \ov{K}' \in \mathcal{W}$ the wall between $K$ and $\ov{K}'$.
We will now show that $\ov{K}' = K'$ and $\ov{W} = W$.
If not, then there must be an index $\ov{i}' \in \{ 1, \ldots, n-1 \}, \ov{i}' \neq \ov{i}$ such that $\ov{W}$ is adjacent to $\ov{E}^*_{\ov{i}'}, \ov{E}^*_{\ov{i}'+1}, E^\circ$.
If $\ov{i}' < \ov{i}$, then $\ov{W}$ is also adjacent to $E^*_i$ (which is then between $\ov{E}^*_{\ov{i}'}$ and $E^\circ$), contradicting Lemma \ref{Lem:EEinKisatree}.
If $\ov{i}' > \ov{i}$, then $\ov{W}$ is also adjacent to $E^*_{i+1}$, contradicting Lemma \ref{Lem:EEinKisatree} as well.
So indeed $Q, \ov{Q} \subset \ov{K}' = K'$; let $\ov{E} \in \mathcal{E}_{K'}$ be the column that contains $\ov{Q}$.
Let now $\ov{Q}' \subset \ov{E}_{\ov{i}}$ be the cell that $\ov\gamma$ intersects right before $W$.

Recall that $Q' \subset E_i$ and $\ov{Q}'\subset \ov{E}_{\ov{i}}$, that $Q'$ is adjacent to $E$, $\ov{Q}'$ is adjacent to $\ov{E}$ and that $E, \ov{E}$ are both adjacent to $Q^\circ$.
Moreover, by our previous conclusions $\dist^H_K (Q', \ov{Q}') \leq 3$.
Let $Q'' \subset E_i$ be a cell that is adjacent to $\ov{E}$.
Then by Lemma \ref{Lem:QQestimatesatW}(b) $\dist^V_K (\ov{Q}', Q'') < 4C_0$ and by Lemma \ref{Lem:QQestimatesatW}(c) $\dist^V_K (Q'', Q') < C_0$.
Hence $\dist^V_K (Q', \ov{Q}') < 5C_0$.
This finishes the proof of the Lemma.
\end{proof}

The next Lemma is a preparation for the combinatorial convexity estimate stated in Proposition \ref{Prop:combinatorialconvexityseveralSeifert}. 

\begin{Lemma} \label{Lem:combconvexseveralSeifertpreparation}
There are constants $\eta^* > 0$ and $H^* < \infty$ such that if $\eta \leq \eta^*$ and $H \geq H^*$, then the following holds:

Let $K \in \mathcal{K}$ be a chamber and $Q_0, Q_1, Q_2 \subset K$ cells such that $Q_1$ and $Q_2$ are vertically aligned.
Assume that $\dist (Q_0, Q_1), \dist(Q_0, Q_2) \leq R$ for some $R \geq 0$.
Then for any cell $Q^* \subset K$ between $Q_1$ and $Q_2$, we have $\dist (Q_0, Q^*) < R + 8H$.
\end{Lemma}

\begin{proof}
We prove this Lemma by induction on $R$ (observe that we are only interested in a discrete set of values of $R$) and then on $\dist^V_K(Q_1,Q_2)$.
Consider the action $\varphi : \IZ \curvearrowright \td{M}$ by deck transformations of the universal covering $\td{M} \to M$ which acts as a vertical shift on $K$, leaving $\td{V}$ and hence the cell structure and combinatorial distance function invariant and choose $z \in \IZ$ such that $Q^* = \varphi_z (Q_1)$.
We may assume $z \neq 0$.

Without loss of generality, we can assume that $Q^*$ lies between $Q_1$ and $Q^{**} =\varphi_{-z} (Q_2)$.
Otherwise, we can interchange the roles of $Q_1$ and $Q_2$.
Let $\gamma_1, \gamma_2$ be minimizing curves between $Q_0$ and $Q_1, Q_2$.
We can now apply Lemma \ref{Lem:interchangeorderinK} to $\varphi_z \circ \gamma_1$ and $\gamma_2$ to obtain cells $Q'_1, Q'_2 \subset K$ on $\varphi_z \circ \gamma_1$ and $\gamma_2$ with $\dist (Q'_1, Q'_2) < 4 H$.
Then $\varphi_{-z} (Q'_1)$ lies on $\gamma_1$ and hence
\begin{equation} \label{eq:Q0phiQ1}
 \dist (Q_0, \varphi_{-z} (Q'_1) ) + \dist (\varphi_{-z} (Q'_1), Q_1) = \dist (Q_0, Q_1) \leq R.
\end{equation}
We also have
\begin{equation} \label{eq:Q0Qs2Q2}
 \dist (Q_0, Q'_2) + \dist (Q'_2, Q_2) = \dist (Q_0, Q_2) \leq R.
\end{equation}
If $\dist (Q_0, Q'_2) + \dist (\varphi_{-z} (Q'_1), Q_1) \leq R + 4H$, then
\[ \dist (Q_0, Q^*) \leq \dist(Q_0, Q'_2) + \dist (Q'_2, Q'_1) + \dist (Q'_1, \varphi_z (Q_1) )
< R + 8H, \]
which proves the desired estimate.
On the other hand, assume that $\dist (Q_0, Q'_2) + \dist (\varphi_{-z} (Q'_1), Q_1) > R + 4H$.
Then (\ref{eq:Q0phiQ1}) and (\ref{eq:Q0Qs2Q2}) give us
\[ \dist (Q_0, \varphi_{-z} (Q'_1)) + \dist (Q'_2, Q_2) < R - 4H . \]
It follows that
\begin{multline*}
 \dist (\varphi_{-z} (Q'_1), Q^{**}) = \dist (Q'_1, Q_2) \leq \dist (Q'_1, Q'_2) + \dist(Q'_2, Q_2) \\
 < 4 H + R - 4H - \dist (Q_0, \varphi_{-z} (Q'_1)) = R - \dist (Q_0, \varphi_{-z} (Q'_1)).
 \end{multline*}
Also by (\ref{eq:Q0phiQ1})
\[ \dist (\varphi_{-z} (Q'_1), Q_1) \leq R - \dist (Q_0, \varphi_{-z} (Q'_1)). \]
So by the induction hypothesis, we find that
\[ \dist (\varphi_{-z}  (Q'_1), Q^*) < R - \dist (Q_0, \varphi_{-z} (Q'_1)) + 8H. \]
This implies
\[ \dist (Q_0, Q^*) \leq \dist (Q_0, \varphi_{-z} (Q'_1)) + \dist (\varphi_{-z} (Q'_1), Q^*) < R + 8 H. \qedhere \]
\end{proof}

\begin{Proposition} \label{Prop:combinatorialconvexityseveralSeifert}
Assume that $M$ satisfies condition (C).
There are constants $\eta^* > 0$ and $H^* < \infty$ such that whenever $\eta \leq \eta^*$ and $H \geq H^*$, then the following holds:

Consider a cell $Q_0 \in \mathcal{Q}$, a chamber $K \in \mathcal{K}$ (not necessarily containing $Q_0$) and cells $Q_1, Q_2 \subset K$ which are vertically aligned within $K$.
Assume that $\dist (Q_0, Q_1)$, $\dist(Q_0, Q_2) \leq R$ for some $R \geq 0$.
Then for any cell $Q^* \subset K$ which is vertically aligned with $Q_1, Q_2$ and vertically between $Q_1$ and $Q_2$, we have $\dist (Q_0, Q^*) < R + 10 H$.
\end{Proposition}

\begin{proof}
If $Q_0 \subset K$, we are done by the previous Lemma.
So assume that $Q_0$ lies outside of $K$ and let $\gamma_1, \gamma_2$ be minimizing curves from $Q_0$ to $Q_1, Q_2$.

Then there is a unique wall $W \subset \partial K$ through which both $\gamma_1$ and $\gamma_2$ enter $K$.
Let $Q'_1, Q'_2 \subset K$ be the first cells in $K$ which are intersected by $\gamma_1, \gamma_2$.
So both cells are adjacent to $W$.
By Lemma \ref{Lem:distg1g2afterW} we know that $\dist(Q'_1, Q'_2) \leq 4$.

So
\[ \dist (Q'_1, Q_1) \leq R - \dist (Q_0, Q'_1) \]
and
\begin{multline*}
 \dist (Q'_1, Q_2) \leq \dist (Q'_1, Q'_2) + \dist (Q'_2, Q_2) 
 \leq 4 + R - \dist (Q_0, Q'_2) \\
 \leq 4 + R - \dist(Q_0, Q'_1) + \dist(Q'_1, Q'_2)  \leq R + 8 - \dist (Q_0, Q'_1) .
\end{multline*}
We can no apply Lemma \ref{Prop:combinatorialconvexityseveralSeifert} to obtain
\[ \dist(Q'_1, Q^*) < R + 8 - \dist(Q_0, Q'_1) + 8H . \]
So $\dist(Q_0, Q^*) < R + 10 H$ for $H > 4$.
\end{proof}

\subsection{A combinatorial convexity estimate if $M$ satisfies condition (B)} \label{subsec:combconvexincaseB}
Assume now that $M$ satisfies condition (B) in Proposition \ref{Prop:easiermaincombinatorialresultCaseb}, i.e. that $M$ is the total space of an $S^1$-bundle over a closed, orientable surface of genus $\geq 2$.
In this setting we will establish the same combinatorial convexity estimate as in Proposition \ref{Prop:combinatorialconvexityseveralSeifert}.
It will be stated in Proposition \ref{Prop:combinatorialconvexityCaseB}.
Its proof will resemble the proof in the previous subsection, except that most Lemmas will be simpler.

We first let $\mathcal{E} = \bigcup_{K \in \mathcal{K}} \mathcal{E}_K$ be the set of all columns of $\td{M}$.
We say that two columns $E_1, E_2 \in \mathcal{E}$ are \emph{adjacent} if they intersect in a point of $\td{V} \setminus \td{V}^{(1)}$.
In other words, $E_1, E_2$ are adjacent if and only if we can find cells $Q_1 \subset E_1$, $Q_2 \subset E_2$ such that $Q_1, Q_2$ are adjacent.
Observe that in the setting of condition (B) every column $E_1 \in \mathcal{E}$ is adjacent to only finitely many columns $E_2 \in \mathcal{E}$ and every wall $W \in \mathcal{W}$ intersects its adjacent columns $E \in \mathcal{E}$ from either side in parallel strips $E \cap W$ (see Lemma \ref{Lem:QQestimatesatW}).

The first Lemma is an analog of Lemma \ref{Lem:gammainsidetwoKKs}.

\begin{Lemma} \label{Lem:gamminKKsCaseB}
There are constants $\eta^* > 0$ and $H^* < \infty$ such that if $\eta \leq \eta^*$ and $H \geq H^*$, the following holds:

Consider two chambers $K, K' \in \mathcal{K}$ which are adjacent to one another across a wall $W = K \cap K' \in \mathcal{W}$.
Assume that $\gamma : [0,1] \to \td{M}$ is combinatorially minimizing and that it is fully contained in $K \cup K'$.
Then $\gamma$ intersects $W$ at most twice and $\gamma$ does not reenter any column, i.e. $\gamma^{-1} (E)$ is an interval for all $E \in \mathcal{E}$.

Consider first the case in which $\gamma$ intersects $W$ exactly once.
Then the columns on $\gamma$ which are adjacent to $W$ form two minimal chains in $K$ and $K'$, moving in the same direction, which are adjacent to one another in a unique pair of columns $E \in \mathcal{E}_K$ and $E' \in \mathcal{E}_{K'}$.

Consider now the case in which $\gamma$ intersects $W$ exactly twice and assume that $\gamma(0), \gamma(1) \in K$.
Let $E_1, E_2 \in \mathcal{E}_K$ be the columns that $\gamma$ intersects right before and after $W$.
Then $\gamma$ does not intersect any column of $K$ which is adjacent to $W$ and which horizontally lies strictly between $E_1$ and $E_2$.
Moreover, $\dist^H_K(E_1, E_2) > H$.
\end{Lemma}

\begin{proof}
It is easy to see that every subsegment of $\gamma$, which does not intersect $W$ and whose endpoints lie in columns which are adjacent to $W$, stays within columns which are adjacent to $W$ and does not reenter any column.
So we can restrict our attention to the case in which $\gamma$ only intersects columns which are adjacent to $W$.

Assume first that $\gamma$ intersects $W$ exactly once and assume without loss of generality that $\gamma(0) \in K$.
Let $E_1, \ldots, E_n \in \mathcal{E}_K$ and $E'_1, \ldots, E'_{n'} \in \mathcal{E}_{K'}$ be the columns that $\gamma$ intersects in that order.
Then both sequences of columns form minimal chains which move along $W$ and $E_n, E'_1$ are adjacent across $W$.
We now show that $E_i$ can only be adjacent to $E'_{i'}$ if $i = n$ and $i' = 1$.
This will also imply that the directions of both minimal chains agrees.
Assume that this was not the case and assume without loss of generality that $E_i$ is adjacent to $E'_{i'}$ for some $i < n$ and $i' \geq 1$ (otherwise we reverse the orientation of $\gamma$).
Let $Q_1 \subset E_1$ be the cell that contains $\gamma(0)$, $Q_2 \subset E_n$, $Q_3 \subset E'_1$ the cells that $\gamma$ intersects right before and after $W$ and $Q_4 \subset E_{i'}$ a cell that intersects $\gamma$.
Choose moreover a cell $Q^* \subset E_i$ which is adjacent to $Q_4$.
By Lemma \ref{Lem:QQestimatesatW}(e)
\[ \dist^V_K (Q_2, Q^*) < \dist^V_{K'} (Q_3, Q_4) + C_0 \dist^H_K (Q_2, Q^*) + C_0. \]
So
\begin{multline*}
 \dist^V_K (Q_1, Q^*) \leq \dist^V_K (Q_1, Q_2) + \dist^V_K (Q_2, Q^*) \\
  < \dist^V_K (Q_1, Q_2) + \dist^V_{K'} (Q_3, Q_4) + C_0 \dist^H_K (Q_2, Q^*) + C_0.
\end{multline*}
Hence
\begin{multline*}
 \dist(Q_1, Q_4) \leq \dist(Q_1, Q^*) + H < \dist^H_K (Q_1, Q^*) + \eta \dist^V_K (Q_1, Q_2) \\ + \eta \dist^V_{K'} (Q_3, Q_4) + \eta C_0 \dist^H_K (Q_2, Q^*) + \eta C_0 + H.
\end{multline*}
On the other hand, the minimizing property of $\gamma$ yields
\begin{multline*}
 \dist(Q_1, Q_4) = \dist(Q_1, Q_2) + H + \dist (Q_3, Q_4) \\ \geq \dist^H_K (Q_1, Q_2) + \eta \dist^V_K (Q_1, Q_2) + \eta \dist^V_{K'} (Q_3, Q_4) + H
\end{multline*}
Combining both inequalities yields
\[ \dist^H_K (Q_1, Q_2) < \dist^H_K (Q_1, Q^*) + \eta C_0 \dist^H_K (Q_2, Q^*) + \eta C_0. \]
Since $\dist^H_K (Q_1, Q_2) = n-1$, $\dist^H_K (Q_1, Q^*) = i-1$ and $\dist^H_K (Q_2, Q^*) = n-i \geq 1$, we obtain
\[ n -1 < i - 1 + \eta C_0 (n-i) + \eta C_0. \]
This yields a contradiction if $\eta < (2C_0)^{-1}$.

Assume next that $\gamma$ intersects $W$ exactly twice and that $\gamma(0), \gamma(1) \in K$.
Define $E_1, E_2 \in \mathcal{E}_K$ as in the statement of the Lemma.
We now establish the bound $\dist^H_K (E_1, E_2) > H$ (for sufficiently small $\eta$ and large $H$).
It is then easy to see, by the previous conclusion applied to subsegments of $\gamma$, that $\gamma$ cannot intersect any column of $K$ which is adjacent to $W$ and lies strictly between $E_1, E_2$.
Let now $Q_1 \subset E_1$ and $Q_2 \subset E_2$ be the cells that $\gamma$ intersects right before and after $W$ and let $Q'_1, Q'_2 \subset K'$ be the cells that $\gamma$ intersects right after $Q_1$ and right before $Q_2$.
Then
\[ \dist(Q_1, Q_2) = 2H + \dist(Q'_1, Q'_2) = 2H + \dist^H_{K'} (Q'_1, Q'_2) + \eta \dist^V_{K'} (Q'_1, Q'_2). \]
By Lemma \ref{Lem:QQestimatesatW}(e)
\[  \dist^V_K (Q_1, Q_2) < \dist^V_{K'} (Q'_1, Q'_2) + C_0 \dist^H_K (Q_1, Q_2) + C_0. \]
So
\[ \dist(Q_1, Q_2) < \dist^H_K (Q_1, Q_2)  + \eta \dist^V_{K'} (Q'_1, Q'_2) + \eta C_0 \dist^H_K (Q_1, Q_2) + \eta C_0. \]
Hence
\[ 2 H < (1 + \eta C_0) \dist^H_K (Q_1, Q_2) + \eta C_0. \]
The result follows for $H > 10$ and $\eta < (2 C_0)^{-1}$.

We finally show that $\gamma$ cannot intersect $W$ more than twice.
Assume it does.
By passing to a subsegment and possibly interchanging the roles of $K$ and $K'$, we can assume that $\gamma$ intersects $W$ exactly three times and $\gamma(0) \in K$.
Let $Q_1, Q_2, Q_3 \subset K$ be the cells in $K$ that $\gamma$ intersects before the first and third and after the second intersection with $W$ and let $Q'_1, Q'_2, Q'_3 \subset K'$ be the cells of $K'$ that $\gamma$ intersects after the first and third and before the second intersection with $W$.
Then
\begin{multline} \label{eq:distQ1Qs3CaseB}
 \dist(Q_1, Q'_3) = 3H + \dist^H_{K'} (Q'_1, Q'_2) + \eta \dist^V_{K'} (Q'_1, Q'_2) \\
  + \dist^H_K (Q_2, Q_3) + \eta \dist^V_K (Q_2, Q_3).
\end{multline}

Let $Q^* \subset K$ be the cell which is adjacent to $W$ and which is located relatively to $Q_1, Q_2, Q_3$ such that $Q_1, Q_2, Q_3, Q^*$ forms a ``parallelogram'', i.e. 
\begin{multline*}
\dist^H_K (Q_1, Q^*) = \dist^H_K (Q_2, Q_3), \qquad \dist^V_K (Q_1, Q^*) = \dist^V_K (Q_2, Q_3), \\ 
\dist^H_K (Q^*, Q_3) = \dist^H_K (Q_1, Q_2), \qquad \dist^V_K (Q^*, Q_3) = \dist^V_K (Q_1, Q_2)
\end{multline*}
in an oriented sense.
Let moreover $Q^{* \prime} \subset K'$ be a cell which is adjacent to $Q^*$.
Then by Lemma \ref{Lem:QQestimatesatW}(g)
\[ \dist^H_{K'} (Q^{* \prime}, Q'_3 ) < \dist^H_{K'} (Q'_1, Q'_2) + C_0, \quad \dist^V_{K'} (Q^{* \prime}, Q'_3 ) < \dist^V_{K'} (Q'_1, Q'_2) + C_0. \]
So
\begin{multline*}
\dist(Q_1, Q'_3) \leq \dist(Q_1, Q^*) + \dist (Q^*, Q^{* \prime}) + \dist (Q^{* \prime}, Q'_3) \\
\leq \dist^H_K (Q_1, Q^*) + \eta \dist^V_K (Q_1, Q^*) + H + \dist^H_{K'} (Q^{* \prime}, Q'_3) + \eta \dist^V_{K'} (Q^{* \prime}, Q'_3) \\
< H + \dist^H_K (Q_2, Q_3) + \eta \dist^V_K (Q_2, Q_3) + \dist^H_{K'} (Q'_1, Q'_2) + C_0 \\ + \eta \dist^V_{K'} (Q'_1, Q'_2) + \eta C_0.
\end{multline*}
Together with (\ref{eq:distQ1Qs3CaseB}) this yields
\[ 2H < C_0 + \eta C_0 \]
and hence a contradiction for $H > C_0$ and $\eta < 1$.
\end{proof}

The following Proposition and its proof is similar to Proposition \ref{Prop:characterizationminimizinggamma}.

\begin{Proposition} \label{Prop:characteriztionminimizinggammaCaseB}
Assume that $M$ satisfies condition (B).
There are constants $\eta^* > 0$ and $H^* < \infty$ such that if $\eta \leq \eta^*$ and $H \geq H^*$, the following holds:

Consider a combinatorially minimizing curve $\gamma : [0,1] \to \td{M}$.
Then
\begin{enumerate}[label=(\alph*)]
\item For every column $E \in \mathcal{E}$, the preimage $\gamma^{-1} (E)$ is an interval.
\item $\gamma$ intersects every wall $W \in \mathcal{W}$ at most twice.
Assume that $K, K' \in \mathcal{K}$ are two chambers which are adjacent to a wall $W \in \mathcal{W}$ from either side.
Then
\begin{enumerate}[label=(b\arabic*)]
\item If $\gamma$ intersects $W$ exactly once then the following holds:
Assume that $\gamma(0)$ lies on the same side of $W$ as $K$.
Let $E \in \mathcal{E}_{K}$ be the first column which is intersected by $\gamma$ and which is adjacent to $W$.
Then for every column $E^* \in \mathcal{E}_K$ which $\gamma$ intersects before $E$, the minimal chain between $E^*$ and $E$ intersects $W$ in at most two columns.
\item If $\gamma$ intersects $W$ exactly twice and its endpoints lie on the same side of $W$ as $K$, then the columns $E_1, E_2 \in \mathcal{E}_K$ that $\gamma$ intersects right before and after $W$ satisfy $\dist^H_K (E_1, E_2) > H$.
Moreover, $\gamma$ does not intersect any column of $K$ which is adjacent to $W$ and which horizontally lies strictly between $E_1$ and $E_2$.
\item If $\gamma$ intersects two columns $E_1, E_2 \in \mathcal{E}$ which are both adjacent to $W$, then $\gamma$ stays within $K \cup K'$ in between $E_1, E_2$ and only intersects columns which are adjacent to $W$.
\end{enumerate}
\item Consider a chamber $K \in \mathcal{K}$ and let $E_1, \ldots, E_n \in \mathcal{E}_K$ be the columns of $K$ that $\gamma$ intersects in that order.
Then there are columns $E^*_1, \ldots, E^*_n \in \mathcal{E}_K$ such that assertions (c1)--(c6) of Proposition \ref{Prop:characterizationminimizinggamma} hold.
\end{enumerate}
\end{Proposition}

\begin{proof}
We use again induction on the combinatorial length $|\gamma|$ of $\gamma$.
Assume that $|\gamma| > 0$, since for $|\gamma|=0$ there is nothing to prove.
The first part of assertion (b) follows as in the proof or Proposition \ref{Prop:characterizationminimizinggamma}.

We now establish assertion (b1).
So assume that $\gamma$ intersects $W$ exactly once and that $\gamma(0)$ lies on the same side of $W$ as $K$ and consider the columns $E, E^* \in \mathcal{E}_K$.
Apply assertion (c) of the induction hypothesis to the subsegment of $\gamma$ between $E^*$ and $E$.
We obtain sequences $E_1, \ldots, E_n$ and $E^*_1, \ldots, E^*_n$ with $E_1 = E^*_1 = E^*$ and $E_n = E^*_n = E$.
If $E_{n-1}, E_n$ are adjacent, then $E_{n-1} = E^*_{n-1}$ lies on the minimal chain between $E^*$ and $E$ and by assumption $E_{n-1}$ cannot be adjacent to $W$; so we are done.
If $E_{n-1}, E_n$ are not adjacent, then $\gamma$ intersects a wall $W' \subset \partial K$, $W' \neq W$ twice between $E_{n-1}, E_n$.
All columns on the minimal chain between $E^*_{n-1}$ and $E_n$ are adjacent to $W'$.
By Lemma \ref{Lem:EEinKisatree} at most $2$ of those columns can also be adjacent to $W$.

Assertion (b2) follows immediately from assertion (b3) of the induction hypothesis and Lemma \ref{Lem:gamminKKsCaseB}.

Next, we establish assertions (c) and (a).
It suffices to consider the case in which $\gamma(0), \gamma(1) \in K$.
Let $E_1, \ldots, E_n \in \mathcal{E}_K$ be as defined in the proposition.
If $n \leq 2$, then we are done using assertion (b); so assume $n \geq 3$.
If $E_{n-1}$ lies on the minimal chain between $E_1$ and $E_n$, then we are done as in the proof of Proposition \ref{Prop:characterizationminimizinggamma}.
So assume that $E_{n-1}$ does not lie on the minimal chain between $E_1, E_n$.

We show that $E_{n-1}, E_n$ cannot be adjacent.
Otherwise, as in the proof of Proposition \ref{Prop:characterizationminimizinggamma}, $\gamma$ intersects a wall $W \subset \partial K$ twice between $E_{n-2}$ and $E_{n-1}$ and the columns $E_{n-2}, E^*_{n-2}, E_n, E_{n-1}$ lie along $W$ in that order.
This contradicts assertion (b2).

So there is a wall $W \subset \partial K$ which is adjacent to both $E_{n-1}, E_n$ and $\gamma$ crosses $W$ twice between those two columns.
We now proceed as in the proof of Proposition \ref{Prop:characterizationminimizinggamma}, but we have to be careful whenever we make use of assertion (b).
As in this proof, we can find a column $E^* \in \mathcal{E}_K$ which lies on the three minimizing chains between $E_{n-1}, E_n$ and $E_1, E_{n-1}$ and $E_1, E_n$ and $E^* \neq E_{n-1}$.
We also know that $E_{n-2}$ cannot be adjacent to to $E_{n-1}$, since otherwise it would lie on the minimal chain between $E_{n-1}$ and $E^*$ along $W$, in contradiction to assertion (b2).
So by assertion (c6) of the induction hypothesis $E_{n-2}, E^*_{n-2}, E_{n-1}$ are adjacent to a wall $W' \subset \partial K$ such that $\gamma$ intersects $W'$ twice between $E_{n-2}, E_{n-1}$.
This implies $W' \neq W$.
Now both $W$ and $W'$ are adjacent to all columns on the minimal chain between $E_{n-1}, E^*$ or between $E_{n-1}, E^*_{n-2}$, whichever is shorter.
So by Lemma \ref{Lem:EEinKisatree} we must have $\dist^H_K (E^*, E_{n-1}) = 1$.
Assertion (c1)--(c6) now follow as in the proof of Proposition \ref{Prop:characterizationminimizinggamma}.

Now for assertion (a), we may assume that $\gamma(0), \gamma(1) \in E \in \mathcal{E}_K$ in view of the induction hypothesis.
Assertion (c) now immediately implies that $\gamma$ is fully contained in $E$.

Finally, we establish assertion (b3).
In view of the induction hypothesis it suffices to consider the case in which $E_1, E_2 \in \mathcal{E}_K$ and in which $\gamma$ does not intersect $W$.
Apply now assertion (c) to obtain sequences $E'_1, \ldots, E'_n$ and $E^{\prime *}_1, \ldots, E^{\prime *}_n$ with $E^{\prime *}_1 = E'_1 = E_1$ and $E^{\prime *}_n = E'_n = E_2$.
It follows that all columns $E^*_1, \ldots, E^*_n$ are adjacent to $W$.
If $\gamma$ crossed a wall $W' \subset \partial K$ twice in between some $E'_i, E'_{i+1}$, then all columns between $E^{\prime *}_i, E^{\prime *}_{i+1}$ would be adjacent to $W'$ and $W$.
This is impossible by Lemma \ref{Lem:EEinKisatree}.

Finally, for assertion (a) observe that if $\gamma$ intersects $E \in \mathcal{E}_K$ twice, then by assertion (b3) it has to stay within $K \cup K'$ for every 
\end{proof}

The next Lemma is an analog of Lemma \ref{Lem:interchangeorderinK}.
Note that in the setting of condition (B), we don't need to work inside a single chamber.
This fact will later compensate us for the lack of an analog for Lemma \ref{Lem:distg1g2afterW}.

\begin{Lemma} \label{Lem:interchangeorderinKCaseB}
There are constants $\eta^* > 0$ and $H^* < \infty$ such that if $\eta \leq \eta^*$ and $H \geq H^*$,  the following holds:

Let $E^\circ_1, E^\circ_2 \in \mathcal{E}$ be two columns and $Q_1, \ov{Q}_1 \subset E^\circ_1$, $Q_2, \ov{Q}_2 \subset E^\circ_2$ cells such that the vertical orders of $Q_1, \ov{Q}_1$ and $Q_2, \ov{Q}_2$ are opposite to each other.
Let $\gamma, \ov\gamma : [0,1] \to \td{M}$ be minimizing curves from $Q_1$ to $Q_2$ and from $\ov{Q}_1$ to $\ov{Q}_2$.
Then we can find cells $Q', \ov{Q}' \in \mathcal{Q}$ which intersect $\gamma, \ov\gamma$ and such that $\dist(Q', \ov{Q}') < 3H$.
\end{Lemma}

\begin{proof}
Consider first a wall $W \in \mathcal{W}$ that intersects $\gamma$ (and hence also $\ov\gamma$) exactly once.
Let $K, K' \in \mathcal{K}$ be the cambers which are adjacent to $W$ from either side in such a way that $\gamma(0)$ and $\ov\gamma(0)$ lie on the same side of $W$ as $K$.
Let $E \in \mathcal{E}_K$ be the first column on $\gamma$ which is adjacent to $W$ and choose $\ov{E} \in \mathcal{E}_K$ analogously.
We argue similarly as in the proof of Lemma \ref{Lem:distg1g2afterW} that $\dist^H_K (E, \ov{E}) \leq 3$.
Let $E^* \in \mathcal{E}_K$ be the first column on $\gamma$ inside $K$ and define $\ov{E}^*$ analogously.
Then either $E^* = \ov{E}^* = E^\circ_1$ or $E^\circ_1 \not\in \mathcal{E}_K$.
In both cases there is a wall $W^* \subset \partial K$, $W^* \neq W$ which is adjacent to both $E^*$ and $\ov{E}^*$.
Let $E^{**} \in \mathcal{E}_K$ the last column on the minimal chain between $E$ and $E^*$ and define $\ov{E}^{**} \in \mathcal{E}_K$ analogously.
By Proposition \ref{Prop:characteriztionminimizinggammaCaseB}(b1) we have $\dist^H_K (E, E^{**}), \dist^H_K (\ov{E}, \ov{E}^{**}) \leq 1$.
It now follows as in the proof of Lemma \ref{Lem:distg1g2afterW} that $\dist^H_K (E^{**}, \ov{E}^{**}) \leq 1$ and hence $\dist^H_K (E, \ov{E}) \leq 3$ (observe that this part of the proof only makes use of the tree property of $\mathcal{E}_K$ from Lemma \ref{Lem:EEinKisatree}).

Let now $W_1, \ldots, W_h$ be all the walls that $\gamma$ intersects exactly once in this order.
Then also $\ov\gamma$ intersects each of these walls exactly once in this order.
For each $i = 1, \ldots, h$ let $E'_i \in \mathcal{E}_K$ be the first and $E''_i$ the last column on $\gamma$ which is adjacent to $W_i$.
Define $\ov{E}'_i$ and $\ov{E}''_i$ accordingly.
By the last paragraph, we obtain that $E'_i, \ov{E}'_i$ and $E''_i, \ov{E}''_i$ have horizontal distance $\leq 3$ in the chamber in which they are contained (the bound on the horizontal distance between $E''_i$ and $\ov{E}''_i$ can be obtained by reversing the orientation of $\gamma$ and $\ov\gamma$).
Choose cells $Q'_i \subset E'_i$, $Q''_i \subset E''_i$ resp. $\ov{Q}'_i \subset \ov{E}'_i$, $\ov{Q}''_i \subset \ov{E}''_i$ which intersect $\gamma$ resp. $\ov\gamma$.

We first the case in which there is some $i \in \{ 1, \ldots, h \}$ such that the vertical orders of $Q'_i, \ov{Q}'_i$ and $Q''_i, \ov{Q}''_i$ are different.
Observe that by Proposition \ref{Prop:characteriztionminimizinggammaCaseB}(b3) the curve $\gamma$ only intersects cells adjacent to $W_i$ between $Q'_i$ and $Q''_i$; the same is true for $\ov\gamma$.
Let $S \subset \td{M}$ be the union of all cells which $\gamma$ intersects between $Q'_i, Q''_i$ and define $\ov{S}$ accordingly.
It is not difficult to see, using Lemma \ref{Lem:gamminKKsCaseB}, that either $S \cap W_i$ and $\ov{S} \cap W_i$ intersect or there is a cell $Q' \in \mathcal{Q}$ on $\gamma$ with $\dist(Q', \ov{Q}'_i) \leq 3 + H$ or $\dist(Q', \ov{Q}''_i) \leq 3 + H$ or there is a cell $\ov{Q}' \in \mathcal{Q}$ on $\ov\gamma$ with $\dist(\ov{Q}', Q'_i) \leq 3+ H$ or $\dist(\ov{Q}', Q''_i) \leq 3+H$.
In all these cases we are done.

So assume from now on that the vertical orders of $Q'_i, \ov{Q}'_i$ and $Q''_i, \ov{Q}''_i$ are the same for all $i = 1, \ldots, h$.
Choose now $i \in \{ 1, \ldots, h \}$ minimal such that the vertical order of $Q'_i, \ov{Q}'_i$ differs from that of $Q_1, \ov{Q}_1$.
If there is no such $i$, then the vertical orders of $Q'_h, \ov{Q}'_h$ and $Q_2, \ov{Q}_2$ are opposite and we can get rid of this case by reversing the orientations of $\gamma$ and $\ov\gamma$.
Let $K \in \mathcal{K}$ be the chamber which contains $Q'_i, \ov{Q}'_i$.
If $i > 1$, the choice of $i$ implies that the vertical order $Q''_{i-1}, \ov{Q}''_{i-1}$ is different from that of $Q'_i, \ov{Q}'_i \subset K$.
If $i = 1$, then the vertical order of $Q_1, \ov{Q}_1 \subset K$ is different from that of $Q'_1, \ov{Q}'_1$.
Apply Proposition \ref{Prop:characteriztionminimizinggammaCaseB}(c) to the subsegment of $\gamma$ between $Q'_{i-1}$ or $Q_1$ and $Q''_i$ to obtain columns $E_1, \ldots, E_n$ and $E^*_1, \ldots, E^*_n \in \mathcal{E}_K$.
Similarly we obtain the columns $\ov{E}_1, \ldots, \ov{E}_{\ov{n}}$ and $\ov{E}^*_1, \ldots, \ov{E}^*_{\ov{n}} \in \mathcal{E}_K$ for the corresponding subsegment of $\ov\gamma$.
Note that $\dist^H_K (E_1, \ov{E}_1), \dist^H_K (E_n, \ov{E}_{\ov{n}}) \leq 3$.

If $\dist^H_K (E_1, E_n), \dist^H_K (\ov{E}_1, \ov{E}_{\ov{n}}) \leq 6$, then by Proposition \ref{Prop:characteriztionminimizinggammaCaseB}(c), all columns $E_i$ and $\ov{E}_i$ have distance $\leq 17$ from one another and hence we can just pick cells $Q', \ov{Q}'$ which are horizontally aligned to show the Lemma.
So assume from now on that this is not the case and let $L$ and $\ov{L}$ be the minimal chains between $E_1, E_n$ and $\ov{E}_1, \ov{E}_{\ov{n}}$.
By the tree property as explained in Lemma \ref{Lem:EEinKisatree}, $L$ and $\ov{L}$ intersect in a minimal chain $L^\circ$ such that every column on $(L \cup \ov{L}) \setminus L^\circ$ has horizontal distance $\leq 3$ from $L^\circ$.

As in the proof of Lemma \ref{Lem:interchangeorderinK} define the sets $S \subset L \cup E_1 \cup \ldots \cup E_n$, $S' \subset L$ and $\ov{S} \subset \ov{L} \cup \ov{E}_1 \cup \ldots \cup \ov{E}_{\ov{n}}$, $\ov{S}' \subset \ov{L}$ .
Observe that $S', \ov{S}'$ lie in different sets and might not intersect as before.
However, we can still find cells $Q^\circ \subset S'$, $\ov{Q}^\circ \subset \ov{S}'$ such that 
\[ \dist^H_K (Q^\circ, \ov{Q}^\circ) \leq 3 \qquad \text{and} \qquad \dist^V_K (Q^\circ, \ov{Q}^\circ) = 0. \]
We will work with these cells now instead of $Q^\circ$ alone.
By the definition of $S'$ there is a cell $Q^{\circ\circ} \subset S$ which is either equal to $Q^\circ$ or adjacent to $Q^\circ$ and horizontally aligned with it, i.e. $\dist^H_K (Q^{\circ\circ}, Q^\circ) \leq 1$ and $\dist^V_K (Q^{\circ\circ}, Q^\circ) = 0$.
Again, by the definition of $S$, there is a cell $Q' \in \mathcal{Q}$ on $\gamma$ which is either equal to $Q^{\circ\circ}$ or adjacent to it across a wall, i.e. $\dist(Q', Q^{\circ\circ}) \leq H$.
Altogether this implies that $\dist (Q', Q^\circ) \leq 1 + H$.
By an analogous argument, we can find a cell $\ov{Q}'$ on $\ov\gamma$ with $\dist(\ov{Q}', \ov{Q}^\circ) \leq 1 + H$.
Hence $\dist(Q', \ov{Q'}) \leq 5 + 2H < 3H$ for large enough $H$.
\end{proof}

\begin{Proposition} \label{Prop:combinatorialconvexityCaseB}
Proposition \ref{Prop:combinatorialconvexityseveralSeifert} also holds in the case in which $M$ satisfies condition (B).
\end{Proposition}

\begin{proof}
We follow the proof of Lemma \ref{Lem:combconvexseveralSeifertpreparation}.
Observe that since $M$ satisfies condition (B), the action $\varphi : \IZ \curvearrowright \td{M}$ acts as a vertical shift on each column of $\td{M}$.
So we do not need to restrict to the case in which the cells $Q_0, Q_1, Q_2$ lie in the same chamber.
Instead of applying Lemma \ref{Lem:interchangeorderinK}, we now make use of Lemma \ref{Lem:interchangeorderinKCaseB} to obtain cells $Q'_1, Q'_2 \subset K$ on $\varphi_z \circ \gamma_1$ and $\gamma_2$ with $\dist(Q'_1, Q'_2) < 3H < 4H$.
The rest of the proof is exactly the same as that of Lemma \ref{Lem:combconvexseveralSeifertpreparation}
\end{proof}

\subsection{Proof of Proposition \ref{Prop:easiermaincombinatorialresultCaseb} if $M$ satisfies condition (B) or (C)} \label{subsec:proofofeasiermaincombinatorialresult}
We will now apply the combinatorial convexity estimates from Propositions \ref{Prop:combinatorialconvexityseveralSeifert} and \ref{Prop:characteriztionminimizinggammaCaseB} to construct large polyhedral balls in $\td{M}$ which consist of cells.
In the following we will always assume that $M$ satisfies condition (B) or (C) and that $\eta$ resp. $H$ have been chosen smaller resp. larger than than all constants $\eta^*$ resp. $H^*$ which appeared the Lemmas and Propositions of subsections \ref{subsec:combconvexincaseC} and \ref{subsec:combconvexincaseB}.

\begin{Lemma} \label{Lem:Sinchambersimplyconnected}
Let $K \in \mathcal{K}$ be a chamber of $\td{M}$ and consider a finite union of cells $S \subset K$ whose interior is connected.
Assume that $S$ has the property that for any two cells $Q_1, Q_2 \subset S$ which are vertically aligned, $S$ also contains all cells which are vertically between $Q_1$ and $Q_2$.
Then $S$ is homeomorphic to a closed $3$-disk and the intersection of $S$ with every wall $W \subset \partial K$ has connected interior in $W$.
More precisely, there is a continuous, injective map $b : D^3 \to \td{M}$ with $b(D^3) = S$ which is an embedding on $B^3 \cup (S^2 \setminus b^{-1} (\td{V}^{(1)}))$ and for all walls $W \subset \partial K$ the preimage $b^{-1} (W)$ is either empty or a (connected) topological disk which is the union of rectangles.
\end{Lemma}

\begin{proof}
The Lemma is obviously true if $S$ only consists of cells which are vertically aligned.
Observe that the columns of $K$ are bounded by subsets of $\partial K$ and subsets of components of $\pi^{-1} (C_K)$.
Those components correspond to curves of $\td{C}^*_K \subset \td\Sigma_K$, are diffeomorphic to $I \times \IR$ and every two adjacent columns intersect in exactly one such component.
Moreover, each such component separates $K$ into two components.

Consider now such a component $X \subset \pi^{-1} (C_K)$ with the property that not all cells of $S$ lie on one side of $X$.
This is always possible if not all cells of $S$ are vertically aligned.
Let $S_1, S_2 \subset K$ be the closures of the two components of $S \setminus X$.
Then $S_1 \cap S_2$ is a connected rectangle and so the interiors of $S_1, S_2$ must be connected and hence $S_1, S_2$ are homeomorphic to $3$-disks.
Since the interior of $S_1 \cap S_2$ in $X$ is a (connected) disk, we find that $S = S_1 \cup S_2$ is a topological $3$-disk as well.

Next, let $W \subset \partial K$ be a wall and assume that two cells $Q, Q' \subset S$ are adjacent to $W$.
Let $E, E' \in \mathcal{E}_K$ be the columns that contain $Q, Q'$.
Since $S$ is connected, we can find a chain $(E_0, \ldots, E_n)$ between $E, E'$ such that $E_i$ contains a cell of $S$ for all $i = 0, \ldots, n$.
We may assume that we have picked the chain such that $n$ is minimal.
It is then easy to see that this chain cannot contain any column twice.
Hence it is minimal and so all its columns are adjacent to $W$.
It follows easily that we can connect $Q$ with $Q'$ through cells in $K$ which are adjacent to $W$ and hence $S \cap W$ is connected.
By the property of $S$, this intersection can only be a topological $2$-disk.

The existence of the map $b$ follows along the lines of this proof.
\end{proof}

Let $Q_0 \in \mathcal{Q}$ be an arbitrary cell and $R > 0$ a positive number.
Then we define
\[ B_R (Q_0) = \bigcup \big\{ Q \in \mathcal{Q} \;\; : \;\; \dist (Q, Q_0) < R \big\}. \]
Next, consider the distance function $\dist^{\mathcal{K}} : \mathcal{K} \times \mathcal{K} \to [0, \infty)$ which assigns to every pair of chambers $K_1, K_2$ the length of the minimal chain between $K_1, K_2$.
This length is equal to the minimal number of intersections of a curve between $K_1, K_2$ with the walls of $\td{M}$.
For two cells $Q_1 \subset K_1, Q_2 \subset K_2$ we set $\dist^{\mathcal{K}} (Q_1, Q_2) = \dist^{\mathcal{K}} (K_1, K_2)$.
Observe that
\[ \dist (Q_1, Q_2) \geq H \dist^{\mathcal{K}} (K_1, K_2). \]
Let $J > 0$ be a large constant whose value we will determine later.
We define a new distance function $\dist' (\cdot, \cdot)$ on $\mathcal{Q}$ as follows
\[ \dist' (Q_1, Q_2) := \dist(Q_1, Q_2) + J \dist^{\mathcal{K}} (Q_1, Q_2). \]
Obviously, $(\mathcal{Q}, \dist')$ is a metric space.
Set moreover
\[ B'_R (Q_0) = \bigcup \big\{ Q \in \mathcal{Q} \;\; : \;\; \dist'(Q, Q_0)  < R \big\}. \]
Finally, we define
\[ P_R (Q_0) = \bigcup \Bigg\{ Q \in \mathcal{Q} \;\; : \;\; \begin{array}{l} Q \subset K \in \mathcal{K} \; \text{and there are cells $Q_1, Q_2 \subset B'_R (Q_0)$} \\ \text{in $K$ such that $Q_1, Q, Q_2$ are vertically aligned} \\ \text{and $Q$ lies vertically in between $Q_1, Q_2$} \end{array} \Bigg\}. \]

\begin{Proposition} \label{Prop:distballsinQQareballs}
Assume that $M$ satisfies condition (B) or (C).
Then there are choices for $\eta, H, J$ and a constant $C_2 < \infty$ such that the following holds:

For all $Q_0 \in \mathcal{Q}$ and all $R > 0$ we have
\[ B'_R (Q_0) \subset P_R (Q_0) \subset \Int B'_{R + C_2} (Q_0) \cup \partial \td{M}. \]
Moreover, there is a continuous map $b_{R, Q_0} : D^3 \to \td{M}$ such that $b_{R, Q_0} (D^3) = P_R(Q_0)$ and $b_{R, Q_0} (S^2) = \partial P_R(Q_0)$ and such that $b_{R, Q_0}$ is an injective embedding on $B^3 \cup (S^2 \setminus b_{R, Q_0}^{-1} (\td{V}^{(1)}))$.

Finally, let $K_0 \in \mathcal{K}$ be the chamber that contains $Q_0$.
Then for all cells $Q \subset B'_R(Q_0) \cap K_0$ we have $\dist^H_{K_0} (Q, Q_0), \dist^V_{K_0} (Q, Q_0) < C_2 R$.
\end{Proposition}

\begin{proof}
We will see that the proposition holds for $J = 11H$.

We first show that
\begin{equation} \label{eq:easyinclusionofcellballs}
 B'_R (Q_0) \subset P_R (Q_0) \subset B'_{R + 10H} (Q_0).
\end{equation}
The first inclusion property is trivial.
For the second inclusion property consider a cell $Q \subset P_R(Q_0)$.
Let $K \in \mathcal{K}$ be the chamber that contains $Q$ and choose cells $Q_1, Q_2 \subset B'_R (Q_0) \cap K$ such that $Q_1, Q, Q_2$ are vertically aligned and $Q$ lies vertically in between $Q_1, Q_2$.
Then $\dist (Q_1, Q_0) = \dist' (Q_1, Q_0) - J \dist^{\mathcal{K}} (K, K_0) < R - J \dist^{\mathcal{K}} (K_0, K)$ and similarly $\dist (Q_2, Q_0) < R - J \dist^{\mathcal{K}} (K, K_0)$.
It follows from Proposition \ref{Prop:combinatorialconvexityseveralSeifert} in case (C) and Proposition \ref{Prop:combinatorialconvexityCaseB} in case (B) that $\dist (Q, Q_0) < R + 10H - J \dist^{\mathcal{K}} (K, K_0)$.
So $\dist'(Q, Q_0) < R + 10H$ and (\ref{eq:easyinclusionofcellballs}) follows.
In order to establish the inclusion property of this proposition, it hence suffices to choose $C_2$ larger than $10H + J$ plus the maximal number of cells that can intersect in one point.

Next, choose a sequence $K_1, K_2, \ldots \in \mathcal{K}$ such that $\mathcal{K} = \{ K_0, K_1, K_2, \ldots \}$ and such that $\dist^{\mathcal{K}} (K_n, K_0)$ is non-decreasing in $n$.
We will first show that the interior of $B'_R(Q_0) \cap (K_0 \cup \ldots \cup K_n)$ is connected for each $n \geq 0$:
Fix $n$, choose a cell $Q \subset B'_R(Q_0) \cap (K_0 \cup \ldots \cup K_n)$, $Q \neq Q_0$, let $K_i$ be the chamber that contains $Q$ and consider a combinatorially minimizing curve $\gamma : [0,1] \to \td{M}$ from $Q_0$ to $Q$.
We show by induction on the number of cells that intersect $\gamma$ that $\Int Q$ lies in the same connected component of $\Int (B'_R(Q_0) \cap (K_0 \cup \ldots \cup K_n))$ as $\Int Q_0$.
Let $Q' \in \mathcal{Q}$ be the cell that $\gamma$ intersects prior to $Q$.
If $Q' \subset K_i$, then we are done by the induction hypothesis since then $\dist'(Q', Q_0) < \dist'(Q, Q_0)$ and hence $Q' \subset B'_R(Q_0) \cap (K_0 \cup \ldots \cup K_n)$.
Assume next that $Q' \subset K_j \in \mathcal{K}$ for $j \neq i$ and that $\gamma$ crosses a wall $W = K_i \cap K_j \in \mathcal{W}$ in between $Q'$ and $Q$.
Then $\dist (Q', Q_0) = \dist (Q, Q_0) - H$ and $\dist^{\mathcal{K}} (K_j, K_0) = \dist^{\mathcal{K}} (K_i, K_0) \pm 1$.
It suffices to consider the case in which $\dist^{\mathcal{K}} (K_j, K_0) = \dist^{\mathcal{K}} (K_i, K_0) + 1$ since otherwise we are again done by the induction hypothesis.
In this case $\gamma$ must cross $W$ twice and there is a cell $Q'' \subset K_i$ that $\gamma$ intersects right before intersecting $W$ for the first time.
By Proposition \ref{Prop:characterizationminimizinggamma}(b2) in case (C) or Proposition \ref{Prop:characteriztionminimizinggammaCaseB}(b3) in case (B), the curve $\gamma$ only intersects cells which lie in $K_j$ and which are adjacent to $W$ between $Q''$ and $Q'$.
Consider now all cells $Q^* \subset K_i$ which are adjacent to a cell $Q^{**} \subset K_j$ which intersects $\gamma$.
For each such $Q^*$ we have
\[ \dist(Q^*, Q_0) = H + \dist(Q^{**}, Q_0) \leq H + \dist(Q', Q_0) = \dist(Q, Q_0) \]
and thus $\dist'(Q^*, Q_0) \leq \dist' (Q, Q_0)$ and $Q^* \subset B'_R(Q_0) \cap K_i$.
It is then easy to conclude that $Q'$ and $Q''$ lie in the same connected component of $B'_R (Q_0) \cap K_i$.
This finishes the induction argument.

So also the interior of $P_R (Q_0) \cap (K_0 \cup \ldots \cup K_n)$ is connected for all $n \geq 0$.
We will now show by induction on $n$ that there is a continuous map $b_n : D^3 \to \td{M}$ whose image is equal to the closure of this interior and which is an injective embedding when restricted to $B^3 \cup (S^2 \setminus b^{-1}_n (\td{V}^{(1)}))$.
For $n = 0$ this statement follows immediately from Lemma \ref{Lem:Sinchambersimplyconnected} and the fact that the interior of $P_R(Q_0) \cap K_0$ is connected.
Assume now that $n \geq 1$.
There is a unique $i \in \{ 1, \ldots, n-1 \}$ such that $K_i$ is adjacent to $K_n$.
So $\dist^{\mathcal{K}} (K_i, K_0) = \dist^{\mathcal{K}} (K_n, K_0) - 1$.
Let $W = K_i \cap K_n \in \mathcal{W}$ be the wall between $K_i$ and $K_n$.
Observe that for every cell $Q \subset P_R(Q_0) \cap K_n$ which is adjacent to $W$ and every cell $Q' \subset K_i$ which is adjacent to $Q$ we have $\dist(Q', Q_0) \leq \dist(Q, Q_0) + H$.
So by (\ref{eq:easyinclusionofcellballs})
\[ \dist' (Q', Q_0) \leq \dist' (Q, Q_0) + H - J < R+ 11H - J = R. \]
Hence $Q' \subset B'_R (Q_0) \subset P_R (Q_0)$.
This implies
\[ \ov{P_R (Q_0) \cap \Int K_n} \cap W \subset P_R (Q_0) \cap W = \ov{P_R (Q_0) \cap \Int K_i} \cap W = b_{n-1} (D^3) \cap W . \]
By Lemma \ref{Lem:Sinchambersimplyconnected} the set $\ov{P_R (Q_0) \cap \Int K_n}$ is the union of the images of maps $b' : D^3 \to \td{M}$ with the appropriate regularity properties and any two such images intersect in at most an edge of $\td{V}$.
Moreover, the preimage of $W$ under every such map $b'$ is a (connected) topological disk which is contained in $b_{n-1}(D^3) \cap W$.
It is now easy to see that we can combine $b_{n-1}$ with the maps $b'$ to obtain a map whose image is equal to the closure of the interior of $P_R (Q_0) \cap (K_0 \cup \ldots \cup K_n)$.
Smoothing this map in the interior of $D^3$ yields $b_n$.
This finishes the induction and proves the second assertion of the Proposition for large $n$.

Finally, we show the last statement.
Let $Q \subset B'_R(Q_0) \cap K_0$.
Then $\dist (Q, Q_0) < R$.
Consider a minimizing curve $\gamma : [0,1] \to \td{M}$ between $Q_0, Q$.
By Proposition \ref{Prop:characterizationminimizinggamma} or Proposition \ref{Prop:characteriztionminimizinggammaCaseB} the curve $\gamma$ stays within the union of $K_0$ with the chambers which are adjacent to $K_0$.
Let $Q_0, Q_1, \ldots, Q_n = Q' \subset K_0$ be the cells of $K_0$ that $\gamma$ intersects in that order.
Then for all $i = 0, \ldots, n-1$ either $\dist^H_{K_0} (Q_i, Q_{i+1}) + \dist^V_{K_0} (Q_i, Q_{i+1}) = 1 \leq \eta^{-1} \dist(Q_i, Q_{i+1})$ or $\gamma$ intersects a wall $W \subset \partial K$ right after $Q_i$ and right before $Q_{i+1}$.
In this case let $K' \in \mathcal{K}$ be the chamber on the other side of $W$ and let $Q'_i, Q'_{i+1} \subset K'$ be the cells that $\gamma$ intersects right after $Q_i$ and right before $Q_{i+1}$.
By Lemma \ref{Lem:QQestimatesatW}(f) we have
\[ \dist^H_{K_0} (Q_i, Q_{i+1}),\; \dist^V_{K_0} (Q_i, Q_{i+1}) < C_0 \eta^{-1} \dist(Q'_i, Q'_{i+1}) + C_0 . \]
If $H > 1$, then the right hand side is bounded by $C_0 \eta^{-1} \dist(Q_i, Q_{i+1})$.
The rest follows from the triangle inequality for $\dist^H_{K_0}$ and $\dist^V_{K_0}$ with $C_2 > C_0 \eta^{-1}$.
\end{proof}

We can finally establish Proposition \ref{Prop:easiermaincombinatorialresultCaseb} and hence Proposition \ref{Prop:maincombinatorialresult}(a) (see subsection \ref{subsec:reducifnoT2bundle}).

\begin{proof}[Proof of Proposition \ref{Prop:easiermaincombinatorialresultCaseb}]
By Proposition \ref{Prop:CasAiseasy}, we may assume that $M$ satisfies condition (B) or (C).

Observe first that the universal covering $\pi : \td{M} \to M$ can be seen as the restriction of the universal covering $\pi : \td{M}_0 \to M_0$ to a component of $\pi^{-1} (M)$.
Consider the simplicial complex $V \subset M$ as defined in subsection \ref{subsec:constructionofV} and let $f_0 : V \to M$ be the inclusion map.
Recall that $f_0$ lifts to the inclusion map $\td{f}_0 : \td{V} \to \td{M}$ in the universal covering $\pi : \td{M} \to M$.
Consider the Riemannian metric $g$ on $M_0$ and the map $f : V \to M$ from the assumptions of the Proposition.
Let $H : V \times [0,1] \to M_0$ be the homotopy between $f_0$ and $f$ and let $L$ be a strict upper bound on the length of the curves $t \mapsto H(x, t)$ (note that $V$ is compact).
Since this homotopy leaves $\partial V$ invariant and embedded in $\partial M$, we can extend $H$ to a homotopy $H^* : (V \cup \partial M) \times [0,1] \to M_0$ between the inclusion map $f_0^* : V \cup \partial M \to M$ and the extension $f^* : V \cup \partial M \to M_0$ of $f$ such that $H^* (\cdot , t)$ restricted to $\partial M$ is the identity for all $t \in [0,1]$.
Here we view $V \cup \partial M$ as a connected simplicial complex.
The homotopy $H^*$ can be lifted to a homotopy $\td{H}^* : (\td{V} \cup \partial \td{M}) \times [0,1] \to \td{M}_0$ between the inclusion map $\td{f}_0 : \td{V} \cup \td{M} \to \td{M}$ and a lift $\td{f}^* : \td{V} \cup \partial \td{M} \to \td{M}_0$ of $f^*$, i.e. $f^* \circ \pi |_{\td{V} \cup \partial \td{M}} = \pi \circ \td{f}^*$.
Note that $\pi (\td{H}^* (x,t)) = H^* (\pi(x),t)$ for all $(x,t) \in (\td{V} \cup \partial \td{M}) \times [0,1]$.
Still, the lengths of the curves $t \mapsto \td{H}^*(x, t)$ are bounded by $L$.

Consider the solid torus $S \subset \Int M$ and pick a component $\td{S} \subset \pi^{-1} (S) \subset \Int \td{M}$.
Then by our assumptions $\td{S} \approx D^2 \times \IR$ and $\pi |_{\td{S}} : \td{S} \to S$ is a universal covering map of $S$.
Let $F \subset \td{S}$ be a fundamental domain of the solid torus which arises from cutting $S \approx S^1 \times D^2$ along an embedded disk $\approx \{ \text{pt} \} \times D^2$.
The central loop $\sigma \approx S^1 \times \{ 0 \} \subset S \approx S^1 \times D^2$ induces a deck transformation $\varphi : \td{M}_0 \to \td{M}_0$ which is an isometry and $\td{S}$ is covered by domains of the form $\varphi^{(n)} (F)$ where $n \in \IZ$.
Observe also that $\td\sigma = \pi^{-1} (\sigma) \cap \td{S}$ is a properly embedded, infinite line which is invariant under $\varphi$.

Choose a chamber $K_0 \in \mathcal{K}$ for which the displacement $\dist^{\mathcal{K}} (K_0, \varphi(K_0))$ is minimal.
Next, if $\varphi(K_0) = K_0$ choose a column $E_0 \in \mathcal{E}_{K_0}$ for which the displacement $\dist^H_{K_0} (E_0, \varphi(E_0))$ is minimal.
If $\varphi(K_0) \neq K_0$, the column $E_0 \in \mathcal{E}_{K_0}$ can be chosen arbitrarily.
Finally, choose an arbitrary cell $Q_0 \subset E_0$.
We will now show that there is a universal constant $c > 0$ which only depends on the structure of $V$ (and not on $S$!) such that for all $n \in \IZ$
\begin{equation} \label{eq:distgrowslinearlyinQQ}
 \dist'(Q_0, \varphi^{(n)} (Q_0)) \geq c |n|.
\end{equation}
If $\varphi (K_0) \neq K_0$, then we argue as follows.
Consider the minimal chain between $K_0$ and $\varphi(K_0)$ in the adjacency graph of $\mathcal{K}$.
The images of this minimal chain under the deck transformations $\varphi^{(0)}, \ldots, \varphi^{(n-1)}$ are each minimal and can be concatenated along $\varphi^{(1)} (K_0), \ldots, \varphi^{(n-1)} (K_0)$ to a chain between $K_0$ and $\varphi^{(n)} (K_0)$.
We now claim that this chain is minimal.
Otherwise, there are elements in this chain which occur at least twice.
Since the adjacency graph of $\mathcal{K}$ is a tree (see Lemma \ref{Lem:KKistree}), there must then be even two consecutive elements in this chain which are equal.
These two elements can only come from two distinct images of the minimal chain between $K_0$ and $\varphi(K_0)$.
So if $K'_0, K''_0 \in \mathcal{K}$ are the second and second last elements on this minimal chain then we must have $\varphi^{(i+1)} (K'_0) = \varphi^{(i)}(K''_0)$ for some $i \in \{ 0, \ldots, n-1 \}$.
But this would imply that $\varphi(K'_0) = K''_0$ and hence $\dist^{\mathcal{K}} ( K'_0, \varphi(K'_0)) = \dist^{\mathcal{K}} ( K'_0, K''_0)  = \dist^{\mathcal{K}} (K_0, \varphi(K_0)) - 2$; contradicting the minimal choice of $K_0$.
So we conclude that the chain in question is minimal and hence $\dist^{\mathcal{K}} (K_0, \varphi^{(n)}(K_0) ) \geq |n|$ for all $n \in \IZ$ which establishes (\ref{eq:distgrowslinearlyinQQ}) for $c < H+J$.

If $\varphi(K_0) = K_0$ but $\varphi(E_0) \neq E_0$, then we can draw the same conclusions for $\mathcal{E}_{K_0}$ instead of $\mathcal{K}$ and obtain $\dist^H_{K_0} (Q_0, \varphi^{(n)}(Q_0)) \geq |n|$ for all $n \in \IZ$.
If $\varphi(E_0) = E_0$, then $\dist^V_{K_0} (Q_0, \varphi^{(n)}(Q_0)) \geq |n|$ for all $n \in \IZ$.
So in the latter two cases (\ref{eq:distgrowslinearlyinQQ}) follows by the last assertion of Proposition \ref{Prop:distballsinQQareballs}.

Let $N \geq 1$ be some large natural number whose value we will determine at the end of the proof.
The sets $\td{S}_+ = F \cup \varphi (F) \cup \ldots \varphi^{(N-1)} (F)$ and $\td{S}_- = \varphi^{(-1)} (F) \cup \ldots \varphi^{(-N)} (F)$ are each diffeomorphic to a solid cylinder $\approx I \times D^2$ and are bounded by annuli inside $\partial \td{S}$ as well as disks $D_0, D_+$ and $D_0, D_-$ where $D_+ = \varphi^{(N)} (D_0)$ and $D_- = \varphi^{(-N)} (D_0)$.
Let $\td\sigma_+, \td\sigma_-$ be the subsegments of $\td\sigma$ which connect $D_0$ with $D_+$ and $D_0$ with $D_-$, i.e. $\td\sigma_+ = \td\sigma \cap \td{S}_+$ and $\td\sigma_- = \td\sigma \cap \td{S}_-$.

Choose $R_0 > 0$ large enough such that $B'_{R_0} (Q_0)$ contains all points of $\td{M}$ which have distance at most $L$ from $F$.
Then for all $n \in \IZ$ the set $B'_{R_0} ( \varphi^{(n)}(Q_0))$ contains all points of $\td{M}$ which have distance at most $L$ from $\varphi^{(n)} (F)$.
Consider for the moment some number $R$ such that
\[ R_0 \leq R \leq c N - C_2 - R_0. \]
($C_2$ is the constant from Proposition \ref{Prop:distballsinQQareballs}.)
Then we have $B'_{R_0} (Q_0) \subset B'_R (Q_0) \subset P_R (Q_0)$, so every point of $\td{M}$ which has distance at most $L$ from $D_0 \subset F$ is contained in $P_R(Q_0)$.
We now claim that for all $n \in \IZ$ with $|n| \geq N$ the set $B'_{R_0} (\varphi^{(n)} (Q_0))$ is disjoint from the interior of $B'_{R + C_2} (Q_0)$.
Assume not and let $Q' \subset B'_{R_0} (\varphi^{(n)} (Q_0)) \cap B'_{R + C_2} (Q_0)$ be a cell in the intersection.
Then we obtain the following contraction using (\ref{eq:distgrowslinearlyinQQ}):
\[ c|n| \leq \dist'(Q_0, \varphi^{(n)} (Q_0)) \leq \dist'(Q_0, Q') + \dist' (Q', \varphi^{(n)} (Q_0)) < R_0 + R + C_2 \leq cN. \]
So $B'_{R_0} (\varphi^{(n)} (Q_0))$ is disjoint from $\Int B'_{R + C_2} (Q_0)$ and hence also from $P_R (Q_0) \subset \Int B'_{R + C_2} (Q_0) \cup \partial \td{M}$.
So $P_R(Q_0)$ does not contain any point of $\td{M}$ which has distance at most $L$ from $\td{S} \setminus (\td{S}_+ \cup \td{S}_-)$ and thus also from $D_+, D_-$.
This implies in particular that the curves $\td\sigma_+$ and $\td\sigma_-$ have intersection number $1$ with the restriction $b_{R, Q_0} |_{S^2} : S^2 \to \td{M}$, whose image is $\partial P_R(Q_0)$ (see Proposition \ref{Prop:distballsinQQareballs}).

Our conclusions imply that the homotopy $\td{H}^*$ restricted to $\partial P_R(Q_0)$ does not intersect $D_0 \cup D_+ \cup D_-$ or, more generally, that it stays away from $\td{S} \setminus (\td{S}_- \cup \td{S}_+)$.
If we view $b_{R, Q_0} |_{S^2}$ as a map from $S^2$ to $\td{V} \cup \partial \td{M}$, then $(x,t) \mapsto \td{H}^* (b_{R, Q_0} |_{S^2}(x), t)$ is a homotopy between $b_{R, Q_0} |_{S^2} = \td{f}^*_0 \circ b_{R, Q_0} |_{S^2} : S^2 \to \td{M}$ and $s_{R, Q_0} = \td{f}^* \circ b_{R, Q_0}|_{S^2} : S^2 \to \td{M}_0$ whose image is disjoint from $D_0 \cup D_+ \cup D_-$  and $\td{S} \setminus (\td{S}_- \cup \td{S}_+)$.
So $s_{R, Q_0}$ has intersection number $1$ with $\td\sigma_+$ and $\td\sigma_-$.
Choose a small perturbation $s'_{R, Q_0} : S^2 \to \td{M}_0$ of $s_{R, Q_0}$ which intersects $\partial \td{S}$ transversally, which still stays away from $D_0, D_+, D_-$ and and $\td{S} \setminus (\td{S}_- \cup \td{S}_+)$ and which satisfies
\begin{equation} \label{eq:sRQ0ssRQ0}
 \area s'_{R, Q_0} \big|_{{s'}^{-1}_{R, Q_0} (\td{S}_+ \cup \td{S}_-)} < 2 \area s_{R, Q_0} \big|_{s^{-1}_{R, Q_0}(\td{S}_+ \cup \td{S}_-)}.
\end{equation}
(This can always be achieved by perturbing the composition of $s_{R, Q_0}$ with a diffeomorphism of $\td{M}_0$ which slightly expands $\td{S}$.)
Set
\[ X = {s'}_{R, Q_0}^{-1} (\td{S}), \qquad  X_+ = {s'}_{R, Q_0}^{-1} (\td{S}_+), \qquad X_- = {s'}_{R, Q_0}^{-1} (\td{S}_-). \]
Then $X, X_+, X_-$ are compact smooth domains of $S^2$ and we have $X = X_+ \dotcup X_- $, $s'( \partial X) \subset \partial \td{S}$ and $s'_{R, Q_0}$ restricted to $X_+$ and $X_-$ has non-zero intersection number with $\td{\sigma}$.
Let $X'_+ \subset X_+$ be the union of all components of $X_+$ on which $s'_{R, Q_0}$ has non-zero intersection number with $\td\sigma$, define $X'_- \subset X_-$ analogously and set $X' = X'_+ \cup X'_-$.
Then $X'_+, X'_- \neq \emptyset$ and $X'_+, X'_- \neq S^2$ and every component $Y \subset X'$ is bounded by at least one circle $Z \subset \partial Y$ such that $s'_{R, Q_0} |_Z : Z \to \partial \td{S}$ is non-contractible in $\partial \td{S}$.
Each such circle $Z$ bounds two disks $E_1, E_2 \subset S^2$ on either side (one of these disks contains $Y$ and the other one doesn't).
Consider now the set of all such disks $E \subset S^2$ coming from all components $Y$ of $X'$ and all boundary circles $Z \subset \partial Y$ for which $s'_{R, Q_0} |_Z : Z \to \partial \td{S}$ is non-contractible in $\partial \td{S}$.
Any two such disks are either disjoint or one is contained in the other.
We can hence choose a component $Y \subset X'$, a boundary circle $Z \subset \partial X'$ with the aforementioned property and a disk $E \subset S^2$ bounded by $Z$ such that $E$ is minimal with respect to inclusion.
We argue that $s'_{R, Q_0}$ restricted to every other boundary circle $Z' \subset \partial Y$ is contractible in $\td{S}$:
If this was not the case, then $Y$ must be disjoint from the interior of $E$, since otherwise $Z' \subset Y \subset E$ bounds a disk $E' \subsetneq E$.
By the same argument, $E$ cannot contain any other component $Y'$ of $X'$, because otherwise we would find a boundary circle $Z'' \subset \partial Y' \subset \Int E$ such that $s'_{R, Q_0} |_{Z''}$ is non-contractible in $\partial \td{S}$.
So $\Int E$ must be be disjoint from $X'$ and hence $s'_{R, Q_0} |_{E}$ describes a nullhomotopy of a non-contractible curve in $\partial \td{S}$ which does not intersect $\td\sigma$.
Since $\pi_2 (\td{M}_0) = \pi_2 (M_0) = 0$, this nullhomotopy can be homotoped relative boundary to a nullhomotopy which has non-zero intersection number with $\td\sigma$.
This is however impossible and we obtain a contradiction.
So $s'_{R, Q_0}$ restricted to all other boundary components of $Y$ is non-contractible in $\td{S}$ and hence we have shown that $\Sigma = Y$ and $h = \pi \circ s'_{R, Q_0} |_Y$ satisfy all the claims of the Proposition except for the area bound.

In view of (\ref{eq:sRQ0ssRQ0}) it remains to choose $R$ and $N$ such that $\area s_{R, Q_0} |_{s^{-1}_{R, Q_0}(\td{S}_+ \cup \td{S}_-)}$ can be bounded in terms of $\area f$.
To do this choose radii $R_i = R_0 + C_2 i$ where $i = 0, \ldots, e$ with $e = \lfloor C_2^{-1} (cN - C_2 - 2 R_0) \rfloor$.
Then
\[ R_0 < R_1 < \ldots < R_e \leq c N - C_2 - R_0. \]
By Proposition \ref{Prop:distballsinQQareballs} we know that $\partial P_{R_0} (Q_0) \setminus \partial \td{M}, \ldots, \partial P_{R_e} (Q_0) \setminus \partial \td{M} \subset \td{V} \subset \td{M}$ are pairwise disjoint.
So since $b_{R_i, Q_0} (S^2) = \partial P_{R_i}(Q_0)$ and $s_{R_i, Q_0} = \td{f}^* \circ b_{R_i, Q_0} |_{S^2}$ and $\td{f}^* (\partial \td{M}) = \partial \td{M}$ we have
\[ \area s_{R_0, Q_0} \big|_{s^{-1}_{R_0, Q_0} (\td{S}_+ \cup \td{S}_-)} + \ldots + \area s_{R_e, Q_0} \big|_{s^{-1}_{R_e, Q_0} (\td{S}_+ \cup \td{S}_-)} \leq \area \td{f}^* \big|_{\td{f}_0^{* -1} (\td{S}_+ \cup \td{S}_-)}. \]
Since $\td{S}_+ \cup \td{S}_- = \varphi^{(-N)} (F') \dotcup \ldots \dotcup \varphi^{(N-1)} (F') \cup D_-$ for the half-open set $F' = F \setminus D_0$, we further have
\[ \area \td{f}^* \big|_{\td{f}^{* -1} (\td{S}_+ \cup \td{S}_-)} = \area \td{f}^* \big|_{\td{f}^{* -1} (\varphi^{(-N)} (F'))} + \ldots + \area \td{f}^* \big|_{\td{f}^{* -1} (\varphi^{(N-1)} (F'))}. \]
Observe now that whenever $\td{f}^* (x) = \td{f}^* (y)$ and $\pi(x) = \pi(y)$ for $x, y \in \td{V}$, then $x = y$, since the curves $t \mapsto \td{H}^* (x, t), \td{H}^*(y,t)$ have the same endpoint and project to the same curve $t \mapsto H^* (\pi(x), t)$ under $\pi$.
So for all $n \in \IZ$ the projection $\pi$ restricted to $\td{f}^{* -1}(\varphi^{(n)}(F'))$ is injective.
Since $\pi ( \td{f}^{* -1} (\varphi^{(n)}(F'))) \subset f^{*-1} (S)$, we conclude $\area \td{f}^* \big|_{\td{f}^{* -1} (\varphi^{(n)} (F'))} \leq \area f^* |_{f^{* -1} (S)} < \area f$.
Putting all this together yields
\[ \area s_{R_0, Q_0} \big|_{s^{-1}_{R_0, Q_0} (\td{S}_+ \cup \td{S}_-)} + \ldots + \area s_{R_e, Q_0} \big|_{s^{-1}_{R_e, Q_0} (\td{S}_+ \cup \td{S}_-)} < 2N \area f. \]
So we can find an index $i \in \{ 0, \ldots, e \}$ such that
\[ \area s_{R_i, Q_0} \big|_{s^{-1}_{R_i, Q_0} (\td{S}_+ \cup \td{S}_-)} < \frac{2N}{e+1} \area f \leq \frac{2 C_2 N}{cN - C_2 - 2 R_0} \area f. \]
Choosing $N > 2 c^{-1} (C_2 + 2R_0)$ yields
\[ \area s_{R_i, Q_0} \big|_{s^{-1}_{R_i, Q_0} (\td{S}_+ \cup \td{S}_-)} < 4 c^{-1} C_2 \area f. \]
This finishes the proof of Proposition \ref{Prop:easiermaincombinatorialresultCaseb}
\end{proof}

\subsection{The case in which $M$ is covered by a $T^2$-bundle over a circle} \label{subsec:CaseT2bundleoverS1}
We finally present the proof of Proposition \ref{Prop:maincombinatorialresult}(b).

\begin{Lemma} \label{Lem:geomseriesinSL2Z}
Let $A \in SL(2, \IZ)$ be an invertible $2 \times 2$-matrix with integral entries.
Then for every $k \geq 1$ there is a number $1 \leq d \leq 6^k$ such that
\[ I + A + A^2 + \ldots + A^{d-1} \equiv 0 \quad \mod 3^k. \]
Here $I$ is the identity matrix.
\end{Lemma}

\begin{proof}
We first show the claim for $k = 1$.
Since $\det A = 1$, the Theorem of Cayley-Hamilton yields that $I - (\tr A) A + A^2 = 0$.
So we are done for $\tr A \equiv 2 \mod 3$.
If $\tr A \equiv 0 \mod 3$, then $I + A + A^2 + A^3 = (I + A^2)(I + A) \equiv 0 \mod 3$ and if $\tr A \equiv 1 \mod 3$, then $I + \ldots + A^5 = (I - A + A^2) (I + 2 A + 2 A^2 + A^3) \equiv 0 \mod 3$.

We now apply induction.
Assume that the statement is true for all numbers up to $k \geq 1$.
We will show that it then also holds for $k+1$.
Choose $1\leq d_1 \leq 6$ such that $I + A + \ldots + A^{d_1 - 1}  \equiv 0 \mod 3$.
By the induction hypothesis applied to $A^{d_1} \in SL(2, \IZ)$, there is a number $1 \leq d_2 \leq 6^k$ such that $I + A^{d_1} + A^{2d_1} + \ldots + A^{(d_2-1)d_1} \equiv 0 \mod 3^k$.
So
\begin{multline*}
 I + A  + \ldots + A^{d_1 d_2 - 1} \\
 = (I + A + \ldots + A^{d_1 - 1}) (I + A^{d_1} + \ldots + A^{(d_2-1) d_1}) \equiv 0 \quad \mod 3^{k+1},
\end{multline*}
and $1 \leq d_1 d_2 \leq 6^{k+1}$.
\end{proof}

\begin{Lemma} \label{Lem:Misexpandable}
Assume that $M$ is the total space of a $T^2$-bundle over a circle.
Then for every $n \geq 1$ there is a finite covering map $\pi_n : M \to M$ with the same domain and range such that for every embedded loop $\sigma \subset M$ the preimage $\pi_n^{-1} (\sigma)$ consists of at least $n$ loops.
\end{Lemma}

\begin{proof}
The manifold $M$ is diffeomorphic to a mapping torus of an orientation-preserving diffeomorphism $\phi : T^2 \to T^2$.
The diffeomorphism $\phi$ acts on $\pi_1 (T^2) \cong \IZ^2$ by an element in $A \in SL(2, \IZ)$.
The fundamental group $\pi_1 (M)$ is isomorphic to a semidirect product of $\IZ$ with $\IZ^2$ coming from the action of $A$ on $\IZ^2$.
So $\pi_1(M)$ can be identified with $\IZ^2 \times \IZ$ with the following multiplication
\[ \left( \begin{pmatrix} x_1\\ y_1 \end{pmatrix}, z_1 \right) \cdot \left( \begin{pmatrix} x_2 \\ y_2 \end{pmatrix}, z_2 \right) = \left( \begin{pmatrix} x_1\\ y_1 \end{pmatrix} + A^{z_1} \begin{pmatrix} x_2 \\ y_2 \end{pmatrix}, z_1 + z_2 \right). \]
Since for any $m \geq 1$, the lattice $m \IZ^2 \subset \IZ^2$ is preserved by the action of $A$, the subset
\[ U_m = \left\{ \left( \begin{pmatrix}  m x \\  m y \end{pmatrix}, z \right) \;\; : \;\; x, y, z \in \IZ \right\} \subset \pi_1(M) \]
is a subgroup of $\pi_1(M)$ of index $m^2$.
It is not difficult to see that $U_m$ is isomorphic to $\pi_1(M)$ and that if we consider the corresponding $m^2$-fold covering $\pi'_m : \widehat{M}_m \to M$, then $\widehat{M}_m$ is diffeomorphic to $M$.

It remains to compute the number of components of ${\pi'}_m^{-1} (\sigma)$ and to show that this number can be made arbitrarily large for the right choice of $m$.
Set $m = 3^k$ for a number $k \geq 1$ which we will determine later.
Let $\widehat{\sigma} \subset  \pi^{' -1}_m (\sigma)$ be an arbitrary loop in the preimage of $\sigma$.
Then we can find an element $g = ((x, y), z) \in \pi_1(M)$ in the conjugacy class of $[\sigma]$ such that $\sigma$ represents $g$ in $M$ and such that $\widehat{\sigma}$ represents a multiple of $g$ which is contained in $U_m = \pi_1 (\widehat{M}_m)$ in $\widehat{M}_m$.
Then the restriction $\pi'_m |_{\widehat{\sigma}} : \widehat{\sigma} \to \sigma$ is a covering of a circle and its degree equal to the first exponent $d \geq 1$ for which $g^d \in U_m$.
We will show that $d \leq 6^k$.
To do this observe that for all $i \geq 1$
\[ g_0^i =  \left( \begin{pmatrix} x\\ y \end{pmatrix} + A^{z} \begin{pmatrix} x \\ y \end{pmatrix} + \ldots + A^{(i-1) z} \begin{pmatrix} x \\ y \end{pmatrix}, i z \right). \]
By Lemma \ref{Lem:geomseriesinSL2Z}, there is a number $1 \leq d \leq 6^k$ such that the first two entries of $g^d$ are divisible by $m = 3^k$ and hence $g^d \in U_m$.
This implies the desired bound.

Since the choice of $\widehat{\sigma}$ was arbitrary, we conclude that every loop in ${\pi'}_m^{-1} (\sigma)$ covers $\sigma$ at most $6^k$ times and hence the number of such loops is at least
\[ \frac{m^2}{6^k} = \frac{3^{2k}}{6^k} = \Big( \frac32 \Big)^k. \]
So choosing $k$ such that $( \frac32 )^k > n$ yields the desired result.
\end{proof}

\begin{proof}[Proof of Proposition \ref{Prop:maincombinatorialresult}(b)]
It is easy to see that we only have to consider the case in which $M$ is a $T^2$-bundle over a circle, since for any finite cover $\widehat\pi : \widehat{M} \to M$ we can compose the maps $\widehat{f}_1, \widehat{f}_2, \ldots : V \to \widehat{M}$ obtained for $\widehat{M}$ with $\widehat\pi$ to obtain the maps $f_1 = \widehat\pi \circ \widehat{f}_1, f_2 = \widehat\pi \circ \widehat{f}_2, \ldots$.

We first establish the assertion for the case $n=1$.
Consider the universal covering $\pi : \td{M} \to M$ and choose a fundamental domain $F \subset \td{M}$.
The projection of its boundary, $V = \pi (\partial F) \subset M$, can be seen as an embedded simplicial complex in $M$.
Let $f_1 : V \to M$ be the inclusion map of $V$.
Denote furthermore by $\td{V} = \td\pi^{-1} (V) \subset \td{M}$ the preimage of $V$ and by $\td{f}_1 : \td{V} \to \td{M}$ its inclusion map.
The set $\td{V}$ can be interpreted as an infinite simplicial complex whose $1$-skeleton $\td{V}^{(1)}$ is the preimage of $V^{(1)}$ under $\pi$.
The complement of $\td{V}$ in $\td{M}$ consists of open sets whose closures $Q \subset \td{M}$ are finite polyhedra and which we call \emph{cells}.
The set of cells is again denoted by $\mathcal{Q}$.
Observe that every cell is the image of $F$ under a deck transformation of $\pi : \td{M} \to M$.
We say that two cells $Q_1, Q_2 \in \mathcal{Q}$ are adjacent if they meet in a point of $\td{V} \setminus \td{V}^{(1)}$.
Choose a cell $Q_0 \in \mathcal{Q}$ and consider for each $k \geq 0$ the union $B_k (Q_0)$ of all cells which have distance $\leq k$ in the adjacency graph of $\mathcal{Q}$.
It is not difficult to see that $S_k = \partial B_k(Q_0) \subset \td{V}$ is the image of a continuous map $s_k : \Sigma_k \to \td{V}$ where $\Sigma_k$ is an orientable surface such that $s_k$ is an injective embedding on $\Sigma_k \setminus s_k^{-1} (\td{V}^{(1)})$.

Choose a component $\td\sigma \subset \pi^{-1} (\sigma)$ which intersects $Q_0$.
Then $\td\sigma \subset \td{M}$ is a non-compact, properly embedded line and there is a non-compact ray $\td\sigma^+ \subset \td\sigma$ which starts in $Q_0$.
This implies that $\td\sigma^+$ has non-zero intersection number with the map $\td{f}_1 \circ s_k : \Sigma_k \to \td{M}$ for each $k \geq 1$.

Consider now the continuous map $f'_1 : V \to M$ which is homotopic to $f_1 : V \to M$ via a homotopy $H : V \times [0,1] \to M$ and use this homotopy to construct a lift $\td{f}'_1 : \td{V} \to \td{M}$.
Let $N$ be a bound on the number of cells that each curve of the form $t \mapsto H (\cdot, t)$ intersects.
Then $H$ induces a homotopy from $\td{f}_1 \circ s_N : \Sigma_N \to \td{M}$ to $\td{f}'_1 \circ s_N : \Sigma_N \to \td{M}$ which is disjoint from $Q_0$.
So both maps have the same, non-zero, intersection number with $\td\sigma^+$.
We conclude that $\td{f}'_1( s_N ( \Sigma_N )) \subset \td{f}'_1 ( \td{V})$ intersects $\td\sigma$.
Hence, $f'_1(V)$ intersects $\sigma$.

We finally show the assertion for all remaining $n \geq 2$.
Fix $n$, consider the covering map $\pi_n : M \to M$ from Lemma \ref{Lem:Misexpandable} and set
\[ f_n = \pi_n \circ f_1 : V \to M. \]
Moreover, the preimage $\sigma_n= \pi_n^{-1} (\sigma)$ is the union of at least $n$ loops which all have the property that all its non-trivial multiples are non-contractible in $M$.

Consider a map $f'_n : V \to M$ and a homotopy between $f_n$ and $f'_n$.
This homotopy can be lifted via $\pi_n : M \to M$ to a homotopy between $f_1$ and a map $f'_1 : V \to M$ such that $f'_n = \pi_n \circ f'_1$.
We now have
\[ {f'}_n^{-1} (\sigma) ={f'}^{-1}_1 ( \pi_n^{-1} (\sigma)) =  {f'}^{-1}_1 ( \sigma_n ) = \bigcup_{\sigma' \subset \sigma_n} {f'}^{-1}_n (\sigma'), \]
where the last union is to be understood as the union over all loops $\sigma'$ of $\sigma_n$.
By our previous conclusion, ${f'}^{-1}_1 (\sigma') \neq \emptyset$ for all such $\sigma'$ and all such sets are pairwise disjoint.
This proves the desired result.
\end{proof}

\section{Proof of the main Theorem} \label{sec:mainproof}
\subsection{Existence of short loops and compressing disks of bounded area}
We first establish an analogue of Lemma 7.12 in \cite{Bamler-longtime-II}.
To to do this, we need the following Lemma which is similar to Lemma 7.11 in \cite{Bamler-longtime-II}.

\begin{Lemma} \label{Lem:separtingloopsinmultannulus}
Let $\Sigma \subset \IR^2$ be a compact smooth domain whose boundary circles are denoted by $C_1, \ldots, C_{m}, C'_1, \ldots, C'_{m'}$ with $m, m' \geq 1$.
Moreover, let $g$ be a symmetric non-negative $2$-form on $\Sigma$ (i.e. a degenerate Riemannian metric).

Choose constants $a, b > 0$ and assume we have $\area \Sigma \leq ab$ and $\dist (C_i, C'_{i'}) > a$ for any $i = 1, \ldots, m$ and $i' = 1, \ldots, m'$ (both times with respect to $g$).
Then we can find a collection of pairwise disjoint smoothly embedded loops $\gamma_1, \ldots, \gamma_n \subset \Int \Sigma$ with the property that $\gamma_1 \cup \ldots \cup \gamma_n$ separates $C_1 \cup \ldots \cup C_m$ from $C'_1 \cup \ldots \cup C'_{m'}$ and
\[ \ell(\gamma_1) + \ldots + \ell (\gamma_n) < b. \]
\end{Lemma}

\begin{proof}
We will proceed by induction on $m + m'$.
For $m + m' = 2$, we are done by Lemma 7.10 in \cite{Bamler-longtime-II}.
So assume without loss of generality that $m' \geq 2$.

Let $a_1$ be the infimum of all $a' \geq 0$ such that there is an embedded curve $\sigma \subset \Sigma$ of length $2a'$ which either connects two distinct circles $C_{i_1}, C_{i_2}$ or whose endpoints both lie in the same boundary circle $C_i$ and not all boundary circles of $\Sigma$ lie in the same component of $\Sigma \setminus (\sigma \cup C_i)$.
Pick $\varepsilon > 0$ small such that still $\dist(C_i, C'_{i'}) > a + 2 \varepsilon$ for all $i = 1, \ldots, m$ and $i' = 1, \ldots, m'$ and $\varepsilon < a_1$ if $a_1 > 0$ and choose such a curve $\sigma \subset \Sigma$ of length $2a'$ with $a' < a_1 + \varepsilon$.

Next consider the subsets
\begin{alignat*}{1}
\Sigma_1 &= \{ x \in \Sigma \;\; : \;\; \dist (x, C_1 \cup \ldots \cup C_m) < a_1 - \varepsilon \}, \\
\Sigma_2 &= \{ x \in \Sigma \;\; : \;\; \dist (x, C'_1 \cup \ldots \cup C'_{m'}) < a - a_1 + \varepsilon \}.
\end{alignat*}
Then $\Sigma_1 \cap \Sigma_2 = \emptyset$ and $\sigma \cap \Sigma_2 = \emptyset$.
By definition of $a_1$, the set $\Sigma_1$ is either empty (for $a_1 = 0$) or a disjoint union of half-open domains $\Sigma_{1,1}, \ldots, \Sigma_{1,m}$ such that $\Sigma_{1,i}$ is bounded by $C_i$ for all $i = 1, \ldots, m$.
Moreover, each component $\Sigma_{1,i}$ separates $\Sigma$ into components of which exactly one  contains all boundary circles of $\Sigma$ except $C_i$.
Hence each $\Sigma_{1,i}$ can be completed to a half-open annulus $\Sigma'_{1,i} \subset \Sigma$ and the subsets $\Sigma'_{1,1}, \ldots, \Sigma'_{1,m}, \Sigma_2$ are pairwise disjoint.
Next we let $\Sigma_{2,1}, \ldots, \Sigma_{2,m''}$ be the components of $\Sigma_2$.
We can again extend each $\Sigma_{2,i''}$ to a half-open domain $\Sigma'_{2,i''}$ by filling in all its ``holes'' which don't contain any boundary components of $\partial \Sigma$.
It is then easy to see that the subsets $\Sigma'_{1,1}, \ldots, \Sigma'_{1,m}, \Sigma'_{2,1}, \ldots, \Sigma'_{2,m''}$ are pairwise disjoint and hence
\begin{equation} \label{eq:areasSigma1plusareasSigma2}
 \area \Sigma'_{1,1} + \ldots + \area \Sigma'_{1,m} + \area \Sigma'_{2,1} + \ldots + \area \Sigma'_{2,m''} < \area \Sigma \leq ab.
\end{equation}

Consider first the case in which $a_1 > 0$ and $\area \Sigma'_{1,1} + \ldots + \area \Sigma'_{1,m} < (a_1 - \varepsilon) b$.
Since the two boundary circles of each annulus $\Sigma'_{1,i}$ are $a_1 - \varepsilon$ apart from one another, we can use Lemma 7.10 in \cite{Bamler-longtime-II} to find a geodesic loop $\gamma_i \subset \Sigma'_{1,i}$ which separates both of its boundary circles such that
\[ \ell (\gamma_i) \leq \frac{\area \Sigma'_{1,i}}{a_1 - \varepsilon} \qquad \text{for all} \qquad i = 1, \ldots, m. \]
Then $\ell(\gamma_1) + \ldots + \ell(\gamma_m) < b$ and $\gamma_1 \cup \ldots \cup \gamma_m$ separates $C_1 \cup \ldots \cup C_m$ from $C'_1 \cup \ldots \cup C'_{m'}$.
So in this case we are done.

Now consider the opposite case.
Then by (\ref{eq:areasSigma1plusareasSigma2}) we must have
\[ \area \Sigma'_{2,1} + \ldots + \area \Sigma'_{2,m''} < (a - a_1 + \varepsilon) b. \]
We claim that each $\Sigma'_{2,i''}$ has at most $m + m' -1$ boundary circles (and open ends):
Note that every such open end encloses a boundary circle of $\Sigma$ on the other side.
So we just need to show that there is an open end of $\Sigma_{2, i''}$ and hence of $\Sigma'_{2,i''}$ which encloses at least two boundary circles of $\Sigma$ on the other side.
Recall that $\sigma$ is disjoint from $\Sigma_2$.
Consider the open end of $\Sigma_{2, i''}$ which encloses $\sigma$.
By the choice of $\sigma$ this open end either encloses the two boundary circles of $C_{i_1}, C_{i_2}$ which are connected by $\sigma$ or it encloses the boundary circle $C_i$ and at least one boundary circle which is enclosed by $C_i \cup \sigma$ on the side opposite of $\Sigma_{2, i''}$.
This shows the claim and enables us to use the induction hypothesis on each $\Sigma'_{2,i''}$ to construct smoothly embedded, pairwise disjoint loops $\gamma_1, \ldots, \gamma_n \subset \Sigma'_{2,1} \cup \ldots \cup \Sigma'_{2,m''}$ such that $(\gamma_1 \cup \ldots \cup \gamma_n) \cap \Sigma'_{2,i''}$ separates $(C'_1 \cup \ldots \cup C'_{m'}) \cap \Sigma'_{2, i''}$ from the open ends of $\Sigma'_{2,i''}$ for each $i'' = 1, \ldots, m''$ and such that
\[ \ell (\gamma_1) + \ldots + \ell(\gamma_n) < \frac{\area \Sigma'_{2,1}}{a - a_1 + \varepsilon} + \ldots +  \frac{\area \Sigma'_{2,m''}}{a - a_1 + \varepsilon} < b. \]
It is clear that $\gamma_1, \ldots, \gamma_n$ have the desired topological properties.
\end{proof}

The next Lemma provides a compressing disk of bounded area in a solid torus given a ``compressing multi-annulus'' of bounded area in a larger solid torus.

\begin{Lemma} \label{Lem:makediskfrommultiannulus}
For every $A, K < \infty$ there is an $\td{h}_0 = \td{h}_0 (A, K) < \infty$ such that the following holds:

Consider a Riemannian manifold $(M, g)$, a smoothly embedded solid torus $S \subset M$, $S \approx S^1 \times D^2$ and a collar neighborhood $P \subset S$, $\partial S \subset \partial P$, $P \approx T^2 \times I$ of $\partial S$ which is a $\td{h}_0$-precise torus structure at scale $1$ (cf \cite[Definition 7.3]{Bamler-longtime-II}).
Note that $S' = \Int S \setminus \Int P$ is a solid torus.
Assume also that $|{\Rm}|, |{\nabla\Rm}| < K$ on the $1$-neighborhood around $P$.

Let now $\Sigma \subset \IR^2$ be a compact smooth domain and $f : \Sigma \to S$ a smooth map of $\area f < A$ such that $f (\partial \Sigma) \subset \partial S$ and such that $f$ restricted to only the exterior circle of $\Sigma$ is non-contractible in $\partial S$.

Then there is a smooth map $f' : D^2 \to M$ such that $f' (\partial D^2) \subset \partial S'$, such that $f' |_{\partial D^2}$ is non-contractible in $\partial S'$ and such that $\area f' < \area f + 1$.
\end{Lemma}

\begin{proof}
We first show that there are constants $\varepsilon = \varepsilon(K) > 0$ and $C < \infty$ such that the following isoperimetric inequality holds:
Assume that $\td{h}_0 \leq \varepsilon$.
Then for any smooth loop $\gamma : S^1 \to P$ of length $\ell(\gamma) < \varepsilon$ which is contractible in $P$ there is map $h : D^2 \to M$ with $h |_{S^1} = \gamma$
\begin{equation} \label{eq:isoperimetricintorusstruc}
 \area h \leq C \ell(\gamma)^2.
\end{equation}
By a local version of the results of Cheeger, Fukaya and Gromov \cite{CFG} there are universal constants $\rho = \rho(K) > 0$, $K' = K'(K) < \infty$ such that for every $p \in P$ we can find an open neighborhood $B(p, \rho) \subset V \subset M$ and a metric $g'$ on $V$ with $0.9 g < g' < 1.1 g$ whose curvature is bounded by $K'$ such that the injectivity radius in the universal cover $(\td{V}, \td{g}')$ of $(V, g')$ is larger than $\rho$ at every lift $\td{p} \in \td{V}$ of $p$.
Let $\pi : \td{V} \to V$ be the covering projection.

Now assume that $\varepsilon (K) \leq \min \{ \frac1{10} \rho(K), \frac1{10} K^{\prime -1/2} (K) \}$ and pick $p$ on the image of $\gamma$.
Using the fact that $\varepsilon \leq \frac1{10}\rho$, it is not difficult to see that we can find a chunk $P'$ of $P \approx T^2 \times I$ (i.e. $P' \subset P$ corresponds to a subset of the form $I' \times T^2$ for some subinterval $I' \subset I$), such that $P'$ contains the image of $\gamma$ and such that $P' \subset B(p, \rho) \subset V$.
Then $\gamma$ is also contractible in $P'$ and hence we can lift it to a loop $\td\gamma : S^1 \to \td{V}$ based at a lift $\td{p} \in \td{V}$ of $p$.
Using the exponential map at $\td{p}$ with respect to the metric $g'$, we can then construct a map $\td{h} : D^2 \to \td{V}$ with $\td{h} |_{\partial D^2} = \td\gamma$ and $\area_{g'} \td{h} \leq \frac12 C \ell_g' (\td\gamma)^2$, where $C < \infty$ is a universal constant (note that since $\varepsilon \leq \frac1{10} K^{\prime -1/2}$, we have upper and lower bounds on the Jacobian of this exponential map).
Hence $\area_g \td{h} \leq C \ell_g (\td\gamma)^2$.
Now $h = \pi \circ \td{h}$ satisfies the desired properties.

We now present the proof of the Lemma.
To do this, we choose
\[ \td{h}_0 (A, K) = \min \Big\{ \frac{\varepsilon(K)}{A}, \; \varepsilon (K), \; \frac1 {2\sqrt{C} A} \Big\}. \]
Let $\sigma \subset S \setminus P$ be a loop which generates the fundamental group $\pi_1(S) \cong \IZ$.
Note that by our assumptions, $f$ has non-zero intersection number with $\sigma$.

We first perturb $f$ slightly to make it transversal to $\partial S'$.
This can be done such that we still have $\area f < A$ and that $\area f$ increases by less than $\frac12$.
So it suffices to assume that $f$ is transversal to $\partial S'$ if we can construct $f'$ such that $\area f' < \area f + \frac12$.

So $\Sigma^* = f^{-1} (P) \subset \Sigma$ is a (possibly disconnected) compact smooth subdomain of $\Sigma$ which contains $\partial \Sigma$.
Note that $f(\partial \Sigma^* \setminus \partial \Sigma) \subset \partial S'$.
Denote the components of $\Sigma^*$ by $Q_1, \ldots, Q_p$.
Let $C_0, \ldots, C_q \subset \Sigma^* \subset \Sigma$ be the boundary circles of $\Sigma^*$ such that $C_0$ is the outer boundary circle of $\Sigma$.
Each such circle $C_l \subset \Sigma \subset \IR^2$ bounds a disk $D_l \subset \IR^2$.
Set $\Sigma'_l = D_l \cap \Sigma$.
Any two disks $D_{l_1}, D_{l_2}$ are either disjoint or one disk is contained in the other.
The same is true for the domains $\Sigma'_l$.
We can hence pick $\Sigma'_l$ minimal with respect to inclusion such that $f |_{\Sigma'_l}$ has non-zero intersection number with $\sigma$.
Such a $\Sigma'_l$ always exists, since $\Sigma'_0 = \Sigma$.
Let $j \in \{1, \ldots, p \}$ be the index for which $C_l \subset \partial Q_j$.
We claim that $C_l$ is an interior boundary circle of $Q_j$.
Assume the converse.
Observe that the intersection number of $f |_{Q_j}$ with $\sigma$ is zero, so if $C_l$ is an exterior boundary circle of $Q_j$, then the intersection number of $f$ restricted to the closure of $\Sigma'_l \setminus Q_j$ is non-zero.
This closure however is the disjoint union of other sets $\Sigma'_{l'}$ and hence we can pick a $\Sigma'_{l'} \subsetneq \Sigma'_l$ on which $f$ has non-zero intersection number with $\sigma$.
This contradicts the minimal choice of $\Sigma'_l$.
A direct consequence of the fact that $C_l$ is an interior boundary circle of $Q_j$ is that $C_l \not\subset \partial \Sigma$ and hence $f(C_l) \subset \partial S'$.
We fix $l$ for the rest of the proof.

Next we show that for any circle $C_{l'} \subset \Sigma'_l \setminus C_l$ the restriction $f|_{C_{l'}}$ is contractible in $P$.
Note that for any such index $l'$ we have $\Sigma'_{l'} \subsetneq \Sigma'_l$ and hence $f|_{\Sigma'_{l'}}$ has intersection number zero with $\sigma$.
Moreover, every interior boundary circle of $\Sigma'_{l'}$ is also an interior boundary circle of $\Sigma$.
So $f$ restricted to any interior boundary circle of $\Sigma'_{l'}$ is contractible in $P$.
This shows the desired fact.

Recall that every interior boundary circle of $\Sigma'_l$ is also an interior boundary circle of $\Sigma$ and that $f( C_l ) \subset \partial S'$ and $f( \partial\Sigma'_l \setminus C_l ) \subset \partial S$.
So $f^{-1}(P) = Q_1 \cup \ldots \cup Q_p$ separates $C_l$ from $\partial\Sigma'_l \setminus C_l$.
Hence, after possibly rearranging the $Q_j$ we can find a $p' \in \{ 0, \ldots, p  \}$ such that $Q_1, \ldots, Q_{p'} \subset \Sigma'_l$, such that $Q_1 \cup \ldots \cup Q_{p'}$ is a neighborhood of $\partial \Sigma'_l \setminus C_l$ and such that each $Q_j$ contains at least one boundary circle of $\Sigma'_l$.
We now apply Lemma \ref{Lem:separtingloopsinmultannulus} to each $Q_j$ equipped with the pull-back metric $f^* (g)$ where we group the boundary circles of $Q_j$ into those which are contained in $\partial \Sigma'_l \setminus C_l$ and those which are not.
Doing this, we obtain pairwise disjoint, embedded loops $\gamma_1, \ldots, \gamma_n \subset Q_1 \cup \ldots \cup Q_{p'} \subset \Sigma'_l$ whose union separates $\partial \Sigma'_l \setminus C_l$ from $C_l$ and for which
\[ \ell(f|_{\gamma_1}) + \ldots + \ell(f|_{\gamma_n}) < \frac{\area f|_{Q_1}}{\td{h}_0^{-1}} + \ldots + \frac{\area f|_{Q_1}}{\td{h}_0^{-1}} \leq A \td{h}_0 . \]

For each $k = 1, \ldots, n$ let $D'_k \subset \IR^2$ be the disk which is bounded by $\gamma_k$.
It follows that $\Sigma'_l \cup D'_1 \cup \ldots \cup D'_n$ is equal to the disk $D_l$, i.e. the disks $D'_k$ cover the ``holes'' of $\Sigma'_l$.
Any two disks $D'_{k_1}, D'_{k_2}$ are either disjoint or one is contained in the other.
So after possibly rearranging these disks, we can find an $n' \in \{ 0, \ldots, n \}$ such that the disks $D'_1, \ldots, D'_{n'}$ are pairwise disjoint and such that still $\Sigma'_l \cup D'_1 \cup \ldots \cup D'_{n'} = D_l$.
For each $k = 1, \ldots, n'$, the consider the intersection of disk $D'_k$ with the domain $Q_{j'}$ which contains $\gamma_k$ ($j' \in \{ 1, \ldots, n' \}$).
This intersection $D'_k \cap Q_{j'}$ is a domain whose interior boundary circles are boundary circles of $Q_{j'}$ and hence they are contained in $\Sigma'_l \setminus C_l$.
So $f$ restricted to the interior boundary circles of $D'_k \cap Q_{j'}$ is contractible in $P$.
Since $f(D'_k \cap Q_{j'}) \subset P$, it follows that $f |_{\gamma_k}$ is contractible.

Observe that $\ell (f|_{\gamma_k}) < A \td{h}_0 \leq \varepsilon$.
So we can use the isoperimetric inequality (\ref{eq:isoperimetricintorusstruc}) from the beginning of this proof to construct a map $f'_k : D_k \to M$ with $\area f'_k \leq C \ell( f |_{\gamma_k} )^2$ and $f'_k |_{\gamma_k} = f |_{\gamma_k}$.
Let now $f' : D_l \to M$ be the map which is equal to $f$ on $\Sigma'_l \setminus (D'_1 \cup \ldots \cup D'_n)$ and equal to $f'_k$ on each $D'_k$.
Then
\begin{multline*}
 \area f' < \area f + C \big( \ell(f |_{\gamma_1})^2 + \ldots + \ell ( f |_{\gamma_n})^2 \big) \\
 < \area f + C \big( \ell(f |_{\gamma_1}) + \ldots + \ell ( f |_{\gamma_n}) \big)^2  \leq \area f + C A^2 \td{h}_0^2 \leq \area f + \tfrac12.
\end{multline*}
This proves the desired result (after smoothing $f'$).
\end{proof}

The following Lemma is the main result of this subsection.
It will be used instead of Lemma 7.12 in \cite{Bamler-longtime-II} to find short loops in the proofs of Propositions \ref{Prop:better3rdstep} and \ref{Prop:better4thstep}.
In the proof of Proposition \ref{Prop:better4thstep} it will also be used to ensure that these loops bound compressing disks of bounded area .

\begin{Lemma} \label{Lem:shortloopingeneralcase}
For every $\alpha > 0$ and every $A, K < \infty$ there are constants $\td{L}_0 = \td{L}_0 (\alpha, A) < \infty$ and $\td\alpha_0 = \td\alpha_0 (A, K) > 0$, $\td\Gamma = \td\Gamma (K) < \infty$ such that:

Let $(M, g)$ be a Riemannian manifold and $S \subset M$, $S \approx S^1 \times D^2$ a smoothly embedded solid torus.
Let $P \subset S$ be a torus structure of width $\leq 1$ and length $L \geq \td{L}_0$ with $\partial S \subset \partial P$ (i.e. the pair $(S, \Int S \setminus \Int P)$ is diffeomorphic to $(S^1 \times D^2(1), S^1 \times D^2 (\frac12))$).

Consider a compact smooth domain $\Sigma \subset \IR^2$ and a smooth map $f : \Sigma \to S$ with $f(\partial \Sigma) \subset \partial S$ such that $f$ restricted to the outer boundary circle of $\Sigma$ is non-contractible in $\partial S$ and $f$ restricted to all other boundery circles of $\Sigma$ is contractible in $\partial S$.
Moreover, assume that $\area f < A$.

\begin{enumerate}[label=(\alph*)]
\item Then there is a closed loop $\gamma : S^1 \to P$ which is non-contractible in $P$, but contractible in $S$ which has length $\ell(\gamma) < \alpha$ and distance of at least $\frac13 L - 2$ from $\partial P$.
\item Assume that additionally $\alpha \leq \td\alpha_0$, that $|{\Rm}|, |{\nabla\Rm}| < K$ on the $1$-neighborhood around $P$, that $P$ has width $\leq \td{L}_0^{-1}$ and that $\pi_2(M) = 0$.
Then $\gamma$ can be chosen as in part (a) and such that its geodesic curvatures are bounded by $\td\Gamma$ and such that there is a map $h : D^2 \to M$ with $h |_{S^1} = \gamma$ of $\area h < \area f + 1$.
\end{enumerate}

\end{Lemma}

\begin{proof}
The proof is very similar to that of Lemma 7.12 in \cite{Bamler-longtime-II}.
The main difference comes from the existence of the map $f : \Sigma \to S$.
In part (a), this fact simplifies some arguments.
However, in part (b) the map $h$ cannot be obtained anymore by restricting $f$ to a disk.
So in this case, we have to make use of Lemma \ref{Lem:makediskfrommultiannulus}.

We first explain the general setup.
Assume that without loss of generality $\alpha < 0.1$ and set
\begin{equation} \label{eq:constantsinshortlooplem} 
\td{L}_0 (\alpha, A) = \max \Big\{ 12 \frac{A+1}{\alpha} + 3, \; \frac3{\alpha} + 12 \Big \}.
\end{equation}
We divide $P$ into three torus structures $P_1, P_2, P_3$ of width $\leq 1$ and length $> \frac13 L - 1$ in such a way that: $\partial S \subset \partial P_1$ and $P_i$ shares a boundary with $P_{i+1}$.
Then any point in $P_2$ has distance of at least $\frac13 L - 1$ from $\partial P$.
For later use, we define the solid tori
\[ S_1 = S, \qquad S_2 = \ov{S \setminus P_1}, \qquad S_3 =  \ov{S \setminus  (P_1 \cup P_2)}. \]
Moreover, let $P' \subset P_1$ be a torus structure of length $> \frac13 L - 4$ and width $\leq 1$ if we are in the setting of part (a) or of width $\leq \td{L}^{-1}_0$ if we are in the setting of part (b) such that $\partial S \subset \partial P$ and such that $\dist (P', S_2) > 2$.
Finally, choose an embedded loop $\sigma \subset S \setminus P$ which generates $\pi_1(S) \cong \IZ$.

Next, we explain the strategy of the proof.
In the setting of part (b), we obtain by (\ref{eq:constantsinshortlooplem}) that, assuming $\td\alpha_0 \leq \td{h}_0 (A, K)$, the torus structure $P'$ is $\td{h}_0 (A, K)$-precise.
Hence Lemma \ref{Lem:makediskfrommultiannulus} immediately yields a map $f' : D^2 \to M$ of $\area f' < \area f +1$ such that $f' |_{\partial D^2}$ parameterizes a non-contractible loop in $P'$.
Without loss of generality, we may assume that $f'$ is in fact an area minimizing map.
We then set $\Sigma_0 = D^2$ and $f_0 = f' : \Sigma_0 \to M$.
If we are in the setting of part (a), then we simply set $\Sigma_0 = \Sigma$ and $f_0 = f$.
So in either setting, $f_0 (\partial \Sigma_0) \subset P' \subset P_1$ and $f_0$ restricted to only the outer circle of $\Sigma_0$ is non-contractible in $P_1$.
Moreover, $\area f_0 < \area f +1$ and $f_0$ has non-zero intersection number with $\sigma$ (since $\pi_2 (M) = 0$ in part (b)).
In the following, we will construct the loop $\gamma$ from the map $f_0 : \Sigma_0 \to M$ for part (a) and (b) at the same time.
It will then only require a short argument that the additional assertions of part (b) hold.

Let $\varepsilon > 0$ be a small constant that we will determine later ($\varepsilon$ may depend on $M$ and $g$).
We can find a small homotopic perturbation $f_1 : \Sigma_0 \to M$ of $f_0 : \Sigma_0 \to M$ with $f_1 |_{\partial \Sigma_0} = f_0 |_{\partial \Sigma_0}$ which is not more than $\varepsilon$ away from $f_0$ such that the following holds: $f_1$ is transverse to $\partial P_2$ on the interior of $\Sigma_0$ and its area is still bounded: $\area f_1 < A +1$.
Note that $f_1$ still has non-zero intersection number with $\sigma$.

Consider all components $Q_1, \ldots, Q_p$ of $f_1^{-1} (S_2) \subset \Sigma_0 \subset \IR^2$.
Note that $p > 0$.
Each $Q_j$ can be extended to a disk $D_j \subset \IR^2$ by filling in its inner circles.
Let $Q'_j = D_j \cap \Sigma_0$ for each $j = 1, \ldots, p$.
Then any two disks, $D_{j_1}, D_{j_2}$ are either disjoint or one disk is contained in the other.
The same statement holds for the sets $Q'_j$.
It is not difficult to see that, after possibly changing the order of these disks, we can find a $p' \in \{ 1, \ldots, p \}$ such that the subsets $Q_1, \ldots, Q_{p'}$ are pairwise disjoint and such that $Q'_1 \cup \ldots \cup Q'_p \subset Q'_1 \cup \ldots \cup Q'_{p'}$.
It follows that $f_1$ restricted to $Q'_1 \cup \ldots \cup Q'_{p'}$ has the same non-zero intersection number with $\sigma$ as $f_1$.
So there are indices $j \in \{ 1, \ldots, p \}$ such that $f_1 |_{Q'_j}$ has non-zero intersection number with $\sigma$.
We can then choose an index $j \in \{ 1, \ldots, p \}$ with that property such that $Q'_j$ is minimal, i.e. $f_1 |_{Q'_j}$ has non-zero intersection number with $\sigma$, but $f_1 |_{Q'_{j'}}$ has intersection number zero with $\sigma$ whenever $Q'_{j'} \subsetneq Q'_j$.

Let $C_0 = \partial Q'_j$ and observe that $f_1(C_0) \subset \partial S_2$.
Consider the domain $Q'' = Q'_j \setminus f_1^{-1} (\Int S_3)$.
Its outer boundary circle is still $C_0$.
Denote by $C_1, \ldots, C_q \subset \partial Q''$ all its other boundary circles.
These circles are either inner boundary circles of $\Sigma_0$ and are mapped into $\partial S_1$ under $f_1$ (in part (a)) or they belong to $f_1^{-1} (\partial S_3)$ and hence are mapped into $\partial S_3$ under $f_1$.
So for each $l = 1, \ldots, q$ the image $f(C_l) \subset \partial S_1 \cup \partial S_3$ has distance of at least $\frac13 L - 1$ from $f(C_0) \subset \partial S_2$.
For every $l = 1, \ldots, q$ define the weight $w_l$ of $C_l$ to be the intersection number of $f_1$ restricted to the intersection of the disk, which is bounded by $C_l$ in $\IR^2$, and $\Sigma_0$, with the loop $\sigma$.
In particular, this means that $w_l = 0$ if $f_1 (C_l) \subset \partial \Sigma_0$.
It follows easily that the intersection number of $f_1 |_{Q'_j}$ with $\sigma$ is equal to $w_1 + \ldots + w_q \neq 0$.
We now apply Lemma 7.11 from \cite{Bamler-longtime-II} to $Q''$ with the pullback metric $f_1^* (g)$ and find an embedded loop $\gamma' \subset Q''$ of length
\[ \ell (f_1 |_{\gamma'}) < \frac{A + 1}{\frac13 L - 1} \leq \frac{A + 1}{\frac13 \td{L}_0 - 1} \leq \tfrac14 \alpha \]
which encloses boundary circles $C_l$ whose weights $w_l$ don't add up to zero.
Denote by $D' \subset D_j$ the disk which is bounded by $\gamma'$.
Then $f_1$ restricted to $D' \cap \Sigma_0$ has non-zero intersection number with $\sigma$.

We now argue that $\gamma'$ has a point in common with $f_1^{-1} (S_2) = Q_1 \cup \ldots \cup Q_p$.
If not, then $D' \cap f_1^{-1} (S_2) \subset Q'_j$ is the disjoint union of some of the $Q_{j'}$.
By removing some of these $Q_{j'}$ from the list and passing the the primed domains, it is then easy to see that $D' \cap f_1^{-1} (S_2)$ is contained in the disjoint union of some $Q'_{j'}$ which are contained in $Q'_j$.
By the minimality property of $Q'_j$, we conclude that the intersection number of $f_1$ with $\sigma$ is zero on all these $Q'_{j'}$.
This implies that $f_1$ restricted to $D' \cap f_1^{-1}(S_2)$ has intersection number zero with $\sigma$, contradicting the fact that $f_1$ restricted to $D' \cap \Sigma_0$ does not.
Hence $\gamma'$ has to intersect $f_1^{-1} (S_2)$ and thus $f_1(\gamma') \cap S_2 \neq \emptyset$.
Since $\gamma' \subset Q'' = Q'_j \setminus f_1^{-1} (\Int S_3)$, we have $f_1(\gamma') \cap \Int S_3 = \emptyset$.
It follows that $f_1(\gamma') \cap P_2 \neq \emptyset$.
This implies that all points of $f_1(\gamma')$ have distance of at least $\frac13 L - 1 - \alpha > \frac13 L - 2$ from $\partial P$ and that $f_1(\gamma') \subset P$.
Since the intersection number of $f_1 |_{D' \cap \Sigma_0}$ with $\sigma$ is non-zero, $f_1 |_{\gamma'}$ has to be non-contractible in $P$, but contractible in $S$.
This establishes part (a) of the Lemma with $\gamma = f_1 |_{\gamma'}$.

Assume now for the rest of the proof that we are in the setting of part (b).
Then $Q'_j = D_j$ is a disk and $D' \subset \Sigma_0$.
Moreover, $f_1$ is an $\varepsilon$-perturbation of the (stable) area minimizing map $f_0 : \Sigma_0 \approx D^2 \to M$.
By \cite{Gul}, $f_0$ is an immersion on $\Int \Sigma_0$.
So we can additionally assume that the perturbation $f_1$ is a graph over $f_0$.

Consider the following regions: Let $B(P_2, 1)$ and $B(P_2,2) \subset P \setminus P'$ be the (open) $1$ and $2$-tubular neighborhoods of $P_2$ and let $\Sigma_1$ and $\Sigma_2$ be the components of $f^{-1}_0 (B(P_2,1))$ and $f_0^{-1}(B(P_2, 2))$ which contain $\gamma'$, i.e. $\gamma' \subset \Sigma_1 \subset \Sigma_2 \subset \Int \Sigma_0$.
By the results of \cite{Sch}, we obtain a bound on the second fundamental form of $f_0 (\Sigma_1)$ which only depends on $K$: $|{II_{f_0(\Sigma_1)}}| < K'(K)$.
Moreover, this bound and the bound on the curvature on $B(P_2, 1)$ yields a curvature bound $K'' = K''(K) < \infty$ of the metric $f^*_0 (g)$ on $\Sigma_1$ which only depends on $K$.
Since $f_1$ was assumed to be a graph over $f_0 (\Sigma)$, we conclude that
\[ \ell (f_0 |_{\gamma'}) \leq 2 \ell (f_1 |_{\gamma'}) < \tfrac12 \alpha \]
if $\varepsilon$ is small enough depending on these bounds.
The loop $\gamma'$ is non-contractible in $\Sigma_1$, because otherwise $f_1 |_{\gamma'}$ would be contractible in $P$.
So we can apply Lemma 7.9 from \cite{Bamler-longtime-II} to conclude that if $\td{\alpha}_0 < \td\varepsilon_2 (K'')$, then there is an embedded loop $\gamma'' \subset \Sigma_1$ which intersects $\gamma'$, is homotopic to $\gamma'$ in $\Sigma_1$ and which has the following properties:
$\ell( f_0 |_{\gamma''} ) \leq 2 \ell( f_1 |_{\gamma'} ) < \alpha$ and the geodesic curvature on $\gamma''$ in $(\Sigma_1, f^*_0 (g))$ is bounded by $\Gamma'(K'')$.
Obviously, $\gamma''$ bounds a disk $D'' \subset \Sigma_0$ whose area under $f_0$ is bounded by $\area f +1$.
Let now $\gamma = f_0 |_{\gamma''}$.
Then the geodesic curvature on $\gamma$ in $(M, g)$ is bounded by some constant $\td\Gamma = \td\Gamma ( \Gamma'(K''(K)), K'(K)) = \td\Gamma (K) < \infty$.
This establishes assertion (b).
\end{proof}

\subsection{The main argument}
We will now go through section 8 of \cite{Bamler-longtime-II} and point out the necessary changes which are necessary to remove assumption $\TT_2$.
First observe that Lemma 8.1 (referred to as ``first step'') and Proposition 8.2 (referred to as ``second step'') do not make use of assumption $\TT_2$, or to be precise of a filling surface, and hence we will keep them unchanged.
The same is true for Proposition 8.5.
It then remains to adapt Proposition 8.3 (``third step''), Proposition 8.4 (``geometry on late and long time-intervals'') and the proof of the main theorem to our setting.
Note that most of these modifications are straight forward given the results from the previous sections, except for the proof of the main Theorem \ref{Thm:MainTheorem-III} in the case in which a component of $\MM$ is covered by a $T^2$-bundle over a circle.
In this case we in fact need to make use of a new idea.

\begin{Proposition} \label{Prop:better3rdstep}
Proposition 8.3 in \cite{Bamler-longtime-II} is still valid if assumption (iv) is replaced by the following assumption:
\begin{enumerate}[label=($iv'$)]
\item for every smoothly embedded, solid torus $S \subset \Int \MM_{\textnormal{thin}}(t_0)$, $S \approx S^1 \times D^2$, which is incompressible in $\MM(t_0)$, there is a compact smooth domain $\Sigma \subset \IR^2$ and a smooth map $f : \Sigma \to S$ with $f(\partial \Sigma) \subset \partial S$ such that $f$ restricted to the outer boundary circle of $\Sigma$ is non-contractible in $\partial S$ and $f$ restricted to all other boundery circles of $\Sigma$ is contractible in $\partial S$.
Moreover, $\area_{t_0} f < A t_0$.
\end{enumerate}
\end{Proposition}

\begin{proof}
The proof is the same as that of Proposition 8.3 in \cite{Bamler-longtime-II} except for the following changes:
We use the constant $\td{L}_0$ from Lemma \ref{Lem:shortloopingeneralcase} instead of the analogous constant from Lemma 7.12.
Then instead of invoking Lemma 7.12 in the last paragraph, we apply assumption (iv$'$) for $S \leftarrow S_i(t_0)$ to obtain $\Sigma$ and $f : \Sigma \to S_i (t_0)$ and apply Lemma \ref{Lem:shortloopingeneralcase}(a) with $M \leftarrow \MM(t_0)$, $S \leftarrow S_i (t_0)$, $P \leftarrow P'_i$ and $f \leftarrow f$ to obtain a loop $\gamma_i \subset P'_i$ which is non-contractible in $P'_i$, but contractible in $S_i(t_0)$, which has length
\[ \ell_{t_0} (\gamma_i) < \min \big\{ \tfrac1{10} \td\nu (K_2, (h'_i)^{-1}, h'_i), \td\varepsilon_1 (K_2) \big\} r_0 \]
and which has time-$t_0$ distance of at least $\frac13 L_i - 2$ from $\partial P'_i$.

The rest of the proof remains exactly the same.
\end{proof}

Similarly, we obtain:

\begin{Proposition} \label{Prop:better4thstep}
Proposition 8.4 in \cite{Bamler-longtime-II} is still valid if assumption (iv) is replaced by
\begin{enumerate}[label=($iv'$)]
\item for every time $t \in [t_0, t_\omega]$ and every smoothly embedded, solid torus $S \subset \Int \MM_{\textnormal{thin}}(t)$, $S \approx S^1 \times D^2$, which is incompressible in $\MM(t)$, there is a compact smooth domain $\Sigma \subset \IR^2$ and a smooth map $f : \Sigma \to S$ with $f(\partial \Sigma) \subset \partial S$ such that $f$ restricted to the outer boundary circle of $\Sigma$ is non-contractible in $\partial S$ and $f$ restricted to all other boundery circles of $\Sigma$ is contractible in $\partial S$.
Moreover, $\area_t f < A t$.
\end{enumerate}
and if assertion (d) is replaced by
\begin{enumerate}[label=($d'$)]
\item $\area_{t_0} D_i < (A+1) t_0$ for all $i = 1, \ldots, m$.
\end{enumerate}
\end{Proposition}

\begin{proof}
The proof is again almost the same as that of Proposition 8.4 in \cite{Bamler-longtime-II}.
We point out the changes that we have to make.
Obviously, we can use Proposition \ref{Prop:better3rdstep} instead of Proposition 8.3 of \cite{Bamler-longtime-II} everywhere in the argument.
Apart from this, the proofs of Claims 1--3 remain unchanged.

Next, look at the main part of this proof.
Given that $h^{(N)}_i < \eta^*_2$, Claim 3 is used in the first paragraph to construct a subset $P^{(k^*_i)}_{i^*_i, 1} \subset \MM (t_{k^*_i})$, which is non-singular on the time-interval $[t_{-1}, t_\omega]$, such that for all $t \in [t_0, t_\omega]$ we have $|{\Rm}| < K t^{-1}$ on $P^{(k^*_i)}_{i^*_i, 1}$ and $|{\nabla \Rm}| < K_1 t^{-3/2}$ at all points of $P^{(k^*_i)}_{i^*_i, 1}$ which have a time-$t$ distance of at least $\sqrt{t}$ from its boundary ($K < K_1$).
Moreover, $P^{(k^*_i)}_{i^*_i, 1}$ is found to be a $\varphi''(h^{(k^*_i)}_{i^*_i})$-precise torus structure at scale $\sqrt{t}$ and time $t$ and at all times $t \in [t_{-1}, t_\omega]$.
Here $\varphi'' : (0, \infty) \to (0, \infty)$ is a function with $\varphi'' (h) \to 0$ as $h \to 0$ which depends only on $L$.
Due to this fact, the bound on $h^{(k^*_i)}_{i^*}$ from Claim 3 implies that for sufficiently small $\eta^*_4 = \eta^*_4 (L, A, \alpha)$ the following holds:
Whenever $h^{(N)}_i < \eta^*_4$ and $\eta^\circ$, then this torus structure is $\min \{ \alpha, \frac1{10} \}$-precise at scale $\sqrt{t}$ and time $t$ for all $t \in [ t_0, t_\omega ]$ and $( \td{L}_0 ( \min \{ e^{-LK} \alpha, \td\alpha_0 (K_1) \}, A) + 10)^{-1}$-precise at scale $\sqrt{t_0}$ and time $t_0$.
Here $\td{L}_0$ and $\td\alpha_0$ were the constants from Lemma 7.12 in \cite{Bamler-longtime-II}.
In this paragraph we only change the interpretation that $\td{L}_0$ and $\td\alpha_0$ are the constants from Lemma \ref{Lem:shortloopingeneralcase}.
Note that now $\td\alpha_0 = \td\alpha_0 (A, K_1)$ also depends on $A$.

The next paragraph in the proof of Proposition 8.4 in \cite{Bamler-longtime-II} remains unchanged (apart from the reference to Proposition 8.3).
So now look at the last paragraph of this proof.
Observe that the subset $P_i$ is a $( \td{L}_0 ( \min \{ e^{-LK} \alpha, \linebreak[1] \td\alpha_0 (A, \linebreak[2] K_1) \}, \linebreak[2] A) \linebreak[2] + 10)^{-1}$-precise torus structure at scale $\sqrt{t_0}$ and time $t_0$.
We can choose the torus structure $P^*_i \subset P_i$ to be $( \td{L}_0 ( \min \{ e^{-LK} \alpha$, $\td\alpha_0 (K_1) \}, A))^{-1}$-precise at scale $\sqrt{t_0}$ and at time $t_0$ and such that moreover every point of $P^*_i$ has time-$t_0$ distance of at least $2 \sqrt{t_0}$ from the boundary of $P_i$ (instead of $\sqrt{t_0}$).
Then we have $|{\Rm_{t_0}}| < K_1 t_0^{-1}$ and $|{\nabla\Rm_{t_0}}| < K_1 t_0^{-3/2}$ even on the $\sqrt{t_0}$-neighborhood of $P^*_i$.
So instead of applying Lemma 7.12 from \cite{Bamler-longtime-II}, we can now invoke Lemma \ref{Lem:shortloopingeneralcase}(a), (b) with $\alpha \leftarrow \min \{ e^{-LK} \alpha, \td\alpha_0 (A, K_1) \}$, $A \leftarrow A$, $K \leftarrow K_1$, $P \leftarrow P^*_i$ and $S$ being the union of $P^*_i$ with the component of $\MM(t_0) \setminus P^*_i$ which contains $U_i(t_0)$.
The map $f : \Sigma \to S$ needed for this Lemma is obtained from assumption (iv$'$)
Observe moreover that $\pi_2$ of the component of $\MM(t_0)$ which contains $P^*_i$ vanishes due to assumption (iii) and \cite[Proposition 3.3]{Bamler-longtime-II}.
We hence obtain a loop $\gamma_i \subset P^*_i \subset P_i$ of length $\ell_{t_0} (\gamma_i) < e^{-LK} \alpha$ which is non-contractible in $P^*_i$, which spans a disk $D_i \subset \MM(t_0)$ of time-$t_0$ area $\area_{t_0} D_i < (A+1) t_0$ and whose geodesic curvatures on $\gamma_i$ are bounded by $\td\Gamma (K_1) t_0^{-1/2}$ at time $t_0$.
So assertions (d$'$) and (e) hold and the remaining assertions follow as in \cite{Bamler-longtime-II}.
\end{proof}

We can finally establish the main Theorem.

\begin{proof}[Proof of Theorem \ref{Thm:MainTheorem-III}]
In the first part of the proof, we follow closely the proof in \cite{Bamler-longtime-II}.
Let the function $\delta(t)$ the minimum of the functions given in Proposition \ref{Prop:better4thstep} (see Proposition 8.4 of \cite{Bamler-longtime-II} for more details), \cite[Proposition 4.15]{Bamler-longtime-II}, \cite[Corollary 4.3]{Bamler-longtime-II}, and \cite[Proposition 6.4]{Bamler-longtime-II}.

Consider the constant $T_1 < \infty$ and the sub-Ricci flow with surgery $\MM' \subset \MM$ from \cite[Proposition 8.5]{Bamler-longtime-II} defined on the time-interval $[0, \infty)$.
Recall that all components of all time-slices of $\MM'$ at or after time $T_1$ are irreducible and not diffeomorphic to spherical space forms and that all surgeries of $\MM'$ at or after time $T_1$ are trivial.
So at or after time $T_1$ the time-slices of $\MM$ are irreducible and not diffeomorphic to spherical space forms and the surgeries are trivial.
In particular, the topology of the manifold $\MM(t)$ is the same for all $t \geq T_1$.
By the last statement of \cite[Proposition 8.5]{Bamler-longtime-II}, it suffices to establish the desired curvature bound and the finiteness of the surgeries on $\MM'$.
Choose now a sub-Ricci flow with surgery $\MM^* \subset \MM$ defined on the time-interval $[T_1, \infty)$ whose time-slices $\MM^*(t)$ are all connected, closed components of $\MM(t)$.
Since the choice of $\MM^*$ was arbitrary, it suffices to establish the curvature bound and the finiteness of the surgeries on $\MM^*$ instead of $\MM$.

Next, we apply \cite[Proposition 4.15]{Bamler-longtime-II} to $\MM$ and consider the time $T_0 < \infty$, the function $w : [T_0, \infty) \to (0, \infty)$ as well as the decomposition $\MM (t) = \MM_{\textnormal{thick}} (t) \cup \MM_{\textnormal{thin}}(t)$ for all $t \in [T_0, \infty)$.
Set $T_2 = \max \{ T_0, T_1 \}$.

Let now $M = \MM^* (T_2)$ and $M_{\textnormal{hyp}} = \MM_{\textnormal{thick}}(T_2) \cap \MM^*(T_2)$, $M_{\textnormal{Seif}} = \MM_{\textnormal{thin}} (T_2) \cap \MM^*(T_2)$.
So $M = M_{\textnormal{hyp}} \cup M_{\textnormal{Seif}}$ and we can apply Proposition \ref{Prop:maincombinatorialresult} to obtain a simplicial complex $V$ and either a continuous map $f_0 : V \to M$ with $f_0 (\partial V) \subset \partial M_{\textnormal{Seif}} = \partial M_{\textnormal{hyp}}$ which is a smooth immersion on $\partial V$ (if $M$ is not covered by a $T^2$-bundle over a circle) or a sequence of continuous maps $f_1, f_2, \ldots : V \to M$ (if $M$ is covered by a $T^2$-bundle over a circle).
Next, we can apply Proposition \ref{Prop:areaevolutioninMM} to obtain a constant $A_0 < \infty$ and (not necessarily continuous) families of piecewise smooth maps $f_{0, t} : V \to \MM^*(t)$ or $f_{1,t}, f_{2,t}, \ldots : V \to \MM^*(t)$ for all $t \in [T_2, \infty)$ with $f_{0,T_2} |_{\partial V} = f_0 |_{\partial V}$ such that $f_{0, t} |_{\partial V}$ moves by the ambient isotopies of \cite[Proposition 4.15]{Bamler-longtime-II}, $f_{n, t}$ is homotopic to $f_n$ in space-time---restricting to said isotopies on $\partial V$ if $n = 0$---and such that for $n = 0$ or all $n \geq 1$
\begin{equation} \label{eq:areaoffntA0}
 \limsup_{t \to \infty} t^{-1} \area f_{n,t} < A_0.
\end{equation}
We now distinguish the cases in which $M$ is or is not covered by a $T^2$-bundle over a circle.

\textbf{Case 1: $M$ is not covered by a $T^2$-bundle over a circle} \quad
Choose $T_3 > T_2$ such that $\area f_{0,t} < (A_0 + 1) t$ for all $t \geq T_3$.
It now follows from Proposition \ref{Prop:maincombinatorialresult} that for every $t \geq T_3$ and every smoothly embedded solid torus $S \subset \Int\MM_{\textnormal{thin}}(t) \cap \MM^*(t)$, which is incompressible in $\MM^*(t)$, there is a compact smooth domain $\Sigma \subset \IR^2$ and a smooth map $h : \Sigma \to S$ such that $h( \partial \Sigma) \subset \partial S$ and such that $h$ restricted to only the exterior boundary circle of $\Sigma$ is non-contractible in $\partial S$ and such that
\[ \area h < C \area f_{0, t} < C (A_0 + 1) t. \]
Here, the constant $C$ only depends on the topology of the manifold $M$.

Next set
\[ A = C(A_0 + 1), \qquad L = \Big( 1 + \frac{A + 1}{ 4 \pi } \Big)^4 \]
and consider the constant $\Gamma_4 = \Gamma_4(L, A)$ from Proposition \ref{Prop:better4thstep} (see \cite[Proposition 8.4]{Bamler-longtime-II}).
Set
\[ \alpha = \frac{\pi}{\Gamma_4} \]
and choose $T_4 = T_4(L, A, \alpha)$ and $w_4 (L, A, \alpha)$ according to this Proposition.
Choose now $T^* > \max \{ 4 T_3, T_4 \}$ such that $w(t) < w_4$ for all $t \in [ \frac14 T^*, \infty)$ and consider times $t_\omega > L T^*$ and $t_0 = L^{-1} t_\omega$.
Observe that $\MM^*$ is defined on the whole time-interval $[\frac14 t_0, t_\omega]$ and that condition (iv$'$) of Proposition \ref{Prop:better4thstep} holds assuming additionally that $S \subset \MM^*(t)$.
It is then not difficult to see that we can apply Proposition \ref{Prop:better4thstep} to the sub-Ricci flow with surgery $\MM^*$ with the parameters $L, A, \alpha$ (note that this is not strictly the statement of Proposition \ref{Prop:better4thstep}, but it is easy to check that the constructions in the proofs of Propositions \ref{Prop:better3rdstep} and \ref{Prop:better4thstep} can be carried out separately on every component of $\MM$).
We then obtain sub-Ricci flows with surgery $U_1, \ldots, U_m \subset \MM^*$, outside of which we have a curvature bound, and disks $D_1, \ldots, D_m \subset \MM^*(t_0)$ with $\area_{t_0} D_i < (A+1) t_0$ whose boundary loops have length $< \alpha \sqrt{t}$ and geodesic curvature bounded by $\Gamma_4 t^{-1}$ for all $t \in [t_0, t_\omega]$.
Assume that $m \geq 1$.
Since $\alpha \Gamma_4 = \pi < 2\pi$, we obtain a contradiction by \cite[Lemma 7.2(a)]{Bamler-longtime-II}:
\[ t_\omega < \Big( 1 + \frac{A+1}{4(2\pi - \alpha \Gamma_4)} \Big)^4 t_0 = L t_0 = t_\omega. \]
So $m = 0$ an thus we have $|{\Rm_{t_\omega}}| < K_4 t_\omega^{-1}$ on $\MM(t_\omega)$.

We have shown that if $M$ is not covered by a $T^2$-bundle over a circle, then $|{\Rm_t}| < K_4 t^{-1}$ on $\MM^*(t)$ for all $t \geq L T^*$.
So in particular, $\MM^*$ does not develop any singularities past time $L T^*$.

\textbf{Case 2: $M$ is covered by a $T^2$-bundle over a circle} \quad
In this case consider the families of piecewise smooth maps $f_{1,t}, f_{2, t}, \ldots$ and observe that the constant $A_0$ in (\ref{eq:areaoffntA0}) is independent of $n$.
Note also that in the present case $\MM^* (t) \subset \MM_{\textnormal{thin}} (t)$ for all $t \geq T_2$.

Let now, $\ov{r}$, $K_2$ be the functions from \cite[Corollary 4.3]{Bamler-longtime-II} and $\mu_1$ the constant from \cite[Lemma 5.2]{Bamler-longtime-II}.
Set $\mu = \min \{ \mu_1, \frac1{10} \}$ and consider the constants $w_0 = w_0 (\mu, \ov{r}(\cdot, 1), K_2(\cdot, 1))$, $0 < s_2 = s_2 (\mu, \ov{r} (\cdot, 1), K_2(\cdot, 1)) \linebreak[1] < s_1 = s_1 (\mu, \ov{r} (\cdot, 1), K_2 (\cdot, 1)) \linebreak[1] < \frac1{10}$ from \cite[Proposition 5.1]{Bamler-longtime-II}.
Choose $T_3 > T_2$ such that $w(t) < w_0$ for all $t \geq T_3$.
Fix such a time $t$.
We can hence apply \cite[Proposition 5.1]{Bamler-longtime-II} to $\MM^*(t)$ and conclude that either $\diam_t \MM^*(t) < \mu \rho_{\sqrt{t}} (x,t)$ for all $x \in \MM^*(t)$ and $\MM^*(t)$ is diffeomorphic to an infra-nilmanifold or a manifold that carries a metric of non-negative sectional curvature, or we obtain a decomposition $\MM^* (t) = V_1 \cup V_2 \cup V'_2$ satisfying assertions (a)--(c) of this Proposition.
Note that in the first case $\MM^*(t)$ is diffeomorphic to a quotient of the $3$-torus or the nilmanifold.

We now analyze this decomposition (of the second case) further using the tools of \cite[sec 5.3]{Bamler-longtime-II} (observe that we are in case A of this subsection).
As in \cite[Definition 5.3]{Bamler-longtime-II} let $\mathcal{G} \subset \MM^*(t)$ be the union of all components of $V_2$ whose generic $S^1$-fiber is incompressible in $\MM^*(t)$ and all components of $V_1$ or $V'_2$ whose generic fibers are incompressible tori.
Then by \cite[Lemma 5.4]{Bamler-longtime-II} we have $\partial \mathcal{G} \subset V_2 \cap \mathcal{G}$.
Moreover, by \cite[Lemma 5.5]{Bamler-longtime-II} there is a disjoint union of finitely many embedded solid tori $\mathcal{S} \subset \MM^*(t)$ such that $\MM^*(t) = \mathcal{G} \cup \mathcal{S}$.
So we can make the following conclusion:
Either $\mathcal{G} = \MM^* (t)$ or there is a component $\CC \subset V_2$ such that the $S^1$-fibers on $\CC \cap V_{2, \textnormal{reg}}$ are incompressible in $\MM^*(t)$.

Let $x \in \mathcal{G}$ be an arbitrary point and recall the notation (see \cite[Definition 4.1]{Bamler-longtime-II})
\[ \rho_{r_0}(x,t) = \sup \{ r \leq r_0 \;\; : \;\; \sec_t \geq -r^{-2} \; \text{on} \; B(x,t,r) \}. \]
Consider the universal cover $\td\MM^* (t)$ of $M^*(t)$ and choose a lift $\td{x} \in \td\MM^* (t)$ of $x$.
Then by \cite[Lemma 5.2]{Bamler-longtime-II} there is a constant $w_1 = w_1(\mu) > 0$ such that
\[ \vol B(\td{x}, \rho_{\sqrt{t}} (x,t)) > w_1 \rho_{\sqrt{t}}^3(x,t). \]
(This also holds in the case in which $\MM^*(t)$ has small diameter.)
In other words, $x$ is $w_1$-good at scale $\sqrt{t}$ (compare with \cite[Definition 6.1]{Bamler-longtime-II}).
Consider now the constants $T_4 = T(w_1, 1)$, $K = K(w_1)$ and $\ov\rho = \ov\rho(w_1)$ from \cite[Proposition 6.4]{Bamler-longtime-II}.
Assuming that we have picked $t$ such that $t > T_4$, we conclude that
\[ |{\Rm}|(x,t) < K t^{-1} \qquad \text{and} \qquad \rho_{\sqrt{t}} (x,t) > \ov\rho \sqrt{t} \qquad \text{for all} \qquad x \in \mathcal{G}. \]

Consider first the case in which $\mathcal{G} \neq \MM^*(t)$ and hence there a component $\CC \subset V_2 \cap \mathcal{G}$ such that the $S^1$-fibers on $\CC \cap V_{2, \textnormal{reg}}$ are incompressible in $\MM^*(t)$.
Pick $x \in \CC \cap V_{2, \textnormal{reg}}$.
By \cite[Proposition 5.1(c3)]{Bamler-longtime-II} there are an open subset $U \subset \MM^*(t)$ with
\[ B(x, t, \tfrac12 s_2 \rho_{\sqrt{t}} (x)) \subset U \subset B(x, t, s_2 \rho_{\sqrt{t}} (x)) \]
which is diffeomorphic to $B^2 \times S^1$, vector fields $X_1, X_2$ on $U$ and a smooth map $p : U \to \IR^2$ such that:
We have $dp (X_i) = \frac{\partial}{\partial x_i}$ and $|\langle X_i, X_j \rangle - \delta_{ij} | < \frac1{10}$ for all $i,j = 1,2$.
Moreover, $p : U \approx B^2 \times S^1 \to p(U)$ corresponds to the projection to the first factor and the $S^1$-fibers coming from the second factor are isotopic to the $S^1$-fibers in $\CC \cap V_{2, \textnormal{reg}}$ and hence incompressible in $\MM^*(t)$.
It then follows easily that $p$ is $2$-Lipschitz and that $B_0 = B( p(x), \frac14 s_2 \rho_{\sqrt{t}} (x)) \subset p(U)$.

Now recall the maps $f_{1, t}, f_{2, t}, \ldots : V \to \MM^*(t)$ from the beginning of the proof.
By Proposition \ref{Prop:maincombinatorialresult}, we know that for each $n \geq 1$, the map $f_{n, t}$ intersects each $S^1$-fiber on $U$ at least $n$ times.
In other words, $f^{-1}_{n,t} (p^{-1}(y))$ contains at least $n$ elements for each $y \in B_0 \subset p(U)$.
Since $p$ is $2$-Lipschitz, we find that
\[ \area f_{n, t} \geq \frac{n}{4} \area B_0 = \frac{n \pi s_2^2 \rho^2_{\sqrt{t}} (x,t)}{16 \cdot 4} > n \cdot \frac{s_2^2 \ov\rho^2}{100} \cdot t. \]
So it follows that for
\[ n > \frac{100}{s_2^2 \ov\rho^2} (A_0 + 1) \]
we have $\area f_{n, t} > (A_0 + 1) t$.
This however contradicts (\ref{eq:areaoffntA0}) for large $t$.

So there is some constant $T_5 < \infty$ such that whenever $t > T_5$, then $\mathcal{G} = \MM^*(t)$ and hence $|{\Rm_t}| < K t^{-1}$ on $\MM^*(t)$.
As before, this implies that there are no surgeries on $\MM^*$ past time $T_5$.
\end{proof}

\end{document}